\documentclass[11pt,a4paper]{amsart}

\usepackage[top=2.5cm,bottom=2.5cm,left=2.5cm,right=2.5cm]{geometry}
\usepackage{fancyhdr}
\usepackage{nccmath}

\usepackage{amsmath}
\usepackage{amssymb}
\usepackage{amsfonts}
\usepackage{amsthm}
\usepackage{amscd}
\usepackage{mathrsfs}
\usepackage{bm}
\usepackage{latexsym}
\usepackage{colonequals}

\usepackage{graphicx}
\usepackage{float}
\usepackage{tikz-cd}
\usepackage[all]{xy}
\usepackage{pstricks}

\usepackage{indentfirst}
\usepackage{color}
\usepackage{pifont}
\usepackage{aliascnt}
\usepackage{appendix}
\usepackage{todonotes}
\usepackage{ulem}
\usepackage{setspace}

\let\oldthebibliography\thebibliography
\let\endoldthebibliography\endthebibliography
\renewenvironment{thebibliography}[1]{%
    \oldthebibliography{#1}%
    \onehalfspacing  
    \setlength{\parskip}{0.2\baselineskip}
}{%
    \endoldthebibliography
}

\usepackage[CJKbookmarks=true,colorlinks,linkcolor=red,anchorcolor=blue,citecolor=blue]{hyperref}
 \usepackage[hyperpageref]{backref}
 
\usepackage{xcolor}
\hypersetup{
	colorlinks,
	linkcolor={red},
	citecolor={blue},
	urlcolor={blue}
}
\usepackage{cleveref}

\theoremstyle{plain}

\newtheorem{thm}{Theorem}[section]
\newtheorem{proposition}[thm]{Proposition}
\newtheorem{lemma}[thm]{Lemma}
\newtheorem{corollary}[thm]{Corollary}

\newtheorem{conj}[thm]{Conjecture}

\theoremstyle{definition}

\newtheorem{definition}[thm]{Definition}
\newtheorem{remark}[thm]{Remark}
\newtheorem{example}[thm]{Example}
\newtheorem{question}[thm]{Question}

\newtheorem*{acknowledgement*}{Acknowledgements}

\newcommand{\tr}{\operatorname{tr}}
\newcommand{\com}{\mathrm{Comp}}
\newcommand{\p}{\partial}
\newcommand{\C}{\mathbb C}
\newcommand{\A}{\mathcal A}

\newcommand{\End}{\operatorname{End}}
\newcommand{\cc}{\frac{1}{2}}
\newcommand{\st}{{\star_{h_0}}}
\newcommand{\E}{{\mathcal E}}
\newcommand{\Ch}{D_{h_0}}
\newcommand{\bg}{g^{\star_{h_0}}}
\newcommand{\id}{\operatorname{id}}
\newcommand{\ob}{\operatorname{ob}}
\newcommand{\barHb}{{\mathbb H^1(\overline{X_0}, (\End \overline {E_0}^\vee,\operatorname{ad}(\theta_0^\st)))}}
\DeclareMathOperator{\KS}{KS}
\DeclareMathOperator{\Dol}{Dol}
\DeclareMathOperator{\dR}{dR}
\newcommand{\NHC}{\mathrm{NHC}}

\newcommand{\Betti}{\mathrm B}

\newcommand{\Ext}{\operatorname{Ext}}

\title[Isomonodromic deformations and non-abelian Noether--Lefschetz loci]{Isomonodromic deformations of Higgs bundles and characterization of the non-abelian Noether--Lefschetz locus}

\author[Tianzhi Hu]{Tianzhi Hu}
\address{ School of Mathematics and Statistics, Wuhan University, Luojiashan, Wuchang, Wuhan, Hubei, 430072, P.R. China}
\email{hutianzhi@whu.edu.cn}

\author{Ruiran Sun}
 \address{School of Mathematical Sciences, Xiamen University, Xiamen 361005, China}
\email{ruiransun@xmu.edu.cn}

\author{Jinbang Yang}
\address{School of Mathematical Sciences, University of Science and Technology of China, Hefei, Anhui 230026, PR China}
\email{yjb@mail.ustc.edu.cn}

\author[Kang Zuo]{Kang Zuo}
\address{ School of Mathematics and Statistics, Wuhan University, Luojiashan, Wuchang, Wuhan, Hubei, 430072, P.R. China; Institut f\"ur Mathematik, Universit\"at Mainz, Mainz, Germany, 55099}
\email{zuok@uni-mainz.de}

\makeatletter
\def\@settitle{%
  \begin{center}%
    \normalfont
    \Large
    \bfseries
    \@title\par
  \end{center}%
}
\makeatother

\begin{document}

\begin{abstract}
Let $f:X\to S$ be a smooth proper family of smooth projective varieties.  An irreducible complex local system on a fiber admits an isomonodromic deformation, hence determines a holomorphic section of the relative de Rham moduli space.  Applying the relative non-abelian Hodge correspondence produces a real-analytic section
$\sigma_{\Dol}:S\to M_{\Dol}(X/S)$ of the relative Dolbeault moduli space.

In this paper, we investigate when this real-analytic section is holomorphic.  The first approach uses the first-order infinitesimal deformation: we prove a Cauchy--Riemann type criterion showing that holomorphicity in a tangent direction of $S$ is measured by the composition of the Kodaira--Spencer map with the non-abelian Higgs field.  The second approach involves higher-order derivatives: after restricting $\sigma_{\Dol}$ to infinitesimal thickenings of the reference point in $S$, we introduce obstruction classes measuring the failure of holomorphicity and relate them to the Taylor expansion of the harmonic metric.

We apply these criteria to three problems.  First, we study the interaction between the $\C^*$-action on Higgs bundles and isomonodromic deformations.  For $\lambda\in S^1$, we prove that, on any complex analytic subvariety $U\subset S$, the rescaled family $\lambda\cdot\sigma_{\Dol}|_U$ is again isomonodromic whenever $\sigma_{\Dol}|_U$ is holomorphic. Second, for an initial polarized complex variation of Hodge structures, we consider the associated non-abelian Noether--Lefschetz locus.  We prove that this locus is precisely the maximal complex analytic subvariety of $S$ on which the real-analytic isomonodromic deformation $\sigma_{\Dol}$ becomes holomorphic.  Both the first-order and higher-order methods yield proofs of this characterization. Lastly, using the higher-order method, we prove that if the initial Higgs bundle is generically regular nilpotent and the isomonodromic deformation is holomorphic, then every member of the family is represented by a nilpotent Higgs bundle.
\end{abstract}
\maketitle

\setcounter{tocdepth}{1}
\hypersetup{linkcolor=black}
\tableofcontents
\hypersetup{linkcolor=red}

\section{Introduction} \label{sec_intro}
Let $f:X\longrightarrow S$ be a smooth proper family of smooth projective varieties.  In the local setting, we take $S$ to be a contractible complex manifold, such as a polydisc; in global statements, $S$ may instead be quasi-projective.  Following Simpson \cite{Simp94I,Simp94II}, the family $f$ has relative de Rham and Dolbeault moduli spaces
\[
\begin{tikzcd}
M_{\dR}(X/S) \arrow[rd] & & M_{\Dol}(X/S) \arrow[ld] \\
& S &
\end{tikzcd}
\]
whose fibers over $s\in S$ parametrize flat bundles and Higgs bundles on $X_s$, respectively.  The non-abelian Hodge correspondence, established in the vector-bundle setting by Donaldson \cite{Don} and Uhlenbeck--Yau \cite{UY}, and in the Higgs-bundle setting by Hitchin \cite{Hit}, Corlette \cite{Corl}, and Simpson \cite{Simp88,Simp92}, yields a
{\it real-analytic} isomorphism
\[
  \NHC:M_{\dR}(X/S)\xrightarrow{\ \sim\ } M_{\Dol}(X/S)
\]
covering the identity on $S$.  We will also use the real-analyticity of this correspondence in families; see Theorem~\ref{R-analyticity}.

\medskip
Fix a point $0\in S$, and let $\mathbb V$ be an irreducible $\C$-local system on $X_0$.  On each nearby fiber $X_s$, the flat bundle $(\mathbb V\otimes\mathcal O_{X_0},\nabla)$ admits an {\it isomonodromic deformation} $(\mathbb V\otimes\mathcal O_{X_s},\nabla_s)$.  Equivalently, these deformations define a holomorphic, possibly multivalued, section
\[
  \sigma_{\dR}:S\longrightarrow M_{\dR}(X/S).
\]
In the local setting, the contractibility of $S$ makes this section single-valued.  Composing with the relative non-abelian Hodge correspondence gives a real-analytic section, equivalently a real-analytic family of stable Higgs bundles,
\[
  \sigma_{\Dol}:=\NHC\circ \sigma_{\dR}
  :
  S\longrightarrow M_{\Dol}(X/S).
\]
We refer to both $\sigma_{\dR}$ and $\sigma_{\Dol}$ as isomonodromic deformations.  Since $\sigma_{\Dol}$ is generally only real-analytic, the following question is fundamental.

\begin{question}\label{q_hol}
How can one characterize the holomorphicity of $\sigma_{\Dol}$?
\end{question}

A second question concerns the geometric consequences of holomorphicity.
\begin{question}\label{q_whyhol}
If $\sigma_{\Dol}$ is holomorphic over $S$, or over a complex analytic subvariety of $S$, what additional structure or rigidity does $\sigma_{\Dol}$ acquire?
\end{question}

Section~\ref{sec_intro1} develops first- and higher-order criteria for the holomorphicity of $\sigma_{\Dol}$, thereby addressing Question~\ref{q_hol}.  These criteria lead to a characterization of the non-abelian Noether--Lefschetz locus conjectured by Esnault and Kerz in Question~\ref{EsnKer}.  We give two proofs: one uses the compatibility between isomonodromic deformation and the $\C^*$-action discussed in Section~\ref{sec_intro2}, while the other uses the higher-order holomorphicity criterion.

Question~\ref{q_whyhol} is intentionally broad.  In Sections~\ref{sec_intro2}--\ref{sec_intro4}, we formulate more precise versions of it---Question~\ref{Ker}, Question~\ref{EsnKer}, and Conjecture~\ref{conj_nil}---and establish several rigidity results.

\subsection{Higher-order holomorphicity of isomonodromic deformations}\label{sec_intro1}
We begin with Question~\ref{q_hol}.  The first-order infinitesimal criterion is the following Cauchy--Riemann-type statement.
\begin{proposition}[=Proposition~\ref{prop_firstob}]\label{prop_intro}
Fix a point $0\in S$, and suppose that $\sigma_{\Dol}(0)=[(E_0,\theta_0)]\in M_{\Dol}(X_0)$.  Let
\[
  \tau_0:T_0S\longrightarrow H^1(X_0,T_{X_0})
\]
be the Kodaira--Spencer map of the family $X\to S$.  The Higgs field $\theta$ induces a morphism of deformation complexes, denoted
$\theta_0:(T_{X_0},0)\to (\End E_0,\operatorname{ad}(\theta_0))$, and hence a map
\begin{align}\label{eq_nabHiggs}
  \theta_{0,*}:H^1(X_0,T_{X_0})
  \longrightarrow
  \mathbb{H}^1\bigl(X_0,(\End E_0,\operatorname{ad}(\theta_0))\bigr).
\end{align}
Then $\sigma_{\Dol}$ is holomorphic at $0$ in the direction $v\in T_0S$ if and only if $\theta_{0,*}\circ\tau_0(v)=0$.
\end{proposition}
For families of relative dimension one, related results appear in \cite[Theorem A]{HSZ} and \cite{CTW}.

Because $\sigma_{\Dol}$ is real-analytic, its holomorphicity can also be tested using finite-order jets at a fixed point.  Let $U\subset S$ be a complex analytic subvariety containing $0$, and, for $n\geq 1$, set $A_n:=\mathbb C[t]/(t^{n+1})$.  A morphism
\begin{align*}
\gamma:\operatorname{Spec} A_n \to U, \qquad \gamma(0)=0,
\end{align*}
is called an order-$n$ curve in $U$ through $0$.  We denote the set of all such curves by $J^n_0U$ and call it the {\it $n$-jet space of $U$ at $0$}.  We use the following criterion.

\medskip
\noindent\textbf{Criterion.}  Fix $0\in U$.  The restriction $\sigma_{\Dol}|_U$ is holomorphic if and only if, for every $n\geq 1$ and every $\gamma\in J_0^nU$, the pullback family $\gamma^*(\sigma_{\Dol}):=\sigma_{\Dol}\circ\gamma$ over $\operatorname{Spec} A_n$ is holomorphic.

\medskip
Thus there are two complementary approaches to the holomorphicity of $\sigma_{\Dol}$:
\begin{enumerate}
\item the first-order approach, which gives the explicit Cauchy--Riemann equation in Proposition~\ref{prop_intro}; and
\item the higher-order approach, which studies the Taylor expansion of the real-analytic section along all finite-order jets through a fixed point.
\end{enumerate}
The first approach is conceptually direct, but it must be applied pointwise and direction by direction.  The second is technically more involved, but it allows us to work in the formal neighborhood of a single point.

For $\gamma\in J^n_0S$, we construct obstruction classes to the holomorphicity of $\gamma^*(\sigma_{\Dol})$.  The main difficulty is that higher-order holomorphicity of a real-analytic map imposes many conditions, including the vanishing of derivatives in the $t\bar t$, $\bar t^2$, and higher mixed directions.  We divide the non-holomorphic derivatives into two types:
\begin{enumerate}
    \item purely anti-holomorphic derivatives, namely those with respect to $\bar t,\bar t^2,\ldots$; and
    \item mixed derivatives, namely those with respect to $t\bar t,t\bar t^2,t^2\bar t,\ldots$.
\end{enumerate}
In this paper, we isolate the purely anti-holomorphic derivatives.  This produces a sequence of partial obstruction classes whose successive vanishing measures increasing holomorphicity in the $\bar t$-directions.

\medskip
Let $X_n\to \operatorname{Spec}A_n$ be the $n$-th order deformation of $X_0$ determined by $\gamma\in J_0^nS$.  We view $X_n$ as a deformation of the complex structure on the fixed underlying differentiable manifold of $X_0$.  Suppose that $\sigma_{\Dol}(0)=[(E_0,\theta_0)]\in M_{\Dol}(X_0)$, and let $\E$ be the smooth vector bundle underlying $E_0$.  The Dolbeault operator $\bar\partial_0$ determines the holomorphic structure of $E_0$, so that $(\E,\bar\partial_0,\theta_0)=(E_0,\theta_0)$.  Truncating the isomonodromic deformation gives a real-analytic deformation on $X_n$, denoted
\[
  (\E,\bar\partial_t,\theta_t),
\]
where $\bar\partial_t$ defines a family of holomorphic structures on $\E$ over $X_n$, and $\theta_t$ is the corresponding family of Higgs fields.  Both depend on $t$ and $\bar t$; see Definition~\ref{def_real_ana_deform}.  The deformation is holomorphic precisely when it depends only on $t$; see Definition~\ref{def_holo_deform}.

For $k=1,2,\ldots,n$ and the ideal
\[
(t,\bar t^{k+1})\subset \mathbb C[t,\bar t]/(t,\bar t)^{n+1},
\]
we say that the deformation is \textit{modulo-$(t,\bar t^{k+1})$-holomorphic} if, after quotienting by $(t\bar t,\bar t^{k+1})$, it depends only on $t$; see Definition~\ref{def_modulo_hol}.  If $(\E,\bar\p_t,\theta_t)$ is modulo-$(t,\bar t^k)$-holomorphic, we construct an obstruction class
\[
  \ob_k
  \in
  \barHb
  \bigr)
\]
whose vanishing is equivalent to modulo-$(t,\bar t^{k+1})$-holomorphicity; see Proposition~\ref{prop_existence_ob}.  Consequently, the successive vanishing of
\[
  \ob_1,\ob_2,\ldots,\ob_n
\]
is equivalent to modulo-$(t)$-holomorphicity, that is, to the disappearance of all purely anti-holomorphic derivatives through order $n$.  By definition, $\ob_1$ is exactly $\theta_{0,*}\circ \tau_0$.

\medskip
We then describe these obstruction classes explicitly.  When the initial Higgs bundle is {\it stable}, the relevant obstruction data can be extracted from the deformed harmonic metric.  Let $h_0$ be the harmonic metric of the initial stable Higgs bundle $(\E,\bar\p_0,\theta_0)$, and let $h_t$, defined in \eqref{metric}, be the harmonic metric of the deformed Higgs bundle $(\E,\bar\partial_t,\theta_t)$.  After identifying all bundles with the fixed smooth bundle $\E$, we may write
\[
  h_t
  =
  h_0\left(
    \id
    +\sum_{i=1}^n t^i g_i
    +\sum_{i=1}^n \bar t^i g_i^{\st}
    +\text{mixed terms}
  \right),
\]
where $g_i\in \A^0(\End \E)$.  These coefficients are obtained by solving the Hermitian--Yang--Mills--Higgs equation for the isomonodromic deformation.  The harmonic-map formula
\[
  \theta_t
  =
  \left(
    -\frac12 h_t^{-1}dh_t
  \right)^{1,0}
\]
shows that the Taylor coefficients of the deformed Higgs field $\theta_t$ and Dolbeault operator $\bar\p_t$ in the directions $t,t^2,\ldots,t^n$ and $\bar t,\bar t^2,\ldots,\bar t^n$ are determined by the coefficients $g_i$ and $g_i^{\st}$; see Proposition~\ref{prop_Taylor_Exp_DolHig}.  The holomorphic part of the deformation is controlled by the $g_i$, whereas the obstruction classes $\ob_1,\ldots,\ob_n$ are controlled by the adjoint coefficients $g_i^{\st}$.

Using gauge theory, we compute $\ob_k$ in Proposition~\ref{prop_ob} and Proposition~\ref{prop_obstruction_class_higher_dim}.  More precisely, we give an \textbf{explicit expression} for $\ob_k$ in terms of the initial data $(\E,\bar\p_0,\theta_0,h_0)$, the deformation $X_n$, and the coefficients $\{\bg_i\}_{i=1}^k$.  Although each $\bg_i$ is determined uniquely by the initial data, this dependence is transcendental.  Thus $\ob_k$ is, in essence, determined by $X_n$ and the initial Higgs bundle.

In summary, the classes $\ob_1,\ob_2,\ldots$ explicitly encode all purely anti-holomorphic derivatives of $\sigma_{\Dol}$.  We expect the same method to extend to mixed derivatives and their obstruction classes.

\bigskip
Finally, the local deformation-theoretic framework used here---describing deformations of the complex structure on $X_0$ through the Kodaira--Spencer differential graded Lie algebra and combining this description with the harmonic theory underlying the non-abelian Hodge correspondence---has also been applied to the classical Hodge locus; see \cite{LiuShen26II}.

\subsection{Isomonodromic deformations and the $\C^*$-action}\label{sec_intro2}
We next address Question~\ref{q_whyhol} by studying the natural $\C^*$-action on the Dolbeault moduli space.

The {\it $\C^*$-action on the isomonodromic deformation $\sigma_{\Dol}$} is defined as follows.  If $\sigma_{\Dol}(s)=[(E_s,\theta_s)]$ for $s\in U$ and $\lambda\in \C^*$, then
\[
  (\lambda\cdot\sigma_{\Dol})(s):=[(E_s,\lambda\theta_s)].
\]
This again defines a real-analytic deformation of stable Higgs bundles.

Kerz asked whether this action is compatible with isomonodromic deformation.
\begin{question}[Kerz]\label{Ker}
Let $U\subset S$ be a complex analytic subvariety such that
\[
  \sigma_{\Dol}|_U:U\longrightarrow M_{\Dol}(X/S)
\]
is \textbf{holomorphic}.  For any $\lambda\in \C^*$, is $\lambda\cdot\sigma_{\Dol}|_U$ also an isomonodromic deformation?
\end{question}

The holomorphicity assumption is essential: without it, the two operations need not commute \cite[Theorem C]{HSZ}.  We use the term ``commutativity'' because Question~\ref{Ker} asks whether the following diagram represents two equivalent procedures:
\[
\begin{tikzcd}[column sep=4cm]
\sigma_{\Dol}(0) \arrow[r, "\text{isomonodromic deformation}"] \arrow[d,"\lambda\cdot"]& \sigma_{\Dol}|_U \arrow[d,"\lambda\cdot"]\\
(\lambda\cdot\sigma_{\Dol})(0) \arrow[r, "\text{isomonodromic deformation}"]  & \lambda\cdot\sigma_{\Dol}|_U.
\end{tikzcd}
\]
For technical reasons, we first treat $\lambda\in S^1:=\{z\in\C^*\mid |z|=1\}$.  For $\lambda\in S^1\setminus\{\pm1\}$, we prove the following stronger statement.
\begin{thm}\label{thm_main_1}
Let $\lambda\in S^1\setminus \{\pm1\}$, and let $U\subset S$ be a complex analytic subvariety.  Then $\sigma_{\Dol}|_U$ is holomorphic if and only if $\lambda\cdot\sigma_{\Dol}|_U$ is an isomonodromic deformation.
\end{thm}
For $\lambda=\pm1$, we prove that $\lambda\cdot\sigma_{\Dol}$ is isomonodromic without assuming that $\sigma_{\Dol}$ is holomorphic; see Proposition~\ref{prop_pm1}.  Together with Theorem~\ref{thm_main_1}, this answers Question~\ref{Ker} affirmatively for every $\lambda\in S^1$.

For $\lambda\in\mathbb R^*$, we also prove that $\lambda\cdot\sigma_{\Dol}$ is isomonodromic when the initial Higgs bundle has rank one; see Proposition~\ref{prop_rank1}.  Thus Question~\ref{Ker} has an affirmative answer in rank one.  In higher rank, the corresponding question for general $\lambda\in\mathbb R^*$ remains open.

The proof of Theorem~\ref{thm_main_1} uses first-order deformation theory.  The argument reduces the assertion to an infinitesimal statement and then applies Proposition~\ref{prop_intro}; see Proposition~\ref{prop_firstorder} for details.

\medskip
When $f:X\to S$ is the universal curve over Teichm\"uller space $\mathcal T_g$, Question~\ref{Ker} can also be approached analytically.  For example, \cite[Theorem 1.1]{Tosic} treats the case $\lambda=i$ and shows that the commutativity property is equivalent to the failure of strict plurisubharmonicity of the energy functional associated with $\sigma_{\Dol}$.  Subsequent work \cite{HSZ,CTW} identifies this condition with the holomorphicity of $\sigma_{\Dol}$.  After the first version of this paper appeared, we learned that the ``only if'' direction of Theorem~\ref{thm_main_1} for the universal family of curves can also be obtained by studying a horizontal distribution $\mathcal H$ on the relative Dolbeault moduli space; see \cite[Theorem B]{CTW}.

Hitchin has also pointed out that \cite{Hit26} studies the $S^1$-action on flat connections and uses it to investigate the holomorphic structure of the relative Dolbeault moduli space for the universal family of curves over Teichm\"uller space.

\subsection{Characterization of the non-abelian Noether--Lefschetz locus}\label{sec_intro3}
We now apply the preceding results to the non-abelian Noether--Lefschetz locus.

The classical Noether--Lefschetz theorem concerns the variation of algebraic cycles in families of smooth projective varieties.  For example, if $\mathcal X\to B$ is a family of smooth hypersurfaces in $\mathbb P^3$, the Noether--Lefschetz locus consists of the points $t\in B$ at which the Picard rank of $X_t$ jumps.  Equivalently, it is the locus on which a flat cohomology class remains of Hodge type $(1,1)$.  More generally, let $(\mathbb V,F^\bullet,\nabla,Q)$ be a polarized $\mathbb Q$-variation of Hodge structures of even weight $2k$ over $B$, and let $\gamma$ be a flat, possibly multivalued, section such that $\gamma(0)\in F^k\mathcal V:=F^k(\mathbb V\otimes\mathcal O_B)$.  The associated Hodge locus is
\[
  \mathrm{HL}_\gamma
  :=
  \{\, t\in B \mid \gamma(t)\in F^k\mathcal V \,\}.
\]

There are two complementary viewpoints on the Hodge locus:
\begin{enumerate}
\item the {\it global} viewpoint, represented by the Deligne--Griffiths fixed part theorem; and
\item the {\it local} viewpoint, in which the Zariski tangent space of the Hodge locus is controlled by the Higgs field.
\end{enumerate}
More explicitly, suppose that $\gamma$ is a flat section over a quasi-projective base $B$, that $\gamma(0)\in F^k\mathcal V$, and that the monodromy orbit of $\gamma$ is finite.  The Deligne--Griffiths fixed part theorem then implies that $\gamma$ is everywhere of type $(k,k)$.  By contrast, the local tangent-space formula states that (cf. \cite[Section 5.3.2]{VoisinII})
\[
  T^{\mathrm{Zar}}_0 \mathrm{HL}_\gamma
  =
  \ker\bigl((\theta\circ \gamma)(0):T_0B\to E\bigr),
\]
where $(E:=\operatorname{gr}_F \mathcal V,\theta:=\operatorname{gr}_F\nabla)$ is the associated Higgs bundle.  The non-abelian analogue of the global viewpoint has been studied extensively.  The non-abelian counterpart of the local tangent-space formula is the starting point of this paper.
\medskip

Let $(\mathbb V,\mathcal F^\bullet,\nabla,Q)_0$ be a polarized $\mathbb C$-variation of Hodge structures on the central fiber $X_0$ over $0\in S$.  On each nearby fiber $X_s$, the flat bundle $(\mathbb V\otimes \mathcal O_{X_0},\nabla)$ has an isomonodromic deformation $(\mathbb V\otimes \mathcal O_{X_s},\nabla_s)$.  Equivalently, these deformations define a holomorphic section
\[
  \sigma_{\dR}:S\longrightarrow M_{\dR}(X/S).
\]

Simpson introduced the associated {\it non-abelian Noether--Lefschetz locus}
\begin{align}\label{eq_NLlocus}
  \mathcal{NL}
  :=
  \bigl\{
    s\in S
    \mid
    (\mathbb V\otimes \mathcal O_{X_s},\nabla_s)
    \text{ underlies a polarized } \mathbb C\text{-VHS}
  \bigr\}.
\end{align}
The non-abelian Deligne fixed part theorem, proved in \cite{JostZuo,KP,EK}, gives a global criterion ensuring that the isomonodromic deformation of a $\C$-PVHS again underlies a $\C$-PVHS: if the base is quasi-projective and the monodromy orbit of the isomonodromic deformation is finite, then the deformed local system underlies a $\C$-PVHS.

\medskip
We now turn to the local geometry of $\mathcal{NL}$ and therefore assume that $S$ is contractible.  Recall 
\[
  \sigma_{\Dol}:=\NHC\circ \sigma_{\dR}
  :
  S\longrightarrow M_{\Dol}(X/S).
\]
If $(E_0,\theta_0)$ is the graded Higgs bundle associated with the initial variation of Hodge structures on $X_0$, then $\sigma_{\Dol}(0)=[(E_0,\theta_0)]$.  By non-abelian Hodge theory, a flat bundle underlies a polarized $\mathbb C$-VHS precisely when the corresponding Higgs bundle is graded, or equivalently, fixed by the natural $\mathbb C^*$-action on the Dolbeault moduli space \cite{Simp92,Simp97,Simp10}.  Therefore, if
\[
  \mathcal{GR}
  :=
  \bigl\{
    s\in S
    \mid
    \sigma_{\Dol}(s)
    \text{ is represented by a graded Higgs bundle}
  \bigr\},
\]
then
\[
  \mathcal{NL}=\mathcal{GR}.
\]
Thus $\mathcal{NL}$ can be studied entirely in terms of Higgs bundles.  Simpson proved that $\mathcal{NL}$ is a complex analytic subvariety of $S$ and that the restriction
\[
  \sigma_{\Dol}|_{\mathcal{NL}}
  :
  \mathcal{NL}\longrightarrow M_{\Dol}(X/S)
\]
is holomorphic \cite[Theorem 12.1]{Simp97}.  Esnault and Kerz asked whether this property characterizes the non-abelian Noether--Lefschetz locus.

\begin{question}[Esnault--Kerz]\label{EsnKer}
Let $(\mathbb V,\mathcal F^\bullet,\nabla,Q)_0$ be a polarized $\mathbb C$-VHS on $X_0$, and let $U\subset S$ be a closed complex analytic subvariety passing through $0$.  Suppose that
\[
  \sigma_{\Dol}|_U:U\longrightarrow M_{\Dol}(X/S)
\]
is holomorphic.  Must one have
\[
  U\subset \mathcal{NL}?
\]
Equivalently, is $\mathcal{NL}$ the maximal complex analytic subvariety of $S$ on which the real-analytic section $\sigma_{\Dol}$ becomes holomorphic?
\end{question}

We answer this question affirmatively, obtaining a local characterization of the non-abelian Noether--Lefschetz locus in terms of holomorphicity.

\begin{thm}\label{thm_main}
Let $(\mathbb V,\mathcal F^\bullet,\nabla,Q)_0$ be a polarized $\mathbb C$-VHS on $X_0$.  Let $U\subset S$ be a closed complex analytic subvariety passing through $0$ such that
\[
  \sigma_{\Dol}|_U:U\longrightarrow M_{\Dol}(X/S)
\]
is holomorphic.  Then
\[
  U\subset \mathcal{NL}.
\]
\end{thm}

We give two proofs of Theorem~\ref{thm_main}.  The first combines the first-order deformation theory of Section~\ref{sec_intro1} with the commutativity result of Section~\ref{sec_intro2}.

\begin{proof}[Proof of Theorem~\ref{thm_main} via Theorem~\ref{thm_main_1}]
For every $\lambda\in S^1$, the initial Higgs bundle $(E_0,\theta_0)$ is graded, and hence $\sigma_{\Dol}(0)=\lambda\cdot\sigma_{\Dol}(0)$.  By Theorem~\ref{thm_main_1}, the family $\lambda\cdot\sigma_{\Dol}|_U$ is an isomonodromic deformation of $(E_0,\theta_0)$ along $U$.  The uniqueness of isomonodromic deformation therefore gives $\lambda\cdot\sigma_{\Dol}|_U=\sigma_{\Dol}|_U$.  It follows from \cite[Lemma 4.1]{Simp92} that $(E_u,\theta_u)$ is graded for every $u\in U$.
\end{proof}

The second proof is higher-order.  The non-abelian analogue of the classical Zariski tangent-space formula for Hodge loci is
\[
  T^{\mathrm{Zar}}_0\mathcal{NL}
  =
  \ker
  \Bigl(
    \theta_{0,*}\circ \tau_0:
    T_0S
    \longrightarrow
    \mathbb{H}^1\bigl(X_0,(\End E_0,\operatorname{ad}(\theta_0))\bigr)
  \Bigr),
\]
as proved in Theorem~\ref{thm_Zar_NL}.  By Proposition~\ref{prop_intro}, first-order holomorphicity of $\sigma_{\Dol}$ along a tangent vector is equivalent to first-order liftability of the graded structure.  The higher-order theory developed in Section~\ref{sec_intro1} gives the following extension.

\medskip

\begin{thm}[Truncated version]\label{thm_main_tru}
Let $(\E,\bar\partial_0,\theta_0)$ be a graded stable Higgs bundle on $X_0$.  Let
\[
  (\E,\bar\partial_t,\theta_t)
\]
be the isomonodromic deformation of $(\E,\bar\partial_0,\theta_0)$ over an $n$-th order deformation $X_n$ of $X_0$.  If this real-analytic deformation is holomorphic, then it is graded.
\end{thm}
\begin{proof}[Proof of Theorem~\ref{thm_main} via Theorem~\ref{thm_main_tru}]
We may reduce to the case in which $U$ is smooth near $0$.  Indeed, if $U$ is singular, let $\pi:\widehat U\to U$ be a resolution of singularities and pull back both $X$ and $\sigma_{\Dol}$ to $\widehat U$.  This gives a real-analytic family of Higgs bundles over $X_{\widehat U}:=X\times_U\widehat U$, denoted
\[
  \widehat\sigma_{\Dol}:\widehat U\longrightarrow M_{\Dol}(X_{\widehat U}/\widehat U).
\]
For any $\widehat u\in\widehat U$ with $\pi(\widehat u)=0$, the Higgs bundle $\widehat\sigma_{\Dol}(\widehat u)$ is stable and graded.  It is therefore enough to prove the theorem after this pullback, so we assume that $U$ is smooth.

The holomorphicity assumption and Theorem~\ref{thm_main_tru} imply that every Higgs bundle parametrized by $\sigma_{\Dol}|_U$ is graded.  Hence $U\subset \mathcal{GR}=\mathcal{NL}$, proving the Esnault--Kerz characterization.
\end{proof}

A related question asks whether, for the universal curve $\mathcal C/\mathcal T_g$ over Teichm\"uller space, the section $\sigma_{\Dol}$ can be holomorphic on all of $\mathcal T_g$.  For rank-two and rank-three non-unitary Higgs bundles, the answer is negative \cite{biswas,HSZII}.  On the other hand, \cite{biswas} constructs high-rank non-unitary examples for which $\sigma_{\Dol}$ is a holomorphic family of graded Higgs bundles over all of $\mathcal T_g$.

\subsection{Isomonodromic deformations of nilpotent Higgs bundles}\label{sec_intro4}
We now assume that $\sigma_{\Dol}(0)=[(E_0,\theta_0)]$ is a stable nilpotent Higgs bundle.  We expect an analogue of the Esnault--Kerz question in this setting and formulate the following conjecture.
\begin{conj}\label{conj_nil}
Assume that the initial Higgs bundle $\sigma_{\Dol}(0)=[(E_0,\theta_0)]$ is nilpotent.  Let $U\subset S$ be a closed complex analytic subvariety passing through $0$ such that $\sigma_{\Dol}|_U$ is holomorphic.  Then $\sigma_{\Dol}(u)$ is nilpotent for every $u\in U$.
\end{conj}

Let $r=\operatorname{rank}(E_0)$.  Using the higher-order deformation and obstruction theories developed in Section~\ref{sec_intro1}, we prove the conjecture for {\it generically regular nilpotent} Higgs bundles, namely, for nilpotent Higgs bundles satisfying
\[
\mathrm{Sym}^{r-1}\theta_0:\mathrm{Sym}^{r-1}T_{X_0}\longrightarrow\End(E_0)
\quad \text{is nonzero}.
\]
From this point onward, we suppress $\mathrm{Sym}$ from the notation and write $\theta^i$ for the corresponding symmetric power.  The main result of this subsection is the following.
\begin{thm}[Holomorphicity and nilpotency]
\label{thm_sec4_arbitrary_nilpotency}
Let $(\E,\bar\p,\theta)$ be a stable generically regular nilpotent Higgs bundle of rank $r$, and let $(\E,\bar\p_t,\theta_t)$ be its order-$n$ isomonodromic deformation on $X_n$.  If this deformation is holomorphic on $X_n$, then
\begin{equation}  \label{eq_sec4_theta_nilpotent}
  \theta_t^r=0
  \qquad\pmod{t^{n+1}}.
\end{equation}
\end{thm}
The theorem proves Conjecture~\ref{conj_nil} for every generically regular nilpotent Higgs bundle.  We expect that the generic regularity assumption can be removed.  A comparable result appears difficult to obtain from first-order deformation theory alone, illustrating the strength of the higher-order approach.

\subsection{Relation to earlier work and organization of the paper}
This paper substantially extends and refines the ideas developed in our earlier preprints arXiv:2606.05794 and arXiv:2606.18768, and we regard it as the definitive version of this line of work.  The remainder of the paper is organized as follows.  Section~\ref{sec_isom_defm_Higgs} develops the infinitesimal deformation theory of isomonodromic Higgs bundles and derives Taylor expansions for the deformed Higgs bundles and harmonic metrics.  Section~\ref{sec_ob_hol} introduces obstruction classes that measure the failure of holomorphicity and gives explicit formulas for the higher-order obstructions.  Section~\ref{sec_prop_first_hol} applies the first-order criterion to prove Theorem~\ref{thm_main_1} on the $S^1$-action and to compute the Zariski tangent space of the non-abelian Noether--Lefschetz locus.  Section~\ref{sec_high_hol_proof} carries out the higher-order arguments, including the proofs of Theorem~\ref{thm_main_tru} and Theorem~\ref{thm_sec4_arbitrary_nilpotency}.  Appendix~\ref{sec_complexification} records our conventions for real-analytic deformations and their first-order Kodaira--Spencer classes.  Appendix~\ref{sec_joint_real_ana} proves the real-analyticity of the relative non-abelian Hodge correspondence used throughout the paper.

\begin{acknowledgement*}
We are deeply grateful to H\'el\`ene Esnault and Moritz Kerz for communicating their conjecture (Question~\ref{EsnKer}) to us and for suggesting that an obstruction-theoretic approach might be the right tool for characterizing the non-abelian Noether--Lefschetz locus. After we completed the first preprint version of this work (arXiv:2606.05794), Kerz raised Question~\ref{Ker}, which led to a simplified proof of the main result Theorem~\ref{thm_main}.  Their insights and encouragement have been instrumental in the development of this work.  We thank Nigel Hitchin for explaining the relationship between our work and \cite{Hit26}.  We also thank Brian Collier for explaining the related works \cite{CTW} and \cite{Tosic}.  We are grateful to Sebastian Heller, J\"urgen Jost, Carlos Simpson, Lin Weng, and Shing-Tung Yau for their interest in this work, as well as for their valuable comments and suggestions.  Finally, we thank Runze Zhang for discussions on higher-order deformation theory and for providing the reference \cite{Ono}.
\end{acknowledgement*}

\section*{Notations}

\begin{itemize}
\item Unless otherwise stated, all indices $i$ (including $i_1, i_2, \dots$) appearing in this paper are positive integers.

\item Let $(\E,\bar\p_0,\theta_0)$ be a polystable Higgs bundle on $X_0$ carrying a harmonic metric $h_0$ and $\Ch=\Ch^{1,0}+\bar\p_0$ be the associated Chern connection. For any $g\in\A^0(\End\E)$, we let $\bg:=h_0^{-1}\bar g^T h_0$.

\item  Let $A_n = \mathbb{C}[t]/(t^{n+1})$ and $B_n = \mathbb{C}[t, \bar{t}]/(t^{n+1}, t^n\bar{t}, \dots, \bar{t}^{n+1})$ be the Artin rings of truncated holomorphic and real-analytic functions on the complex plane $\mathbb{C}$ at the origin, with maximal ideals $(t)$ and $(t, \bar{t})$, respectively.

\item Let $S:=\sum_{i=0}^n t^i s_i\in \mathcal A^l(\End\E)\otimes A_n$ be an $A_n$-valued $l$-section of $\End\E$. The notation $[t^i]S$ denotes the coefficient $s_i$, i.e. the coefficient of $t^i$ in $S$. Similarly for sections in $\mathcal A^l(\End\E)\otimes \C[\bar t]/(\bar t^{n+1}).$

\end{itemize}

\newpage

\section{Isomonodromic deformations of Higgs bundles}
\label{sec_isom_defm_Higgs}

In this section, we study isomonodromic deformations of stable Higgs bundles with trivial Chern classes on a smooth projective variety. We work with infinitesimal deformations over Artin rings and derive explicit Taylor expansions for the deformed Dolbeault operator, the Higgs field, and the associated harmonic metric. The main result is Proposition~\ref{prop_Taylor_Exp_DolHig}, which shows that the coefficients appearing in these expansions are determined solely by the initial Higgs bundle and the infinitesimal deformation of the underlying smooth projective variety. These results will be used throughout the paper.

\medskip 
\subsection{Deformation theory of a smooth projective variety} \label{sec_def_projvar}

Let $X_0$ be a smooth projective variety. Recall that $A_n = \mathbb{C}[t]/(t^{n+1})$ is the Artin ring of truncated holomorphic function germs on the complex plane at the origin with the maximal ideal $\mathfrak m:=(t)$. Let $$D_{X_0}(A_n)=\{\text{deformations } X_n\to\operatorname{Spec}A_n \text{ of }X_0 \}/\sim$$ 
be the set of isomorphism classes of deformations $X_n\to\operatorname{Spec}A_n$ of $X_0$ to $\operatorname{Spec}A_n$. 
Using the theory of \textbf{differential graded Lie algebras} (see \cite[Remark 14.8-14.9]{Gro}), we view a deformation $X_n\to \operatorname{Spec} A_n$ as a family of complex structures on \textbf{a fixed differentiable manifold} $X_0$ (forgetting the initial complex structure of $X_0$). This gives 
\begin{align*}D_{X_0}(A_n)\cong\frac{\{\eta\in \A^{0,1}(T_{X_0})\otimes\mathfrak m\mid \bar\p_{T_{X_0}}\eta+\cc[\eta,\eta]=0\}}{\text{gauge equivalence}}.\end{align*}
More precisely, for any $\eta=\sum_{i=1}^n \eta_i t^i\in \A^{0,1}(T_{X_0})\otimes \mathfrak m$ satisfying the above integrability condition, the corresponding deformation is the ringed space $X_n = (X_0^{\rm Top}, \mathcal{O}_{X_n})$ over $\operatorname{Spec} A_n$, where $\mathcal{O}_{X_n} \subset \mathcal{C}^\infty(X_0) \otimes A_n$ is the subsheaf of functions annihilated by the operator $\bar{\partial}_{X_0} + \eta \circ \partial_{X_0}$. The $A_n$-algebra structure on $\mathcal{O}_{X_n}$ induces the structural morphism $X_n \to \operatorname{Spec} A_n$.

Let $B_n = \mathbb{C}[t, \bar{t}]/(t,\bar{t})^{n+1}$ be the Artin ring of $n$-th order truncated real-analytic function germs at the origin of the complex plane.
Let $\mathcal C^\infty(X_n):=\mathcal C^\infty(X_0)\otimes B_n$ be the \textbf{sheaf of smooth functions} on $X_n$. We define the \textbf{holomorphic cotangent bundle} $\Omega^{1}(X_n/A_n)$ as the locally free sheaf of $\mathcal O_{X_n}$-modules locally generated by $df$, where $f$ is a local holomorphic function of $X_n$. Let $\Omega^{1,0}(X_n/A_n):=\Omega^{1}(X_n/A_n)\otimes_{\mathcal O_{X_n}}\mathcal C^\infty(X_n) $ be the \textbf{smooth $(1,0)$ cotangent bundle}, which is a subsheaf of the \textbf{smooth cotangent bundle} $\mathbb CT_{X_n}^*:=\A^1(X_0)\otimes B_n$. The \textbf{anti-holomorphic cotangent bundle} $\Omega^{0,1}(X_n/A_n)$ is defined as the complex conjugate of $\Omega^{1,0}(X_n/A_n)\subset \mathbb CT^*_{X_n}$.

\subsection{Deformations of Higgs bundles and gauge theory} \label{sec_def_Higg_and_gauge}

Let \((E_0,\theta_0)=(\E,\bar\p_0,\theta_0)\) be a stable Higgs bundle on \(X_0\); throughout, such bundles are assumed to have trivial Chern classes.  Here \(\E\) denotes the underlying smooth vector bundle obtained by forgetting the holomorphic structure of \(E_0\).  To study the deformation theory of the triple \((X_0,E_0,\theta_0)\), we fix the smooth model \((X_0,\E)\) and equip it with a family of complex structures \((\eta,\bar\p_t)\) together with a family of Higgs fields \(\theta_t\), as in the following definition.

\begin{definition} \label{def_real_ana_deform} 
For any order $n$ deformation of $X_0$, denoted by $X_n\in D_{X_0}(A_n)$, we define a \textbf{real-analytic deformation} of the initial stable Higgs bundle $(\E,\bar\p_0,\theta_0)$ on $X_0$ to $X_n$ as a triple $(\E,\bar\p_t,\theta_t)$ with
\begin{align*}
\bar\p_t:&\E\to\E\otimes_{\mathcal C^\infty(X_0)}\Omega^{0,1}(X_n/A_n);\\
\theta_t:&\E\to\E\otimes_{\mathcal C^\infty(X_0)}\Omega^{1,0}(X_n/A_n),
\end{align*}
satisfying the following conditions:
\begin{enumerate}
\item[(i)] $\bar\p_t$ is $\mathbb C$-linear and satisfies the Leibniz rule as a $(0,1)$ connection; $\theta_t$ is $\mathcal C^\infty(X_n)$-linear;
\item[(ii)] Modulo $t,\bar t$, the deformation triple reduces to the initial Higgs bundle, that is, $\bar\p_t\equiv\bar\p_0$ and $\theta_t\equiv\theta_0$;
\item[(iii)] $(\bar\p_t,\theta_t)$ satisfies the integrability conditions
\begin{align}\label{eq_compatibility}
\bar\p_t^2=0;\quad\theta_t\wedge\theta_t=0;\quad \bar\p_t\theta_t=0.
\end{align}
\end{enumerate}
\end{definition}

\begin{example}\label{eg_isomo}
Let $(\E,\bar\p_s,\theta_s)$ be the isomonodromic deformation of the initial Higgs bundle $(\E,\bar\p_0,\theta_0)$ from $X_0$ to the family $X/S$. By \cite[Theorem 4.23]{CTW} when the fibers of $X/S$ are compact Riemann surfaces, and by Theorem~\ref{R-analyticity} in the general case, $(\E,\bar\p_s,\theta_s)$ is a real-analytic deformation of Higgs bundles. We consider any order-$n$ germ of $S$ at $0$, that is, a morphism $\gamma:\operatorname{Spec}A_n\to S$ mapping $0\in\operatorname{Spec}A_n$ to $0\in S$.
The pull-back of $X/S$ via $\gamma:\operatorname{Spec}A_n\to S$ gives an $X_n\in D_{X_0}(A_n)$. The pull-back of $(\E,\bar\p_s,\theta_s)$ via $\gamma$ gives a real-analytic deformation of Higgs bundles.
\end{example}

Let $h_0$ be the harmonic metric of the initial Higgs bundle $(\E,\bar\p_0,\theta_0)$ on $X_0$ and $\Ch=\Ch^{1,0}+\bar\p_0$ be the Chern connection of $(\E,\bar\p_0,h_0)$. The deformed operators in Definition~\ref{def_real_ana_deform} can be expanded explicitly in the deformation parameters $t$ and $\bar t$, as follows.

\begin{lemma}\label{lem_expanding} 
Let $\eta=\sum\limits_{i=1}^n \eta_i t^i\in \A^{0,1}(T_{X_0})\otimes \mathfrak m$ represent $X_n$. Then there exist $\alpha_i,\beta_i,\varphi_i,\psi_i\in \A^1(\End \E)$ such that 
\begin{align*}\bar\p_t&=\bar\p_0-\bar\eta\circ\bar\p_0+\eta\circ\Ch^{1,0}+B+\Psi\pmod{t\bar t};\\
\theta_t&=\theta_0+A+\Phi\pmod{t\bar t},
\end{align*}
where $\eta\circ(\cdot)$ and $\bar\eta\circ(\cdot)$ are contractions here and 
\begin{align*}A=\sum_{i=1}^nt^i\cdot\alpha_i,\ B=\sum_{i=1}^nt^i\cdot\beta_i,\ \Phi=\sum_{i=1}^n\bar t^i\cdot\varphi_i,\ \text{ and } \ \Psi=\sum_{i=1}^n\bar t^i\cdot\psi_i.  
\end{align*}
\end{lemma}
\begin{proof}
We consider the $(0,1)$-part of $\Ch$ with respect to the complex structure of $X_n$, denoted by $\pi''_\eta\Ch$.  One verifies directly, as an identity of operators, that  \begin{align}\label{eq_piCh}
\pi''_\eta\Ch\equiv \bar\p_0-\bar\eta\circ\bar\p_0+\eta\circ\Ch^{1,0}\pmod{t\bar t}.
\end{align}
Then $\bar\p_t-\pi''_\eta\Ch$ is a section of $\End\E\otimes_{\mathcal C^\infty(X_0)} \Omega^{0,1}(X_n/A_n)$. 
This proves the claim.
\end{proof}

\begin{definition} \label{def_holo_deform}
Let $(\E,\bar\p_0,\theta_0)$ be a Higgs bundle on $X_0$. 
A \textbf{(holomorphic) deformation} of $(\E,\bar\p_0,\theta_0)$ is a relative Higgs bundle $(\E,\bar\p_t,\theta_t)$ over $X_n/A_n$ with central fiber $(\E,\bar\p_0,\theta_0)$. 
There is a natural forgetful functor from the category of holomorphic deformations to the category of real-analytic deformations, induced by extending the coefficient sheaf from the holomorphic to the smooth setting. 
A real-analytic deformation is said to be \textbf{holomorphic} if it lies in the essential image of this functor.
\end{definition}

We next give a criterion for the holomorphicity of a real-analytic deformation.  We first introduce some notation. 
Let $\pi'_{\eta}:\C T^*_{X_n}\to \Omega^{1,0}(X_n/A_n)$ and $\pi''_{\eta}:\C T^*_{X_n}\to \Omega^{0,1}(X_n/A_n)$ be two natural projections. Via the trivial extension of smooth forms, we have $\Omega^{1,0}(X_0)\hookrightarrow \C T^*_{X_n}$ and $\Omega^{0,1}(X_0)\hookrightarrow \C T^*_{X_n}$. We define 
\begin{align*}P_\eta':=&\pi'_{\eta}|_{\Omega^{1,0}(X_0)}:\Omega^{1,0}(X_0)\to \Omega^{1,0}(X_n/A_n);\\ P_\eta'':=&\pi''_{\eta}|_{\Omega^{0,1}(X_0)}:\Omega^{0,1}(X_0)\to \Omega^{0,1}(X_n/A_n)
\end{align*}
For any $\alpha\in \Omega^{1,0}(X_0)$, by \cite[P75]{Gro}, we have
\begin{align}\label{eq_Peta}
P_\eta'(\alpha)=\alpha-\eta(\alpha) \quad \text{and}\quad P_\eta''(\bar\alpha)=\bar\alpha-\bar\eta(\bar\alpha),
\end{align}
where $\eta(\alpha)\in\Omega^{0,1}(X_0)\otimes (t)$ denotes contraction. 
We have the following canonical isomorphisms of $\mathcal C^\infty(X_0)$-modules:
\begin{align*}  
B_n \otimes P'_\eta \Omega^{1,0}(X_0)\cong \Omega^{1,0}(X_n/A_n) 
 \quad \text{and}\quad
B_n \otimes P''_\eta  \Omega^{0,1}(X_0) \cong  \Omega^{0,1}(X_n/A_n).
\end{align*}
Henceforth, we identify these sheaves via these canonical isomorphisms.  For a real-analytic deformation $(\E,\bar\p_t,\theta_t)$ in Definition~\ref{def_real_ana_deform}, we view $\bar\p_t-\pi_\eta''\Ch$ as an element in $B_n \otimes P''_\eta  \Omega^{0,1}(X_0) \otimes_{\mathcal C^\infty(X_0)} \End\E$ and view $\theta_t$ as an element in $B_n \otimes P'_\eta \Omega^{1,0}(X_0) \otimes_{\mathcal C^\infty(X_0)} \End\E$. 

\begin{proposition} \label{prop_criterion_holo_deform}
 The real-analytic deformation $(\E,\bar\p_t,\theta_t)$
 is holomorphic if and only if there exists a gauge transformation $\mathscr U$ such that 
\begin{equation}\label{eq_hol}
\begin{aligned}
 \mathscr U^{-1} \circ \bar\p_t \circ \mathscr U - \pi_\eta''\Ch & \in  A_n \otimes P''_\eta  \Omega^{0,1}(X_0) \otimes_{\mathcal C^\infty(X_0)} \End(\E);
 \\  
\mathscr U^{-1}\circ\theta_t \circ \mathscr U & \in A_n \otimes P'_\eta \Omega^{1,0}(X_0) \otimes_{\mathcal C^\infty(X_0)} \End(\E).
\end{aligned}
\end{equation}

\end{proposition}

\begin{proof}
By \cite{Ono}, the condition \eqref{eq_hol} is equivalent to the holomorphicity of $(\E,\mathscr U^{-1} \circ \bar\p_t \circ \mathscr U,\mathscr U^{-1} \circ \theta_t \circ \mathscr U )$. Therefore $(\E,\bar\p_t,\theta_t)$ is also a holomorphic deformation because it differs from $(\E,\mathscr U^{-1} \circ \bar\p_t \circ \mathscr U,\mathscr U^{-1} \circ \theta_t \circ \mathscr U )$ by a gauge equivalence. 
\end{proof}

\begin{definition} \label{def_modulo_hol}
For the ideal $(t,\bar t^{k+1})$ of $B_n$, a real-analytic deformation $(\E,\bar\p_t,\theta_t)$
 is called \textbf{modulo-$(t,\bar t^{k+1})$-holomorphic} if there exists a gauge transformation $\mathscr U$ such that 
\begin{equation}\label{eq_hol_b}
\begin{aligned}
 \mathscr U^{-1} \circ \bar\p_t \circ \mathscr U - \pi_\eta''\Ch & \in  (A_n+(t,\bar t^{k+1})) \otimes P''_\eta  \Omega^{0,1}(X_0) \otimes_{\mathcal C^\infty(X_0)} \End(\E);
 \\  
\mathscr U^{-1}\circ\theta_t \circ \mathscr U & \in (A_n+(t,\bar t^{k+1})) \otimes P'_\eta \Omega^{1,0}(X_0) \otimes_{\mathcal C^\infty(X_0)} \End(\E).
\end{aligned}
\end{equation}

\end{definition}

\begin{remark}
\begin{enumerate}
    \item[(i)] Let $(t,\bar t^{k+1}) \subset (t,\bar t^{k})$ be two ideals of $B_n$. If $(\E,\bar\p_t,\theta_t)$  is modulo-$(t,\bar t^{k+1})$-holomorphic, then it is modulo-$(t,\bar t^{k})$-holomorphic. Moreover, all real-analytic deformations are modulo-$(t,\bar t)$-holomorphic.
    \item[(ii)]  By replacing $\mathscr U$ with $\mathscr U \mathscr U_0^{-1}$ in \eqref{eq_hol}, we may assume that the gauge transformation $\mathscr U\in\A^0(\End \E)\otimes B_n$ satisfies $\mathscr U \equiv \mathrm{id} \pmod{(t,\bar t)}$, where $\mathscr U_0$ is the constant term of $\mathscr U$. 
\end{enumerate}

\end{remark}

\bigskip

Let $(\E,\bar\p_t,\theta_t)$  be a modulo-$(t)$-holomorphic deformation and let 
\begin{equation} \label{eq_Taylor_U}
\mathscr{U} = \operatorname{id} + \sum_{i=1}^n \bar{t}^i u_i + \sum_{i=0}^{n-1} \sum_{j=1}^{n-i} t^j \bar{t}^i u_{j\bar i}\in\A^0(\End\E)\otimes B_n
\end{equation} 
be the Taylor expansion of a gauge transformation such that 
$\mathscr U^{-1}\circ\bar\p_t\circ \mathscr U$ and $\mathscr U^{-1}\circ\theta_t\circ\mathscr U$ satisfy \eqref{eq_hol_b} for $k=n$. Then we have the following lemma describing this form of holomorphicity as a system of equations for the coefficients $u_i$.

\begin{lemma}[Equations for gauge transformation] \label{lem_gauge} 
Let $(\E,\bar\p_t,\theta_t)$  be a real-analytic deformation. Then it is modulo-$(t,\bar t^{k+1})$-holomorphic if and only if all coefficients $u_i$ in the gauge transformation $\mathscr{U}$ given in \eqref{eq_Taylor_U} satisfy the following system of equations:
\begin{align}\label{hol}
\begin{cases}
\Phi U+[\theta, U]\equiv 0\pmod{(t,\bar t^{k+1})};\\
\Psi U+\bar\p_0 U-\bar\eta\circ\bar\p_0 U\equiv 0\pmod{(t,\bar t^{k+1})},
\end{cases}
\end{align}
where $U:=\sum_{i=1}^n \bar t^i u_i$ and $\Phi,\Psi$ are defined in Lemma~\ref{lem_expanding}.
 \end{lemma}

\begin{proof}
By Lemma~\ref{lem_expanding}, we have \begin{align*}\bar\p_t=\bar\p_0-\bar\eta\circ\bar\p_0+\Psi\pmod{(t,\bar t^{k+1})};\ 
\theta_t=\theta_0+\Phi\pmod{(t,\bar t^{k+1})}.
\end{align*}
By the definition of $\mathscr U$ in \eqref{eq_Taylor_U}, we have \eqref{eq_hol_b}. Modulo $(t,\bar t^{k+1})$ in these two identities and using \eqref{eq_piCh}, we have 
\begin{align*}\bar{\partial}_t \circ U&= U\circ (\bar\p_0-\bar\eta\circ\bar\p_0)\pmod{(t,\bar t^{k+1})};\\
\theta_t\circ U&= U\circ\theta_0\pmod{(t,\bar t^{k+1})}.
\end{align*}
After expanding the above expressions and comparing coefficients of $\bar t,\cdots,\bar t^k$, we have \eqref{hol}.
\end{proof}

\subsection{Harmonic metrics and isomonodromic deformation} \label{sec_HarmMet_IsomDef}

Let $(\E,\bar\p_s,\theta_s)$ be the isomonodromic deformation of the initial Higgs bundle $(\E,\bar\p_0,\theta_0)$ from $X_0$ to the family $X/S$. In this subsection, we use $(\E,\bar\p_t,\theta_t)$ to denote a truncated real-analytic deformation on $X_n$ defined in Example~\ref{eg_isomo}.

Let $h_t\in A^0(\bar E^\vee\otimes E^\vee)\otimes B_n$ denote the truncation of the harmonic metric of $(\E,\bar\p_s,\theta_s)$. There exist $g_i,g_{i\bar j}\in \A^0(\End E)$ such that 
\begin{align}\label{metric}
h_t=h_0 \cdot \underbrace{\left(\operatorname{id}+\sum_{i=1}^n t^ig_i+\sum_{i=1}^n \bar t^ig_i^\st+\sum_{i=1}^{n-1}\sum_{j=1}^{n-i}t^i\bar t^j g_{i\bar j}\right)}_{g(t,\bar{t})}.
\end{align}
where $\bg_i$ is the $h_0$-adjoint of $g_i$ defined in Notations. The coefficients $g_i,g_{i\bar j}$ characterize the \textbf{deformation} of the harmonic metric. Let $g(t,\bar t)$ denote the expression in parentheses in the above equation, which is invertible as $g(t,\bar t)-\mathrm{id}$ is nilpotent. The expression $g(t,\bar t)$ must be of this form to ensure that $h_t$ is Hermitian.

\medskip

We view $h_t$ as a harmonic map defined on the universal cover $\widetilde X_n$  of $X_n$. Let $\Theta_t:=-\cc h_t^{-1}dh_t$. 
Let $(\mathcal V,D)$ on $X_0$ be the associated flat bundle given by the non-abelian Hodge correspondence, where $D$ is the smooth flat connection.
 By an argument similar to the proof of \cite[Lemma 9.13]{Gui}, substituting the metric \eqref{metric} into $\Theta_t=-\cc h_t^{-1}dh_t$ yields
\begin{equation}\label{PSI}\begin{aligned} 
\Theta_t =& -\cc g(t,\bar t)^{-1} \cdot (h_0^{-1}dh_0) \cdot g(t,\bar t)-\cc g(t,\bar t)^{-1} \cdot D^{\End} \Big(g(t,\bar t)\Big)\\
=&g(t,\bar t)^{-1}(\theta_0+\theta_0^\st)g(t,\bar t)-\cc g(t,\bar t)^{-1}D^{\End} \Big(g(t,\bar t)\Big).
\end{aligned}\end{equation}

By \cite[Lemma 9.13]{Gui}, the $(1,0)$-part of $\Theta_t$ with respect to $X_n$ is $\theta_t$ and the $(0,1)$-part of $\Theta_t$ with respect to $X_n$ is $\theta_t^{\star_{h_t}}$.
Comparing the $(1,0)$ and $(0,1)$ parts of both sides of \eqref{PSI}, we obtain the main result of this section: when the deformation is isomonodromic, we can express the $\varphi_i, \psi_i, \alpha_i, \beta_i$ from Lemma~\ref{lem_expanding} solely in terms of the stable Higgs bundle on $X_0$ (equivalently, the initial data $(\E, \bar\p_0, \theta_0, h_0, \Ch^{1,0})$) and the order $n$ deformation $X_n$ of $X_0$ (equivalently, the series $\eta = \sum\limits_{i=1}^n t^i \eta_i$). Let 
$G:=\id+\sum\limits_{i=1}^n t^i g_i;\ G^\st=\id+\sum\limits_{i=1}^n \bar t^i\bg_i.$
\begin{proposition} \label{prop_Taylor_Exp_DolHig}
Suppose the deformation is isomonodromic. 
For any $1\leq i\leq n$, $g_i$ is uniquely determined by $(\E,\bar\p_0,\theta_0,h_0,\Ch^{1,0})$ and $\eta$.
The explicit formulas for $\Phi, \Psi, A, B$ (given in Lemma~\ref{lem_expanding}) are:
\begin{align}&\Phi=\bar\eta(\theta_0^\st)+\cc (G^\st)^{-1}([\theta_0,G^\st]-\Ch^{1,0}G^\st)+\cc\bar\eta\left((G^\st)^{-1}([\theta_0^\st,G^\st]-\bar\p_0 G^\st)\right);\label{eq_Phi}\\
 &\Psi=\cc(G^\st)^{-1}(\bar\p_0 G^\st-[\theta_0^\st,G^\st])-\cc\bar\eta\left((G^\st)^{-1}([\theta_0^\st,G^\st]-\bar\p_0 G^\st)\right);\label{eq_Psi}\\
  &A=\frac{1}{2}G^{-1}\bigl([\theta_0,G]-\Ch^{1,0}G\bigr)-\frac{1}{2}\eta\!\left(
      G^{-1}\bigl([\theta_0,G]-\Ch^{1,0}G\bigr)
     \right);
  \label{eq_sec3_A_def}\\
  &B=\frac{1}{2}G^{-1}\bigl(\bar\p_0 G-[\theta_0^\st,G]\bigr)
    +\frac{1}{2}\eta\!\left(
       G^{-1}\bigl(\Ch^{1,0}G-[\theta_0,G]\bigr)
     \right).
  \label{eq_sec3_B_def}
 \end{align}
In these formulas, the operators $\Ch^{1,0}$ and $\bar\p_0$ are the connections of $\End\E$ induced by those of $\E$. We also note that $\eta_i:\A^{1,0}(\End\E)\to \A^{0,1}(\End\E)$ and $\bar\eta_i:\A^{0,1}(\End\E)\to \A^{1,0}(\End\E)$ in the above are contractions.
\end{proposition}

\begin{proof}
Throughout the proof, we work modulo the ideal $(t\bar t)$. We first prove the four identities. Let 
\begin{align*}
\pi_\eta':&B_n\otimes \A^1(\End\E)\to  B_n\otimes\Omega^{1,0}(X_n/A_n)\otimes \End\E; \\\pi_\eta'':&B_n\otimes \A^1(\End\E)\to  B_n\otimes\Omega^{0,1}(X_n/A_n)\otimes \End\E
\end{align*}
be the two projections according to type. Then we have $\pi'_\eta\Theta_t=\theta_t$ and $\pi''_\eta\Theta_t=\theta_t^{\star_{h_t}}$. Using \eqref{PSI}, we have
{\fontsize{10}{12}\selectfont \begin{equation} \label{eq_theta_t}
\begin{aligned}
\theta_t=&g(t,\bar t)^{-1}(\theta_0-\eta(\theta_0)+\bar\eta(\theta_0^\st))g(t,\bar t)-\cc g(t,\bar t)^{-1}(D'-\eta\circ D'+\bar\eta\circ D'')^{\End}\Big(g(t,\bar t)\Big);\\
\theta_t^{\star_{h_t}}=&g(t,\bar t)^{-1}(\theta_0^\st+\eta(\theta_0)-\bar\eta(\theta_0^\st))g(t,\bar t)-\cc g(t,\bar t)^{-1}(D''+\eta\circ D'-\bar\eta\circ D'')^{\End}\Big(g(t,\bar t)\Big),
\end{aligned}
\end{equation}}
where $D':\A^0(X_0,\End\E)\to \A^{1,0}(X_0,\End\E)$ and $D'':\A^0(X_0,\End\E)\to \A^{0,1}(X_0,\End\E)$ with $D=D'+D''$. 
Note that we have $D'=\Ch^{1,0}+\theta_0$ and $D''=\bar\p+\theta_0^\st$. By comparing the coefficient of $t^i,\ i=0,1,2,\cdots,n$ in the above identity of $\theta_t$, we have
\begin{align*}\theta_0+A =G^{-1}(\theta_0-\eta(\theta_0))G-\cc G^{-1}(\Ch^{1,0}+\theta_0-\eta\circ\Ch^{1,0}-\eta(\theta_0))^{\End}G.
\end{align*}
Substituting \begin{align*}G^{-1}\theta_0 G=G^{-1}[\theta_0, G]+\theta_0;\quad \theta_0^{\End}G=[\theta_0,G]\end{align*} into the above identity, we obtain the expression \eqref{eq_sec3_A_def}. The same argument gives \eqref{eq_Phi}. We now derive the remaining two identities. A direct check gives $\pi_\eta'' D=D''+\eta\circ D'-\bar\eta\circ D''$. Thus \begin{align*}\bar\p_t=\pi_\eta'' D-\theta_t^{\star_{h_t}}.\end{align*}
Substituting \eqref{eq_theta_t} into the above identity and comparing the coefficient of $t^i,\ i=1,2,\cdots,n$ in the above identity, we have
\begin{align*}B =\eta(\theta_0)-G^{-1}(\theta_0^\st+\eta(\theta_0))G+\cc G^{-1}(\bar\p+\theta_0^\st+\eta\circ\Ch^{1,0}+\eta(\theta_0))^{\End}G.
\end{align*}
This proves the expression \eqref{eq_sec3_B_def}. The same argument gives \eqref{eq_Psi}.

Substituting these expressions into the integrability conditions \eqref{eq_compatibility}, we obtain several PDEs for $g_i,\ i=1,2,\cdots,n$. By the non-abelian Hodge correspondence and the uniqueness of the harmonic metric, the above PDEs on $g_i$ are solvable and uniquely determine $g_i$ for $i=1,2,\cdots,n$.
\end{proof}

\newpage

\section{Obstruction classes of holomorphicity} \label{sec_ob_hol}

In this section, we investigate the holomorphicity of the real-analytic isomonodromic deformations discussed in the previous section. We introduce a sequence of obstruction classes that measure the failure of an isomonodromic deformation to be holomorphic. These obstruction classes lie in a certain cohomology group, which we explicitly describe using the Dolbeault resolution. The main result of this section is Proposition~\ref{prop_obstruction_class_higher_dim}, where we derive explicit formulas for all higher-order obstruction classes. 

\medskip

\subsection{Obstruction classes of modulo-$(t)$-holomorphicity} \label{sec_ob_mod_t_hol}
\subsubsection{Obstruction class group}
Recall the computation of the hypercohomology group $$\barHb$$ via the Dolbeault resolution. First, we have $\Omega_{\overline{X_0}}^{1,0}=\Omega_{X_0}^{0,1}$ and $\Omega_{\overline{X_0}}^{0,1}=\Omega_{X_0}^{1,0}$, where $\overline{X_0}$ is the complex manifold conjugate to $X_0$. Hence $(\E,\Ch^{1,0},\theta_0^\st)$ is a Higgs bundle on $\overline{X_0}$. For any two $\omega_1\in\A^k(\End\E),\ \omega_2\in \A^l(\End\E)$, we define the following Lie brackets.
\begin{align*}[\omega_1,\omega_2]:=\omega_1\circ\omega_2-(-1)^{kl}\omega_2\circ\omega_1.
\end{align*}

We have the following Dolbeault resolution:
\[
\begin{tikzcd}
\vdots & \vdots & \vdots \\
C^{2,0}:=\mathcal A^{2,0}(\End \E) 
\arrow[r, "{\operatorname{ad}( \theta_0^{\star_{h_0}})}"] \arrow[u]
& C^{2,1}:=\mathcal A^{2,0}(\End \E \otimes \Omega_{X_0}^{0,1}) 
\arrow[r, "{\operatorname{ad}( \theta_0^{\star_{h_0}})}"] \arrow[u]
& C^{2,2}:=\mathcal A^{2,0}(\End \E \otimes \Omega_{X_0}^{0,2}) 
\arrow[r] \arrow[u] & \cdots \\
C^{1,0}:=\mathcal A^{1,0}(\End \E) 
\arrow[r, "{\operatorname{ad}( \theta_0^{\star_{h_0}})}"] 
\arrow[u, "\Ch^{1,0}"] 
& C^{1,1}:=\mathcal A^{1,0}(\End \E \otimes \Omega_{X_0}^{0,1}) 
\arrow[r, "{\operatorname{ad}( \theta_0^{\star_{h_0}})}"] 
\arrow[u, "\Ch^{1,0}"] 
& C^{1,2}:=\mathcal A^{1,0}(\End \E \otimes \Omega_{X_0}^{0,2}) 
\arrow[r] \arrow[u, "\Ch^{1,0}"] & \cdots \\
C^{0,0}:=\mathcal A^{0,0}(\End \E) 
\arrow[r, "{\operatorname{ad}( \theta_0^{\star_{h_0}})}"] 
\arrow[u, "\Ch^{1,0}"] 
& C^{0,1}:=\mathcal A^{0,0}(\End \E \otimes \Omega_{X_0}^{0,1}) 
\arrow[r, "{\operatorname{ad}( \theta_0^{\star_{h_0}})}"] 
\arrow[u, "\Ch^{1,0}"] 
& C^{0,2}:=\mathcal A^{0,0}(\End \E \otimes \Omega_{X_0}^{0,2}) 
\arrow[r] \arrow[u, "\Ch^{1,0}"] & \cdots
\end{tikzcd}
\]
which gives the following truncated complex 
\[C^{0,0} \overset{d_c^{0}}{\longrightarrow} C^{1,0} \oplus C^{0,1} \overset{d_c^{1}}{\longrightarrow} C^{2,0}\oplus C^{1,1}\oplus C^{0,2}{\longrightarrow} \cdots\] 
where 
\begin{align*}
 d_c^{0}(g) &= (\Ch^{1,0} g, [\theta_0^{\star_{h_0}},g])\in C^{1,0}\oplus C^{0,1} \quad \text{for } g \in C^{0,0},\\
 d_c^{1}(\varphi, \psi) &= (\Ch^{1,0}\varphi,\Ch^{1,0} \psi +[\theta_0^{\star_{h_0}},\varphi],[\theta_0^\st,\psi])\in C^{2,0}\oplus C^{1,1}\oplus C^{0,2} \quad \text{for } (\varphi, \psi) \in C^{1,0} \oplus C^{0,1}.
\end{align*}
Hence
$$\barHb=\frac{\mathrm{Ker\ }d_c^{1}}{\mathrm{Im\ }d_c^{0}}.$$
In the next subsection we shall see that this $\barHb$ is the desired obstruction group.

\subsubsection{Existence of obstruction classes}
Given a real-analytic deformation of $(\E,\bar\p_0,\theta_0)$ on $X_0$ to $X/S$, let $(\E,\bar\p_t,\theta_t)$ be the truncation to $X_n$ via a holomorphic $n$-jet $c:\operatorname{Spec}A_n\to S,\ c(0)=0$.
\begin{proposition}\label{prop_existence_ob}Suppose $(\E,\bar\p_t,\theta_t)$ is modulo-$(t,\bar t^{k})$-holomorphic with $k\leq n$ (defined in Definition~\ref{def_modulo_hol}). Then there exists a class $\ob_k \in \barHb$ such that $\ob_k$ vanishes if and only if $(\E,\bar\p_t,\theta_t)$  is modulo-$(t,\bar t^{k+1})$-holomorphic. In particular, $\ob_1,\ob_2,\cdots,\ob_n$ all vanish successively if and only if $(\E,\bar\p_t,\theta_t)$ is modulo-$(t)$-holomorphic.
 
\end{proposition}

\begin{remark}
 Modulo-$(t)$-holomorphicity is strictly weaker than holomorphicity, so further obstructions are needed to obtain full holomorphicity.
\end{remark}
\begin{proof}

By definition, 
\((\mathcal E,\bar\partial_t,\theta_t)\) is given by a map
\[
\sigma:\operatorname{Spec}\mathbb C[t,\bar t]/(t,\bar t)^{n+1}
\longrightarrow M_{\mathrm{Dol}}(X/S)=:\mathcal{M}
\]
with $
\sigma(0)=(\mathcal E,\bar\p_0,\theta_0)$ on \(X_0\). 

By the complexification argument in Section~\ref{sec_joint_realanalytic}, $\sigma$ uniquely extends to a morphism $g$ such that the following diagram commutes
\[
\begin{tikzcd}
\operatorname{Spec}\mathbb C[t,\bar t]/(t,\bar t)^{n+1} \arrow[r,"\sigma"] \arrow[d,hook] & \mathcal M \arrow[d,hook]\\
\operatorname{Spec}(\mathbb C[t]/(t^{n+1})\otimes \C[\bar t]/(\bar t^{n+1}))\arrow[r, "g"] & M_{\Dol}((X\times\overline X)/(S\times \overline S)) 
\end{tikzcd}
\]

Define the jet spaces of maps sending  \(0\) to $o:=
(X_0,\mathcal E,\bar\partial_0,\theta_0)\in \mathcal M$:
\[
\operatorname{Hom}\left(
\operatorname{Spec}\mathbb C[t,\bar t]/(t,\bar t)^{k+1},
\mathcal M
\right)
:= \mathbb{C}J_k\mathcal M,
\]
and
\[
\operatorname{Hom}\left(
\operatorname{Spec}\mathbb C[\bar t]/(\bar t^{k+1}),
M_{\Dol}(\overline X/\overline S)
\right)
:= \overline{ J_k\mathcal M}.
\]

Then $g$ induces a $k$-jet map $\operatorname{Spec}(\C[\bar t]/(\bar t^{k+1}))\to M_{\Dol}(\overline X/\overline S) $ by base changing via $$0\times\operatorname{Spec}(\C[\bar t]/(\bar t^{k+1}))\hookrightarrow  0\times \operatorname{Spec}(\C[\bar t]/(\bar t^{n+1}))\hookrightarrow\operatorname{Spec}(\mathbb C[t]/(t^{n+1})\otimes \C[\bar t]/(\bar t^{n+1})).$$
We denote this $k$-jet map by $p_k([\sigma])\in \overline{ J_k\mathcal M}$.

For the jet \([\sigma]\), we have the following commutative diagram:
\[
\begin{tikzcd}
 & & \mathbb{C}J_k\mathcal M
 \arrow[r, "\Pi_{k-1}^k"]
 \arrow[d, "p_k"']
 & \mathbb{C}J_{k-1}\mathcal M
 \arrow[d, "p_{k-1}"]
 & \\
0 \arrow[r]
& T^{1,0}_{o}M_{\Dol}(\overline X/\overline S) \arrow[r]
& \overline{ J_k\mathcal M}
 \arrow[r, "\pi_{k-1}^k"]
& \overline{ J_{k-1}\mathcal M}
 \arrow[r]
& 0 
\end{tikzcd}
\]

By modulo-$(t,\bar t^{k})$-holomorphicity, \([\sigma]\) maps to zero under
\[
p_{k-1}\circ \Pi_{k-1}^k.
\]
Thus \(p_k([\sigma])\) is given by an element in \(T^{1,0}_{o}M_{\Dol}(\overline X/\overline S)\). Let $\mathbb CT_o\mathcal M=T^{1,0}_{o}\mathcal M\oplus T^{0,1}_{o}\mathcal M$ be the complexification of the real Zariski tangent space of a real-analytic variety $\mathcal M$ at $o$. One can prove directly there is natural isomorphism $T^{1,0}_{o}M_{\Dol}(\overline X/\overline S)\cong T^{0,1}_o M_{\Dol}(X/S)$ by using the argument in Proposition~\ref{prop:KS-real-analytic-manifold}. Thus we obtain \(p_k([\sigma])\) is given by an element in $T^{0,1}_o \mathcal M$.
\bigskip

Now we prove that \(p_k([\sigma])\) is given by an element in
\(T^{0,1}_oM_{\mathrm{Dol}}(X_0)\). 
By definition, the composition
\[
\operatorname{Spec}\mathbb C[t,\bar t]/(t,\bar t)^{n+1}
\xrightarrow{\ \sigma\ }
\mathcal M
\xrightarrow{\ \pi_{\mathrm{Dol}}\ }
S
\]
is independent of $\bar t$ and is equal to the holomorphic \(n\)-germ $c$ of \(S\) through $0$. Therefore 
\[
p_k\bigl([\pi_{\mathrm{Dol}}\circ \sigma]\bigr)=0
\qquad
\text{in } \quad\overline{ J_k S}
\]
by holomorphicity, and we obtain
\[
p_k\bigl([\pi_{\mathrm{Dol}}\circ \sigma]\bigr)
=
(\pi_{\mathrm{Dol},*})\bigl(p_k([\sigma])\bigr)
=0.
\]
Therefore $
p_k([\sigma])\in \ker(\pi_{\mathrm{Dol},*}),$ 
and hence $
p_k([\sigma])\in T^{0,1}_oM_{\mathrm{Dol}}(X_0),$ which gives \[\ob_k:=p_k([\sigma])\in T^{0,1}_oM_{\mathrm{Dol}}(X_0)\cong \barHb.\qedhere \]
\end{proof}

\subsection{First-order holomorphicity of the isomonodromic deformation of a Higgs bundle}
In this subsection, we give the first-order obstruction class $\ob_1$ defined in Proposition~\ref{prop_existence_ob} of holomorphicity of $\sigma_{\Dol}$.

We first give the following Lemma about ``harmonicity'', which will be repeatedly used later.
\begin{lemma} \label{vanlem}
Let $(\E, \bar\p_0, \theta_0)$ be a stable Higgs bundle. In particular, it admits a harmonic metric $h_0$ and thus it is a harmonic bundle. 
 Suppose  $g\in \A^0(\End \E)$ satisfies either of the following two equations:
\begin{align}\label{exact-harm}
\Ch^{1,0}\bar\p_0 g+[\theta_0^\st,[\theta_0,g]]=0,
\end{align}
or 
\begin{align}\label{exact_harm2}
\bar\p_0\Ch^{1,0} g+[\theta_0,[\theta_0^\st,g]]=0,
\end{align}
Then $g=c\cdot\id$ for some $c\in\C$.
\end{lemma}

\begin{remark}
There is a Hodge theoretic interpretation of ``harmonicity'' in Lemma~\ref{vanlem}:
\begin{enumerate}\item  The class $[([\theta_0,g],\bar\p_0 g)]\in \mathbb H^1(X_0,(\End E_0,\operatorname{ad}(\theta_0)))$ is exact. 
\item  The equation \eqref{exact-harm} implies $[([\theta_0,g],\bar\p_0 g)]\in \mathbb H^1(X_0,(\End E_0,\operatorname{ad}(\theta_0)))$ is harmonic.
\end{enumerate}
By Hodge decomposition theory, an exact and harmonic class must be zero.
\end{remark}

\begin{proof}[Proof of Lemma~\ref{vanlem}]
Let $\omega$ be a K\"ahler form of $X_0$. Assuming \eqref{exact-harm}, we have
\begin{align*}\sqrt{-1}\int_{X_0}\operatorname{tr}(g^\st\Ch^{1,0}\bar\p_0 g)\omega^{\dim X_0-1}+\sqrt{-1}\int_{X_0}\operatorname{tr}(g^\st[\theta_0^\st,[\theta_0,g]])\omega^{\dim X_0-1}=0.
\end{align*}
By the K\"ahler identity (see \cite{Simp92} and \cite[Remark 9.2]{Gui}), we have
 $$\sqrt{-1}\int_{X_0}\operatorname{tr}(g^\st\Ch^{1,0}\bar\p_0 g)\omega^{\dim X_0-1}=-\sqrt{-1}\int_{X_0}\operatorname{tr}((\bar\p_0 g)^\st\bar\p_0 g)\omega^{\dim X_0-1}\leq 0.$$ One may verify directly that $$\sqrt{-1}\int_{X_0}\operatorname{tr}(g^\st[\theta_0^\st,[\theta_0,g]])\omega^{\dim X_0-1}=-\sqrt{-1}\int_{X_0}\operatorname{tr}([\theta_0,g]\wedge [\theta_0,g]^\st)\omega^{\dim X_0-1}\leq 0.$$
Thus
we have $\bar\p_0 g=0$ and $[\theta_0,g]=0$. This means that $g\in\mathbb H^0(X_0,(\End E_0,\operatorname{ad}(\theta_0)))$ and by the stability $g=c\cdot\id$ for some $c\in\C$. If $g$ satisfies \eqref{exact_harm2}, one proves the conclusion similarly.
\end{proof}

Let $X_1:=X_n\times_{\operatorname{Spec}A_n} \operatorname{Spec}A_1$; we may pull back $(\E,\bar\p_t,\theta_t)$ to $X_1$. In this case, modulo-$(t,\bar t^2)$-holomorphicity on $X_n$ is equivalent to holomorphicity on $X_1$ after base change. 
\begin{proposition}\label{prop_firstob}
Let $(\E,\bar\p_t,\theta_t)$ be the real-analytic deformation in Definition~\ref{def_real_ana_deform} of $(\E,\bar\p_0,\theta_0)$ on $X_0$ along $X_n$. If it is isomonodromic, then the obstruction class $\ob_1$ of modulo-$(t,\bar t^2)$-holomorphicity is 
\begin{align*}\ob_1=[(\bar\eta_1(\theta_0^\st),0)]\in\barHb.
\end{align*}
\end{proposition}
\begin{proof}
We first prove $[(\bar\eta(\theta_0^\st),0)]$ is a well defined class in $\barHb$. Note that the Higgs field
$\theta_0$ induces a morphism of deformation complexes, denoted $\theta_0:(T_{X_0},0)\to (\End\E_0,\operatorname{ad}(\theta_0))$, and hence a map
\begin{align*}
  \theta_{0,*}:H^1(X_0,T_{X_0})
  \longrightarrow
  \mathbb{H}^1\bigl(X_0,(\End E_0,\operatorname{ad}(\theta_0))\bigr).
\end{align*}
A direct check from the definition shows that $[(0,\eta(\theta_0))]\in \mathbb H^1(X_0,(\End E_0,\operatorname{ad}(\theta_0)))$ is exactly $\theta_{0,*}([\eta])$. A similar argument works for $(\overline{X_0},\overline {E_0}^\vee,\theta_0^\st)$ and $\bar\eta$ in the conjugate case. This proves the assertion.

\textbf{Step1:} modulo-$(t,\bar t^2)$-holomorphicity implies the vanishing of the class $[(\bar\eta_1(\theta_0^\st),0)]$. By equations \eqref{hol} for gauge transformation, the equations \eqref{eq_Phi} and \eqref{eq_Psi} of $\varphi_1$ and $\psi_1$, we have
\begin{equation} \label{g1}
\begin{aligned}
\cc\bar\p_0 \bg_1-\cc[\theta_0^{\star_{h_0}},\bg_1]+\bar\p_0 u_1&=0;\\
\bar \eta_1(\theta_0^{\star_{h_0}})+\cc[\theta_0,\bg_1]-\cc D_{h_0}^{1,0} \bg_1+[\theta_0,u_1]&=0.
\end{aligned}
\end{equation}
Hence
\begin{align*}
\Ch^{1,0}(\bar\p_0 (\cc\bg_1+u_1)-\cc[\theta_0^{\star_{h_0}},\bg_1])=0;\quad[\theta_0^\st,[\theta_0,\cc\bg_1+u_1]+\bar \eta_1(\theta_0^{\star_{h_0}})-\cc D_{h_0}^{1,0} \bg_1]=0.
\end{align*}
Summing them we have 
\begin{align*}
\Ch^{1,0}\bar\p_0 (\cc\bg_1+u_1)+
[\theta_0^\st,[\theta_0,\cc\bg +u_1]]=0,
\end{align*}
because $[\theta_0^\st,\bar\eta_1(\theta_0^\st)]=0$ and $\Ch^{1,0}([\theta_0^\st,\bg_1])=-[\theta_0^\st,\Ch^{1,0}\bg_1]$. By the ``harmonicity'' in Lemma~\ref{vanlem}, we have $u_1=-\cc\bg_1$ (up to adding a term $c\cdot\id$ with $c\in\C$, which we may ignore since it does not affect \eqref{g1}). Thus \eqref{g1} reduces to
\begin{align}\label{eq_ob1exact}0=\cc[\theta_0^{\star_{h_0}},\bg_1];\qquad\ 
\bar \eta_1(\theta_0^{\star_{h_0}})=\cc D_{h_0}^{1,0} \bg_1,\end{align} 
i.e. $\ob_1\in\barHb$ vanishes. 

\textbf{Step2:} the vanishing of the class $[(\bar\eta_1(\theta_0^\st),0)]$ implies modulo-$(t,\bar t^2)$-holomorphicity. We aim to prove the solvability of \eqref{g1} on $u_1$. Since $[(\bar\eta_1(\theta_0^\st),0)]$ vanishes, there exists $f_1\in \A^0(\End\E)$ such that 
\begin{align}\label{eq_ob1vanish}0=[\theta_0^{\star_{h_0}},f_1];\qquad\ 
\bar \eta_1(\theta_0^{\star_{h_0}})= D_{h_0}^{1,0} f_1.
\end{align}
By a similar argument as in \cite[Proposition 4.4 (4.4)]{HSZ}, the condition  $\bar\p_t\theta_t\equiv 0\pmod{(t,\bar t)^2}$ implies that 
\begin{align*}\bar\p_0\Ch^{1,0}\bg_1=2\bar\p_0(\bar\eta(\theta_0^\st))-[\theta_0,[\theta_0^\st,\bg_1]].
\end{align*}
Together, this with the assumption \eqref{eq_ob1vanish} gives 
\begin{align*}\bar\p_0\Ch^{1,0}(\bg_1-2f_1)+[\theta_0,[\theta_0^\st,\bg_1-2f_1]]=0,
\end{align*}
which implies $\bg_1=2f_1+c\cdot\id$ for some $c\in\mathbb C$ by the ``harmonicity'' in Lemma~\ref{vanlem}.
Substituting this and \eqref{eq_ob1vanish} into \eqref{g1}, we have
\begin{align*}
\bar\p_0 (\cc\bg_1+u_1)=0;\qquad
[\theta_0,\cc\bg_1+u_1]=0.
\end{align*}
Thus $u_1=-\cc\bg_1$ is the solution of \eqref{g1}.
\end{proof}

\subsection{Obstruction classes of modulo-\texorpdfstring{$(t)$}{(t)}-holomorphicity of the isomonodromic deformation of a stable Higgs bundle} \label{sec_ob_mod_t_hol_isom_grad}

In this section we always let $(\E,\bar\p_t,\theta_t)$ be the isomonodromic deformation of $(\E,\bar\p_0,\theta_0)$ on $X_0$ to $X_n$, which is a real-analytic deformation as in Definition~\ref{def_real_ana_deform}. We compute the obstruction classes of modulo-$(t)$-holomorphicity defined in Proposition~\ref{prop_existence_ob}.

We define the set of ordered compositions to be
\[
\com(k):=\{I=(i_1,\ldots,i_N)\mid N\ge1,\ i_a\in\mathbb Z_{>0},\ i_1+\cdots+i_N=k\},\quad k\in\mathbb Z_{>0}.
\]
For \(I=(i_1,\ldots,i_N)\) set \(l(I)=N\).  

\noindent \textbf{Auxiliary sequences.} Let $p_{(k)}=1$ and for $N\geq 2$, let
\begin{align*}
p_{(i_1,\cdots,i_N)}:=\prod_{a=2}^N
\frac{i_a}{i_1+\cdots+i_a},\qquad \text{denoted by } p_I.\end{align*}
And we set for any $I=(i_1,\cdots,i_N)$,
\begin{align*}b_I:=\frac{1}{2^{N}}\sum_{j=0}^N
p_{(i_1,i_2,\cdots,i_j)}p_{(i_N,i_{N-1},\cdots,i_{j+1})},\end{align*}
where the summands corresponding to $j=0$ and $j=N$ are understood to be $p_{(i_N,i_{N-1},\cdots,i_1)}$ and $p_{(i_1,\cdots,i_N)}$.

Recall the deformed data $\bg_k$ defined in \eqref{metric}. We define \(x_k\in \A^0(\End E)\) by the following triangular change formula
\begin{align}\label{eq_beta_change}
\bg_k=\sum_{I\in\com(k)}b_Ix_I,
\qquad k\ge1,
\end{align}
where $x_I:=x_{i_1}\cdots x_{i_N}\in \A^0(\End E)$. Note that each $x_i$ can be expressed as a polynomial of $\bg_1,\cdots,\bg_i$. For example,
\[
\bg_1=x_1,
\quad
\bg_2=x_2+\frac12x_1^2,\quad \bg_3
=x_3+\frac12x_1x_2+\frac12x_2x_1+\frac16x_1^3\quad\text{and so on.}
\]

As explained in Lemma~\ref{lem_gauge}, the modulo-$(t)$-holomorphicity can be detected by the gauge equation \eqref{hol}. We have the following proposition on this gauge equation.  
\begin{proposition}\label{prop_ob} If the gauge equation \eqref{hol} is solvable, then any $u_i,\ i=1,2,\cdots,k$ in Lemma~\ref{lem_gauge} must be of the form (after modulo $\C\cdot\id$)
\begin{align}\label{eq_ui}u_i=\sum_{I\in\com(i)}\frac{(-1)^{l(I)}}{2^{l(I)}}p_Ix_I,
\end{align}
where $l(I)=N$ if $I=(i_1,\cdots,i_N)$. Substituting \eqref{eq_ui} into the gauge equation \eqref{hol} gives 
\begin{equation}\label{eq_obvan}
\begin{aligned}\bar\eta_i(\theta_0^\st)-\cc\sum_{j+l=i}\bar\eta_j(\bar\p_0 x_l)+\frac14\sum_{j+l=i}\frac{j}{i}[[\theta_0,x_l],x_j]&=\cc\Ch^{1,0}x_i.\\
\frac14\sum_{j+l=i}\frac{j}{i}[\bar\p_0 x_l,x_j]&=\cc[\theta_0^\st,x_i].
\end{aligned}
\end{equation}
\end{proposition}
We remark the equations in \eqref{eq_obvan} give the desired explicit form of $\ob_k$ as explained in Proposition~\ref{prop_obstruction_class_higher_dim}.

\medskip

We define the Euler derivative operator and use it to prove Proposition~\ref{prop_ob}. The Euler operator 
$\mathscr E:=\bar t\frac{\p}{\p\bar t}.$ More precisely for any section $\sum_{1\leq i\leq n} \bar t^i a_i\in\A^0(\End\E)\otimes (\bar t)$, we have
\begin{align*}\mathscr E(\sum_{1\leq i\leq n} \bar t^i a_i)=\sum_{i\leq n} i\cdot\bar t^i a_i,\quad\text{ and we define }\quad
\mathscr E^{-1}(\sum_{1\leq i\leq n} \bar t^i a_i):=\sum_{1\leq i\leq n} \frac{1}{i}\cdot \bar t^i a_i.
\end{align*}

Let $X:=\sum_{1\leq i\leq n}\bar t^i x_i$. Therefore \eqref{eq_obvan} for any $1\leq i\leq k$ can be rewritten as the following equations  modulo $(\bar t^{k+1})$
\begin{equation}\label{eq_compact_obvan}
\begin{aligned}\bar\eta(\theta_0^\st)-\cc\bar\eta(\bar\p_0 X)+\frac14\mathscr E^{-1}[[\theta_0,X],\mathscr E X]&=\cc\Ch^{1,0}X.\\
\frac14\mathscr E^{-1}[\bar\p_0 X,\mathscr EX]&=\cc[\theta_0^\st,X].
\end{aligned}
\end{equation}

\begin{lemma}\label{lem_wonder_iden}
\begin{enumerate}
\item[(i)] Let $R\in \A^0(\End\E)\otimes \C[\bar t]/(\bar t^{n+1})$ be a solution of the following equation
\begin{equation}\label{eq_ordered_R}
\mathscr E R=-\cc R\mathscr E X,
\qquad R(0)=\id.
\end{equation}
Then $R$ must be of the form
\begin{align}\label{eq_R_ordered_formula}
R=\id+\sum_{i=1}^n\bar t^i\cdot\sum_{I\in\com(i)}\frac{(-1)^{l(I)}}{2^{l(I)}}p_Ix_I.
\end{align}
Moreover, we have
\begin{equation}\label{eq_R_inverse_formula}
R^{-1}
=
\id
+
\sum_{m=1}^n
\bar t^m
\sum_{I\in\com(m)}
\frac{1}{2^{l(I)}}p_{I^\vee}x_I,
\end{equation}
where $I^\vee:=(i_N,i_{N-1},\cdots,i_1)$ denote the reversal of $I$ and
\begin{align}\label{eq_R_inverse_equation}
\mathscr E R^{-1}=\cc(\mathscr EX)R^{-1},
\qquad R^{-1}(0)=\id.
\end{align}

\item[(ii)] There exists a unique $L\in \A^0(\End\E)\otimes \C[\bar t]/(\bar t^{n+1})$ that solves the following equation
\begin{equation}\label{eq_ordered_L}
\mathscr E L=-\cc (\mathscr E X)L,
\qquad L(0)=\id.
\end{equation}

Moreover, we have
\begin{equation}\label{eq_L_inverse_formula}
L^{-1}
=
\id
+
\sum_{m=1}^n
\bar t^m
\sum_{I\in\com(m)}
\frac{1}{2^{l(I)}}p_Ix_I.
\end{equation}
And
\begin{align}\label{eq_L_inverse_equation}
\mathscr E L^{-1}=\cc L^{-1}(\mathscr EX),
\qquad L^{-1}(0)=\id.
\end{align}

\item[(iii)] We have
\begin{align}\label{eq_beta_factorization}
G^\st=L^{-1}R^{-1}.
\end{align}
\end{enumerate}
\end{lemma}

\begin{proof}
We first prove (i). Write
\[
R=\id+\sum_{m=1}^n\bar t^m R_m.
\]
Since
\[
\mathscr E R=\sum_{m=1}^n m\bar t^m R_m,\qquad
\mathscr E X=\sum_{m=1}^n m\bar t^m x_m,
\]
equation \eqref{eq_ordered_R} determines \(R_m\) recursively from
\(R_1,\ldots,R_{m-1}\).  Hence if the solution exists, it must be unique. We now find a solution. Write
\[
R
=
\id+
\sum_{m=1}^n
\bar t^m
\sum_{I\in\com(m)}
c_I x_I.
\]
Fix $I=(i_1,\ldots,i_N)\in\com(m).$
The coefficient of \(x_I=x_{i_1}\cdots x_{i_N}\) in
\(\mathscr E R\) is $
m c_I.$ The coefficient of \(x_I\) in
\(-\cc R\mathscr E X\) is $
-\cc i_N c_{(i_1,\ldots,i_{N-1})},$
where by convention \(c_{\emptyset}=1\).  Hence
\begin{align*}
m c_I=(i_1+\cdots+i_N)c_I
=
-\cc i_N c_{(i_1,\ldots,i_{N-1})}.
\end{align*}
Iterating the recursion gives
\[
c_I
=\frac{(-1)^N}{2^N}
\prod_{a=2}^N
\frac{i_a}{i_1+\cdots+i_a}=\frac{(-1)^{l(I)}}{2^{l(I)}}p_I.
\]
Thus \eqref{eq_R_ordered_formula} follows.

It remains to compute \(R^{-1}\). Applying \(\mathscr E\) to $RR^{-1}=\id$ gives $(\mathscr E R)R^{-1}+R(\mathscr E R^{-1})=0.$ Using \eqref{eq_ordered_R}, we obtain $-\cc R(\mathscr E X)R^{-1}+R(\mathscr E R^{-1})=0$ and this gives \eqref{eq_R_inverse_equation}. We can prove \eqref{eq_R_inverse_formula} by a similar recursive argument as above. The statements in (ii) can be proved similarly. 

Finally we prove (iii).  By \eqref{eq_L_inverse_formula} and \eqref{eq_R_inverse_formula}, the coefficient of $
\bar t^N x_I=\bar t^N x_{i_1}\cdots x_{i_N}$
in \(L^{-1}R^{-1}\) is 
\begin{align*}
\sum_{q=0}^N
\frac{1}{2^q}p_{(i_1,\ldots,i_q)}
\cdot
\frac{1}{2^{N-q}}p_{(i_{q+1},\ldots,i_N)^\vee}
=b_I.
\end{align*}
Since this holds for every \(I\), we obtain \eqref{eq_beta_factorization}.
\end{proof}

\begin{proof}[Proof of Proposition~\ref{prop_ob}] We prove this proposition by an induction on $i=1,2,\cdots,k$. The case $i=1$ has been proved in Proposition~\ref{prop_firstob}. We inductively assume \eqref{eq_ui} and \eqref{eq_obvan} for $i=1,2,\cdots, k-1$ and prove them for $k$. By comparing \eqref{eq_ui} and \eqref{eq_R_ordered_formula}, we aim to prove if the gauge equation \eqref{hol} is solvable, then $U\equiv R\pmod{\bar t^{k+1}}$. By induction assumption, we have
\begin{align}\label{eq_UR}
U= R+\bar t^k v_k\pmod{\bar t^{k+1}},\end{align}
for some $v_k\in\A^0(\End\E)$. In this proof, all computations are carried out modulo $(\bar t^{k+1})$.

\textbf{Step1:} we determine $[\bar t^k](\Psi^{0,1}U+\bar\p_0 U)$, i.e. the coefficient of $\bar t^k$ in $(\Psi^{0,1}U+\bar\p_0 U)$. By \eqref{eq_Psi} and \eqref{eq_beta_factorization}, $\Psi^{0,1}$ can be expressed by $\theta_0^\st,\bar\p_0$ and $L,R$. Hence after some simplification, we have
\begin{equation}\label{eq_S_alpha_beta}
2R^{-1}(\Psi^{0,1}R+\bar\p_0 R)
=
L(\bar\p_0-\operatorname{ad}\theta_0^\st)(L^{-1})
+
R^{-1}(\bar\p_0+\operatorname{ad}\theta_0^\st)R.
\end{equation}

We evaluate Euler derivatives of \(L(\bar\p_0-\operatorname{ad}\theta_0^\st)(L^{-1})\) and \(R^{-1}(\bar\p_0+\operatorname{ad}\theta_0^\st)R\).
By the Leibniz rule and \eqref{eq_ordered_L}, \eqref{eq_L_inverse_equation}, we have
\[
\begin{aligned}
\mathscr E\left(L(\bar\p_0-\operatorname{ad}\theta_0^\st)(L^{-1})\right)
&=
(\mathscr E L)(\bar\p_0-\operatorname{ad}\theta_0^\st)(L^{-1})
+
L(\bar\p_0-\operatorname{ad}\theta_0^\st)(\mathscr E L^{-1})  \\
&=
-\cc(\mathscr E X)L(\bar\p_0-\operatorname{ad}\theta_0^\st)(L^{-1})
+
\cc L(\bar\p_0-\operatorname{ad}\theta_0^\st)
\bigl(L^{-1}\mathscr E X\bigr).
\end{aligned}
\]
Using the Leibniz rule, we have
\begin{equation}\label{eq_Ealpha}
\mathscr E(L(\bar\p_0-\operatorname{ad}\theta_0^\st)(L^{-1}))
=
\cc[L(\bar\p_0-\operatorname{ad}\theta_0^\st)(L^{-1}),\mathscr E X]
+
\cc\bigl(\bar\p_0(\mathscr E X)-[\theta_0^\st,\mathscr E X]\bigr).
\end{equation}

Similarly, by the Leibniz rule and \eqref{eq_ordered_R}, \eqref{eq_R_inverse_equation}, we have
\begin{equation}\label{eq_Ebeta}
\mathscr E\left(R^{-1}(\bar\p_0+\operatorname{ad}\theta_0^\st)R\right)
=
-\cc[R^{-1}(\bar\p_0+\operatorname{ad}\theta_0^\st)R,\mathscr E X]
-\cc\bigl(\bar\p_0(\mathscr E X)+[\theta_0^\st,\mathscr E X]\bigr).
\end{equation}

Adding \eqref{eq_Ealpha} and \eqref{eq_Ebeta}, we find
\begin{equation}\label{eq_ES_first}\begin{aligned}
&\mathscr E \left(L(\bar\p_0-\operatorname{ad}\theta_0^\st)(L^{-1})+R^{-1}(\bar\p_0+\operatorname{ad}\theta_0^\st)R\right)
=\mathscr E\left(2R^{-1}(\Psi^{0,1}R+\bar\p_0 R)\right)
\\=&
-[\theta_0^\st,\mathscr E X]
+
\cc[L(\bar\p_0-\operatorname{ad}\theta_0^\st)(L^{-1})-R(\bar\p_0+\operatorname{ad}\theta_0^\st)(R^{-1}),\mathscr E X].\end{aligned}
\end{equation}
Subtracting \eqref{eq_Ebeta} from \eqref{eq_Ealpha}, we get
\begin{align*}
&\mathscr E\left(L(\bar\p_0-\operatorname{ad}\theta_0^\st)(L^{-1})-R^{-1}(\bar\p+\operatorname{ad}\theta_0^\st)R\right)
\\=&
\bar\p_0(\mathscr E X)
+
\cc[L(\bar\p_0-\operatorname{ad}\theta_0^\st)(L^{-1})+R^{-1}(\bar\p_0+\operatorname{ad}\theta_0^\st)R,\mathscr E X].
\end{align*}
Since \(L(\bar\p_0-\operatorname{ad}\theta_0^\st)(L^{-1})-R^{-1}(\bar\p_0+\operatorname{ad}\theta_0^\st)R\) has no constant term, applying \(\mathscr E^{-1}\) gives
\begin{equation}\label{eq_alpha_minus_beta}
L(\bar\p_0-\operatorname{ad}\theta_0^\st)(L^{-1})-R^{-1}(\bar\p_0+\operatorname{ad}\theta_0^\st)R
=
\bar\p_0 X
+
\cc\mathscr E^{-1}[2R^{-1}(\Psi^{0,1}R+\bar\p_0 R),\mathscr E X].
\end{equation}
Substituting \eqref{eq_alpha_minus_beta} into \eqref{eq_ES_first}, we obtain 
\begin{equation}\label{eq_radial_B_identity}\begin{aligned}
&\mathscr E \left(2R^{-1}(\Psi^{0,1}R+\bar\p_0 R)\right)
\\=&
-[\theta_0^\st,\mathscr E X]
+
\cc[\bar\p_0 X,\mathscr E X]
+
\frac14
[\mathscr E^{-1}[2R^{-1}(\Psi^{0,1}R+\bar\p_0 R),\mathscr E X],\mathscr E X].\end{aligned}
\end{equation}

We now take the \(\bar t^k\)-coefficient. Using \eqref{eq_UR} and $\Psi^{0,1}=O(\bar t)$, we have
 \begin{align}\label{eq_secondgauge}(\Psi^{0,1}U+\bar\p_0 U)
=
(\Psi^{0,1}R+\bar\p_0 R)
+
\bar t^k\bar\p_0 v_k.
 \end{align}
 By the induction hypothesis, the gauge equation is already solved in orders \(<k\):
\begin{equation}\label{eq_S_order_n}
\Psi^{0,1}R+\bar\p_0 R=O(\bar t^k),\ \text{ and thus }\ R^{-1}(\Psi^{0,1}R+\bar\p_0 R)=O(\bar t^k).
\end{equation}
Therefore $
[R^{-1}(\Psi^{0,1}R+\bar\p_0 R),\mathscr E X]=O(\bar t^{k+1}),
\quad
[\mathscr E^{-1}[R^{-1}(\Psi^{0,1}R+\bar\p_0 R),\mathscr E X],\mathscr E X]
=
O(\bar t^{k+2}).$ Thus the last term in \eqref{eq_radial_B_identity} has no \(\bar t^k\)-coefficient and the \(\bar t^k\)-coefficient in \eqref{eq_radial_B_identity} is
\[
-k[\theta_0^\st,x_k]
+
\cc\sum_{j+l=k}j[\bar\p_0 x_l,x_j].
\]
Therefore
\begin{equation}\label{eq_BR_n_formula_relabel}
[\bar t^k](\Psi^{0,1}R+\bar\p_0 R)
=
-\cc[\theta_0^\st,x_k]
+
\frac14\sum_{j+l=k}\frac{j}{k}[\bar\p_0 x_l,x_j].
\end{equation}

Using \eqref{eq_secondgauge} and \eqref{eq_BR_n_formula_relabel}, we conclude that
\begin{equation}\label{eq_B_residual_n}
[\bar t^k](\Psi^{0,1}U+\bar\p_0 U)
=
\frac14\sum_{j+l=k}\frac{j}{k}[\bar\p_0 x_l,x_j]
-\cc[\theta_0^\st,x_k]
+
\bar\p_0 v_k.
\end{equation}
Hence if the second gauge equation of \eqref{hol} is solvable at order \(k\), then \eqref{eq_B_residual_n} is zero for some $v_k$.

\textbf{Step2:} we determine $[\bar t^k](\Phi U+[\theta,U])$. By the same Leibniz-rule computation as above, using
\eqref{eq_ordered_L}, \eqref{eq_L_inverse_equation},
\eqref{eq_ordered_R}, and \eqref{eq_R_inverse_equation}, we have
\begin{equation}\label{eq_radial_A_identity}
\begin{aligned}
&\mathscr E\left(
([\theta_0,L]-\Ch^{1,0}L)L^{-1}
-
R^{-1}([\theta_0,R]+\Ch^{1,0}R)
\right)
\\
={}&
\Ch^{1,0}(\mathscr E X)
-\cc[[\theta_0,X],\mathscr E X]
\\
&\quad
+\frac14
\left[
\mathscr E^{-1}
\left[
([\theta_0,L]-\Ch^{1,0}L)L^{-1}
-
R^{-1}([\theta_0,R]+\Ch^{1,0}R),
\mathscr E X
\right],
\mathscr E X
\right].
\end{aligned}
\end{equation}

We also need to relate the expression in \eqref{eq_radial_A_identity} to the first gauge equation. We claim that
\begin{equation}\label{eq_A_relation}
\begin{aligned}
&([\theta_0,L]-\Ch^{1,0}L)L^{-1}
-
R^{-1}([\theta_0,R]+\Ch^{1,0}R)
\\
={}&
2\bar\eta(\theta_0^\st)
-\bar\eta(\bar\p_0 X)
+
\bar\eta\left(2R^{-1}(\Psi^{0,1}R+\bar\p_0 R)\right)
\\
&\quad
-\cc\bar\eta\mathscr E^{-1}
\left[
2R^{-1}(\Psi^{0,1}R+\bar\p_0 R),
\mathscr E X
\right]
-2R^{-1}
\left(
\Phi R+[\theta_0,R]
+
\bar\eta(\Psi^{0,1}R+\bar\p_0 R)
\right).
\end{aligned}
\end{equation}
We prove \eqref{eq_A_relation}. By \eqref{eq_Phi} and \eqref{eq_Psi}, we have
\[
\begin{aligned}
2R^{-1}\Phi R
=
2R^{-1}\bar\eta(\theta_0^\st)R
+
R^{-1}(G^\st)^{-1}
\bigl([\theta_0,G^\st]-\Ch^{1,0}G^\st\bigr)R -
2R^{-1}\bar\eta(\Psi^{0,1})R.
\end{aligned}
\]
Using \eqref{eq_beta_factorization} and Leibniz rule, we get
\begin{align*}
 R^{-1}(G^\st)^{-1}
\bigl([\theta_0,G^\st]-\Ch^{1,0}G^\st\bigr)R
+
2R^{-1}[\theta_0,R] 
=
R^{-1}([\theta_0,R]+\Ch^{1,0}R)-([\theta_0,L]-\Ch^{1,0}L)L^{-1}.
\end{align*}
Combining the above two equations gives
\begin{equation}\label{eq_intermediate_A_relation}\begin{aligned}
&2R^{-1}\left(
\Phi R+[\theta_0,R]+\bar\eta(\Psi^{0,1}R+\bar\p_0 R)
\right)
\\=&
2R^{-1}\bar\eta(\theta_0^\st)R
+
2R^{-1}\bar\eta(\bar\p_0 R)
-\left(([\theta_0,L]-\Ch^{1,0}L)L^{-1}
-
R^{-1}([\theta_0,R]+\Ch^{1,0}R)\right).\end{aligned}
\end{equation}
It remains to rewrite the first two terms on the right. Firstly, we have
\[
2R^{-1}\bar\eta(\theta_0^\st)R
+
2R^{-1}\bar\eta(\bar\p_0 R)
=
2\bar\eta(\theta_0^\st)
+
2\bar\eta\left(R^{-1}(\bar\p_0+\operatorname{ad}\theta_0^\st)R\right).
\]
Using \eqref{eq_S_alpha_beta} and \eqref{eq_alpha_minus_beta}, we get
\[
\begin{aligned}
&2R^{-1}\bar\eta(\theta_0^\st)R
+
2R^{-1}\bar\eta(\bar\p_0 R) \\
={}&
2\bar\eta(\theta_0^\st)
+
2\bar\eta(R^{-1}(\Psi^{0,1}R+\bar\p_0 R))
-
\bar\eta(\bar\p_0 X)
-\bar\eta\mathscr E^{-1}[R^{-1}(\Psi^{0,1}R+\bar\p_0 R),\mathscr E X].
\end{aligned}
\]
Substituting this into \eqref{eq_intermediate_A_relation}, we obtain exactly
\eqref{eq_A_relation}. Using \eqref{eq_S_order_n} and \(\bar\eta=O(\bar t)\), the terms in \eqref{eq_A_relation} involving
\[
\bar\eta\left(2R^{-1}(\Psi^{0,1}R+\bar\p_0 R)\right)
,\ \bar\eta\mathscr E^{-1}
\left[
2R^{-1}(\Psi^{0,1}R+\bar\p_0 R),
\mathscr E X
\right]\quad \text{ and }\quad\bar\eta(\Psi^{0,1}R+\bar\p_0 R)
\]
are \(O(\bar t^{k+1})\).  
Therefore \eqref{eq_A_relation} gives
\begin{equation}\label{eq_A_relation_n}
\begin{aligned}
&[\bar t^k]\left(
([\theta_0,L]-\Ch^{1,0}L)L^{-1}
-
R^{-1}([\theta_0,R]+\Ch^{1,0}R)
\right)
\\
={}&
2\bar\eta_k(\theta_0^\st)
-
\sum_{j+l=k}\bar\eta_j(\bar\p_0 x_l)
-
2[\bar t^k](\Phi R+[\theta_0,R]).
\end{aligned}
\end{equation}

Next we show that the last term in \eqref{eq_radial_A_identity} has no \(\bar t^k\)-coefficient.  For \(m<k\), by induction 
\[
\frac12[\theta_0^\st,x_m]
=
\frac14\sum_{a+c=m}\frac{c}{m}[\bar\p_0 x_a,x_c],\ \text{ and hence }\ 
2[\bar\eta_i(\theta_0^\st),x_m]
=
\sum_{a+c=m}\frac{c}{m}[\bar\eta_i(\bar\p_0 x_a),x_c].
\]
Multiplying by \(m\) and summing over \(i+m=l<k\), we get
\[
[\bar t^l]\left[
2\bar\eta(\theta_0^\st)-\bar\eta(\bar\p_0 X),
\mathscr E X
\right]=0.
\]
By \eqref{eq_A_relation} and induction, this implies
\[
[\bar t^l]\left[
([\theta_0,L]-\Ch^{1,0}L)L^{-1}
-
R^{-1}([\theta_0,R]+\Ch^{1,0}R),
\mathscr E X
\right]=0,
\qquad l<k.
\]
Thus the last term in \eqref{eq_radial_A_identity} has no \(\bar t^k\)-coefficient. Taking the \(\bar t^k\)-coefficient in \eqref{eq_radial_A_identity}, we obtain
\begin{equation}\label{eq_A_radial_n}
k[\bar t^k]\left(
([\theta_0,L]-\Ch^{1,0}L)L^{-1}
-
R^{-1}([\theta_0,R]+\Ch^{1,0}R)
\right)
=
k\Ch^{1,0}x_k
-
\cc\sum_{j+l=k}j[[\theta_0,x_l],x_j].
\end{equation}
Combining \eqref{eq_A_relation_n} and \eqref{eq_A_radial_n}, we get
\begin{equation}\label{eq_AR_n_formula}
[\bar t^k](\Phi R+[\theta_0,R])
=
\bar\eta_k(\theta_0^\st)
-\cc\sum_{j+l=k}\bar\eta_j(\bar\p_0 x_l)
+\frac14\sum_{j+l=k}\frac{j}{k}[[\theta_0,x_l],x_j]
-\cc\Ch^{1,0}x_k.
\end{equation}
Using \eqref{eq_UR} and \(\Phi=O(\bar t)\), we have $
[\bar t^k](\Phi U+[\theta_0,U])
=
[\bar t^k](\Phi R+[\theta_0,R])
+
[\theta_0,v_k].$ Therefore
\begin{equation}\label{eq_A_residual_n}
\begin{aligned}
{}&[\bar t^k](\Phi U+[\theta_0,U])
\\=&{}
\bar\eta_k(\theta_0^\st)
-\cc\sum_{j+l=k}\bar\eta_j(\bar\p_0 x_l)
+\frac14\sum_{j+l=k}\frac{j}{k}[[\theta_0,x_l],x_j]
-\cc\Ch^{1,0}x_k
+[\theta_0,v_k].
\end{aligned}
\end{equation}
Hence if the first gauge equation of \eqref{hol} is solvable at order \(k\), then \eqref{eq_A_residual_n} is zero for some $v_k$.

We introduce the following notion for $i=1,2,\cdots,n$
\begin{align}
\Omega_{i,1}:={}&\bar\eta_i(\theta_0^\st)
-\cc\sum_{j+l=i}\bar\eta_j(\bar\p_0 x_l)
+\frac14\sum_{j+l=i}\frac{j}{i}[[\theta_0,x_l],x_j],\label{eq_Omega1_definition}\\
\Omega_{i,2}:={}&\frac14\sum_{j+l=i}\frac{j}{i}[\bar\p_0 x_l,x_j].\label{eq_Omega2_definition}
\end{align}

\textbf{Step3:} we prove the following identity
\begin{align}\label{eq_class}
\Ch^{1,0} \Omega_{k,2}+[\theta_0^\st,\Omega_{k,1}]=0.
\end{align}
By the definitions of \(\Omega_{k,1}\) and \(\Omega_{k,2}\),
\begin{equation}\label{eq_mixed_start}
\begin{aligned}
&\Ch^{1,0}\Omega_{k,2}+[\theta_0^\st,\Omega_{k,1}] \\
={}&
-\cc\sum_{i+j=k}[\theta_0^\st,\bar\eta_i(\bar\p_0 x_j)]  +
\frac14\sum_{i+j=k}\frac{j}{k}
\left(
[\theta_0^\st,[[\theta_0,x_i],x_j]]
+\Ch^{1,0}[\bar\p_0 x_i,x_j]
\right).
\end{aligned}
\end{equation}
For $\alpha,\beta\in \A^{0,1}(\End\E)$, we have the following identity $
[\bar\eta(\alpha),\beta]=-[\alpha,\bar\eta(\beta)]$ by the definition of the contraction and Lie bracket. Hence
\[
-\cc\sum_{i+j=k}[\theta_0^\st,\bar\eta_i(\bar\p_0 x_j)]
=
\cc\sum_{i+j=k}[\bar\eta_i(\theta_0^\st),\bar\p_0 x_j].
\]
For \(i<k\), the induction hypothesis gives
\[
\bar\eta_i(\theta_0^\st)
=
\cc\Ch^{1,0}x_i
+\cc\sum_{a+c=i}\bar\eta_a(\bar\p_0 x_c)
-\frac14\sum_{a+c=i}\frac{c}{i}[[\theta_0,x_a],x_c].
\]
Substituting this into the previous equation, and we get
\begin{equation}\label{eq_contraction_part}
\begin{aligned}
-\cc\sum_{i+j=k}[\theta_0^\st,\bar\eta_i(\bar\p_0 x_j)]
={}
\frac14\sum_{i+j=k}[\Ch^{1,0}x_i,\bar\p_0 x_j] -
\frac18\sum_{a+b+c=k}\frac{c}{a+c}
[[[\theta_0,x_a],x_c],\bar\p_0 x_b].
\end{aligned}
\end{equation}

For $g,h\in \A^0(\End\E)$, we have the following identity proved in Lemma~\ref{lem_iden_2} at the end of this section
\begin{align*}
&\relax[\theta_0^\st,[[\theta_0,g],h]]
+\Ch^{1,0}[\bar\p_0 g,h]  \\
={}&
-[[\theta_0,[\theta_0^\st,g]],h]
-[[\theta_0,g],[\theta_0^\st,h]]
-\bar\p_0[\Ch^{1,0}g,h]
-[\bar\p_0 g,\Ch^{1,0}h]
-[\bar\p_0 h,\Ch^{1,0}g].
\end{align*}
Taking \(g=x_i\) and \(h=x_j\) and multiplying by \(j/k\) and summing over \(i+j=k\), we get
\begin{equation}\label{eq_HYM_sum}
\begin{aligned}
&\frac14\sum_{i+j=k}\frac{j}{k}
\left(
[\theta_0^\st,[[\theta_0,x_i],x_j]]
+\Ch^{1,0}[\bar\p_0 x_i,x_j]
\right) \\
={}&
-\frac14\sum_{i+j=k}\frac{j}{k}
\left(
[[\theta_0,[\theta_0^\st,x_i]],x_j]
+
[[\theta_0,x_i],[\theta_0^\st,x_j]]
\right)  \\
&-
\frac14\sum_{i+j=k}\frac{j}{k}\bar\p_0[\Ch^{1,0}x_i,x_j]
-\frac14\sum_{i+j=k}[\bar\p_0 x_i,\Ch^{1,0}x_j].
\end{aligned}
\end{equation}
The first term in \eqref{eq_contraction_part} cancels the last term in
\eqref{eq_HYM_sum} after exchanging \(i\) and \(j\).  Therefore
\begin{equation}\label{eq_mixed_reduced}
\begin{aligned}
\Ch^{1,0}\Omega_{k,2}+[\theta_0^\st,\Omega_{k,1}]
={}&
-\frac14\sum_{i+j=k}\frac{j}{k}
\left(
[[\theta_0,[\theta_0^\st,x_i]],x_j]
+
[[\theta_0,x_i],[\theta_0^\st,x_j]]
\right)  \\
&-
\frac14\sum_{i+j=k}\frac{j}{k}\bar\p_0[\Ch^{1,0}x_i,x_j]  -
\frac18\sum_{a+b+c=k}\frac{c}{a+c}
[[[\theta_0,x_a],x_c],\bar\p_0 x_b].
\end{aligned}
\end{equation}
We simplify the middle term.  By the induction hypothesis, for \(i<k\),
\[
\Ch^{1,0}x_i
=
2\bar\eta_i(\theta_0^\st)
-
\sum_{a+c=i}\bar\eta_a(\bar\p_0 x_c)
+
\frac12\sum_{a+c=i}\frac{c}{i}[[\theta_0,x_a],x_c].
\]
Therefore
\begin{equation}\label{eq_middle_expansion}
\begin{aligned}
\sum_{i+j=k}\frac{j}{k}\bar\p_0[\Ch^{1,0}x_i,x_j]  
=&
2\sum_{i+j=k}\frac{j}{k}
\bar\p_0[\bar\eta_i(\theta_0^\st),x_j] -
\sum_{a+c+j=k}\frac{j}{k}
\bar\p_0[\bar\eta_a(\bar\p_0 x_c),x_j] \\&+
\frac12
\sum_{a+b+c=k}
\frac{bc}{k(a+c)}
\bar\p_0[[[\theta_0,x_a],x_c],x_b].
\end{aligned}
\end{equation}
We claim that the first two sums on the right hand side cancel. Indeed, by the second induction equation, for \(j<k\), we get
\[
[\bar\eta_i(\theta_0^\st),x_j]
=
\cc\sum_{c+b=j}\frac{b}{j}
[\bar\eta_i(\bar\p_0 x_c),x_b].
\]
Multiplying by \(2j/k\) and summing over \(i+j=k\), we obtain
\[
2\sum_{i+j=k}\frac{j}{k}
[\bar\eta_i(\theta_0^\st),x_j]
=
\sum_{a+c+j=k}\frac{j}{k}
[\bar\eta_a(\bar\p_0 x_c),x_j].
\]
Thus the first two sums in
\eqref{eq_middle_expansion} cancel.  Hence \eqref{eq_middle_expansion} is 
\begin{equation}\label{eq_middle_term}
\sum_{i+j=k}\frac{j}{k}\bar\p_0[\Ch^{1,0}x_i,x_j]
=
\frac12
\sum_{a+b+c=k}
\frac{bc}{k(a+c)}
\bar\p_0[[[\theta_0,x_a],x_c],x_b].
\end{equation}

Substituting \eqref{eq_middle_term} into \eqref{eq_mixed_reduced}, it remains to check
\begin{equation}\label{eq_final_cancellation_target}
\begin{aligned}
0={}&
2\sum_{i+j=k}\frac{j}{k}
\left(
[[\theta_0,[\theta_0^\st,x_i]],x_j]
+
[[\theta_0,x_i],[\theta_0^\st,x_j]]
\right) \\
&+
\sum_{a+b+c=k}
\left\{
\frac{bc}{k(a+c)}
\bar\p_0[[[\theta_0,x_a],x_c],x_b]
+
\frac{c}{a+c}
[[[\theta_0,x_a],x_c],\bar\p_0 x_b]
\right\}.
\end{aligned}
\end{equation}

Using the second induction equation, 
the first line of \eqref{eq_final_cancellation_target} becomes
\begin{equation}\label{eq_T_expanded}
\begin{aligned}
\sum_{a+b+c=k}
\frac{bc}{k(a+c)}
[[\theta_0,[\bar\p_0 x_a,x_c]],x_b] +
\sum_{a+b+c=k}
\frac{c}{k}
[[\theta_0,x_a],[\bar\p_0 x_b,x_c]].
\end{aligned}
\end{equation}
Applying the Jacobi identities
\begin{align*}
[\theta_0,[\bar\p_0 x_a,x_c]]
&=
[[\theta_0,\bar\p_0 x_a],x_c]
-
[[\theta_0,x_c],\bar\p_0 x_a],\\
[[\theta_0,x_a],[\bar\p_0 x_b,x_c]]
&=
[[[\theta_0,x_a],\bar\p_0 x_b],x_c]
-
[[[\theta_0,x_a],x_c],\bar\p_0 x_b]
\end{align*}
to \eqref{eq_T_expanded}, we get

\begin{align}\label{eq_T_after_Jacobi}
\begin{aligned}
\eqref{eq_T_expanded}
={}&
\sum_{a+b+c=k}
\frac{bc}{k(a+c)}
\left(
[[[\theta_0,\bar\p_0 x_a],x_c],x_b]
-
[[[\theta_0,x_c],\bar\p_0 x_a],x_b]
\right)  \\
&+
\sum_{a+b+c=k}
\frac{c}{k}
\left(
[[[\theta_0,x_a],\bar\p_0 x_b],x_c]
-
[[[\theta_0,x_a],x_c],\bar\p_0 x_b]
\right).
\end{aligned}
\end{align}

On the other hand, applying Leibniz rule of $\bar\p_0$, we get the second line of \eqref{eq_final_cancellation_target} is
\begin{align*}
&-\sum_{a+b+c=k}
\frac{bc}{k(a+c)}
\left(
[[[\theta_0,\bar\p_0 x_a],x_c],x_b]
+
[[[\theta_0,x_a],\bar\p_0 x_c],x_b]
\right) +
\sum_{a+b+c=k}
\frac{c}{k}
[[[\theta_0,x_a],x_c],\bar\p_0 x_b].
\end{align*}
Summing this and \eqref{eq_T_after_Jacobi}, all terms cancel and we get \eqref{eq_class}.

\textbf{Step4:} we finish the induction. By \eqref{eq_B_residual_n} and \eqref{eq_A_residual_n}, the order \(k\) gauge equations are
\begin{equation}\label{eq_order_n_system}
\Omega_{k,2}-\cc[\theta_0^\st,x_k]+\bar\p_0 v_k=0,
\qquad
\Omega_{k,1}-\cc\Ch^{1,0}x_k+[\theta_0,v_k]=0.
\end{equation}
Apply \(\Ch^{1,0}\) to the first equation and apply \([\theta_0^\st,-]\) to the second one.  Adding the two equations and using \eqref{eq_class} gives $\Ch^{1,0}\bar\p_0 v_k+[\theta_0^\st,[\theta_0,v_k]]=0,$ where the \(x_k\)-terms cancel. 
By the ``harmonicity'' Lemma~\ref{vanlem}, \(v_k=c\cdot\id\) for some $c\in\C$ and we may take $v_k=0.$ Therefore \eqref{eq_UR} gives \(U=R\) at order \(k\), and hence we have \eqref{eq_ui} for $i=k$. Moreover, \eqref{eq_order_n_system} gives $\Omega_{k,1}=\cc\Ch^{1,0}x_k,$ and $\Omega_{k,2}=\cc[\theta_0^\st,x_k].$ By the definitions \eqref{eq_Omega1_definition} and \eqref{eq_Omega2_definition}, this is exactly \eqref{eq_obvan} for \(i=k\).  The induction is complete.
\end{proof}

\begin{proposition}[Explicit formula of $\ob_k$]\label{prop_obstruction_class_higher_dim}
Suppose $(\E,\bar\p_t,\theta_t)$ is isomonodromic and modulo-$(t,\bar t^k)$-holomorphic. Then the obstruction class $\ob_k$ defined in Proposition~\ref{prop_existence_ob} is given by
\[
\ob_k=[(\Omega_{k,1},\Omega_{k,2})],
\]
where $\Omega_{k,1}$ and $\Omega_{k,2}$ are defined in \eqref{eq_Omega1_definition} and \eqref{eq_Omega2_definition}.
\end{proposition}

\begin{proof}We first prove $[(\Omega_{k,1},\Omega_{k,2})]$ is a well-defined class in $\barHb$. By the modulo-$(t,\bar t^k)$-holomorphic assumption and Proposition~\ref{prop_ob}, the gauge equivalent Higgs bundle
\[(\E,R^{-1}\circ\bar\p_t\circ R,R^{-1}\circ\theta_t\circ R)\]
satisfies 
\begin{align*}R^{-1}\circ\bar\p_t\circ R=&\bar\p_0-\bar\eta\circ\bar\p_0+\bar t^k(\Omega_{k,2}-\cc[\theta_0^\st,x_k])\pmod{(t,\bar t^{k+1})},\\R^{-1}\circ\theta_t\circ R=&\theta_0+\bar t^k(\Omega_{k,1}-\cc\Ch^{1,0}x_k)\pmod{(t,\bar t^{k+1})},
\end{align*}
where the terms of $\bar t^k$ above are given in the step 1 and step 2 in the proof of Proposition~\ref{prop_ob}.

The integrability condition \eqref{eq_compatibility} of the Higgs bundle gives the vanishing of 
\[
[\theta_0,\Omega_{k,1}-\cc\Ch^{1,0}x_k],
\ 
\bar\p_0 (\Omega_{k,1}-\cc\Ch^{1,0}x_k)+[\theta_0,\Omega_{k,2}-\cc[\theta_0^\st,x_k]],
\ 
\bar\p_0 (\Omega_{k,2}-\cc[\theta_0^\st,x_k]).
\]

Let $\mathcal D':=\Ch^{1,0}+\operatorname{ad}(\theta_0^\st)$ and $\mathcal D''=\bar\p_0+\operatorname{ad}\theta_0$ and thus \begin{align}\label{eq_D2_van}\mathcal D''(\Omega_{k,1}-\cc\Ch^{1,0}x_k+\Omega_{k,2}-\cc[\theta_0^\st,x_k])=0.\end{align}
 
We have already proved the following in the step 3 in the proof of Proposition~\ref{prop_ob}
\begin{align*}
\Ch^{1,0}\Omega_{k,2}+[\theta_0^\st,\Omega_{k,1}]=0.
\end{align*}
This implies
\begin{align}\label{eq_11_condition}
\Ch^{1,0}(\Omega_{k,2}-\cc[\theta_0^\st,x_k])+[\theta_0^\st,\Omega_{k,1}-\cc\Ch^{1,0}x_k]=0.
\end{align}

The K\"ahler identities for the harmonic bundle give $
(\mathcal D'')^*=-\sqrt{-1}[\Lambda,\mathcal  D']$. This together with \eqref{eq_11_condition} gives
\begin{align*}(\mathcal D'')^*(\Omega_{k,1}-\cc\Ch^{1,0}x_k+\Omega_{k,2}-\cc[\theta_0^\st,x_k])=0.
\end{align*}
Together with \eqref{eq_D2_van}, this shows that \((\Omega_{k,1}-\cc\Ch^{1,0}x_k+\Omega_{k,2}-\cc[\theta_0^\st,x_k])\) is \(\mathcal D''\)-harmonic.

By the harmonic bundle identities, $
\Delta_{\mathcal D''}=\Delta_{\mathcal D'},$ the form \((\Omega_{k,1}-\cc\Ch^{1,0}x_k+\Omega_{k,2}-\cc[\theta_0^\st,x_k])\) is also \(\mathcal D'\)-harmonic. In particular, it is annihilated by $\mathcal D'$ and hence $$[(\Omega_{k,1},\Omega_{k,2})]\in\barHb.$$

It remains to identify its class with the obstruction class. Suppose first that the modulo-$(t,\bar t^{k+1})$-holomorphicity holds. Then we have the vanishing equation \eqref{eq_obvan} for $i=k$ by Proposition~\ref{prop_ob}. Hence $[(\Omega_{k,1},\Omega_{k,2})]=0.$

Conversely, suppose that $
[(\Omega_{k,1},\Omega_{k,2})]=0.$
Then there exists \(f_k\in\A^0(\End\E)\) such that
\[
\Omega_{k,1}=\Ch^{1,0}f_k,
\qquad
\Omega_{k,2}=[\theta_0^\st,f_k].
\]
The integrability condition \eqref{eq_compatibility} of the Higgs bundle gives 
\begin{align*}\bar\p_0 (\Omega_{k,1}-\cc\Ch^{1,0}x_k)+[\theta_0,\Omega_{k,2}-\cc[\theta_0^\st,x_k]]=\bar\p_0 \Ch^{1,0}(f_k-\cc x_k)+[\theta_0,[\theta_0^\st,f_k-\cc x_k]]=0.
\end{align*}
By the ``harmonicity'' Lemma~\ref{vanlem}, we have $f_k=\cc x_k+c\cdot\id$ for some $c\in\C$. This fact and \eqref{eq_order_n_system} imply the modulo-$(t,\bar t^{k+1})$-holomorphicity.
\end{proof}

Now we prove the following technical lemma.
\begin{lemma}\label{lem_iden_2}
For any $g,h\in\A^0(\End\E)$, 
we have 
\begin{align*}
&[\theta_0^\st,  [[\theta_0,g],h]]  +\Ch^{1,0}([\bar\p_0 g,h])\\
=&  
-[[\theta_0,[\theta_0^\st,g]],h]
-[[\theta_0,g],[\theta_0^\st,h]]
-\bar\p_0([\Ch^{1,0}g,h])
-[\bar\p_0 g,\Ch^{1,0}h]
-[\bar\p_0 h,\Ch^{1,0}g]
\end{align*}
\end{lemma}
\begin{proof}
By the Jacobi identity, we have
 \begin{align}\label{eq_firststep_C1}
 [\theta_0^\st,  [[\theta_0,g],h]]=[[\theta_0,g],[h,\theta_0^\st]]+[h,[\theta_0^\st,[\theta_0,g]]].
 \end{align}
By applying the Jacobi identity again to $[\theta_0^\st,[\theta_0,g]]$, we have
 \begin{align*}[h,[\theta_0^\st,[\theta_0,g]]]=[h,[\theta_0,[\theta_0^\st,g]]]+[h,[g,[\theta_0^\st,\theta_0]]].
 \end{align*}
 Substituting this into \eqref{eq_firststep_C1}, we have
 \begin{align*} [\theta_0^\st,  [[\theta_0,g],h]]=&-[[\theta_0,g],[\theta_0^\st,h]]-[[\theta_0,[\theta_0^\st,g]],h]+[h,[g,[\theta_0^\st,\theta_0]]]\\
 =&-[[\theta_0,g],[\theta_0^\st,h]]-[[\theta_0,[\theta_0^\st,g]],h]+[h,(\Ch^{1,0}\bar\p_0+\bar\p_0\Ch^{1,0})g]
 \end{align*}
by the identity  
$[[\theta_0,\theta_0^\st],g]=-F(\Ch)g=-(\Ch^{1,0}\bar\p_0+\bar\p_0\Ch^{1,0})g$ following from the Hermitian--Yang--Mills--Higgs equation. Therefore
 \begin{align*} &[\theta_0^\st,  [[\theta_0,g],h]]+\Ch^{1,0}([\bar\p_0 g,h])
 \\=&-[[\theta_0,g],[\theta_0^\st,h]]-[[\theta_0,[\theta_0^\st,g]],h]-[(\Ch^{1,0}\bar\p_0+\bar\p_0\Ch^{1,0})g,h]+\Ch^{1,0}([\bar\p_0 g,h])\\
 =&-[[\theta_0,g],[\theta_0^\st,h]]-[[\theta_0,[\theta_0^\st,g]],h]-[\bar\p_0\Ch^{1,0}g,h]-[\bar\p_0 g,\Ch^{1,0}h]\\
 =& 
-[[\theta_0,[\theta_0^\st,g]],h]
-[[\theta_0,g],[\theta_0^\st,h]]
-\bar\p_0([\Ch^{1,0}g,h])
-[\bar\p_0 g,\Ch^{1,0}h]
-[\bar\p_0 h,\Ch^{1,0}g]. \qedhere
 \end{align*}
 \end{proof}

\newpage
\section{First-order holomorphicity}\label{sec_prop_first_hol}
In this section, we use the first-order approach to study several questions. In Section~\ref{sec_41} and Section~\ref{sec_42}, we study the commutativity between $\C^*$-action and the isomonodromic deformation and try to answer Question~\ref{Ker}. In Section~\ref{sec_43}, we determine the Zariski tangent space of the non-abelian Noether--Lefschetz locus defined in \eqref{eq_NLlocus} and relate it to the first-order holomorphicity of $\sigma_{\Dol}$ discussed in Proposition~\ref{prop_firstob}.
\subsection{On the commutativity between $S^1$-action and the isomonodromic deformation}\label{sec_41}
For any first-order deformation $X_1$ of $X_0$, let $(\E,\bar\p_0,\theta_0)$ be a stable Higgs bundle on $X_0$ and $(\E,\bar\p_t,\theta_t)$ be the isomonodromic deformation of the initial Higgs bundle along $X_1$. Let $h_t$ be its harmonic metric defined in \eqref{metric}. We reduce Theorem~\ref{thm_main_1} to the following proposition, which can be viewed as the first-order version of Theorem~\ref{thm_main_1}.

\begin{proposition}\label{prop_firstorder} For any $\lambda\in S^1\backslash \{\pm 1\}$, we consider the real-analytic deformation $(\E,\bar\p_t,\lambda\cdot\theta_t)$. Its associated family of flat bundles has smooth relative flat connection
\begin{align}\label{eq_flatconnection}D_t:=D_{h_t}+\lambda\theta_t+\bar\lambda\theta_t^{\star_{h_t}}.
\end{align}
Moreover, the following statements are equivalent:
\begin{enumerate}
\item[(i)] the isomonodromic deformation $(\E,\bar\p_t,\theta_t)$ on $X_1$ is holomorphic;
\item[(ii)] the family of flat connections $D_t$ is a constant family of flat connections on $X_1$, that is, up to a gauge transformation, $D_t= \Ch+\lambda\theta_0+\bar\lambda\theta_0^\st$ on $X_1$.
\end{enumerate}
\end{proposition}
Now we explain how to derive Theorem~\ref{thm_main_1} from Proposition~\ref{prop_firstorder}.
\begin{proof}[Proof of Theorem~\ref{thm_main_1} using Proposition~\ref{prop_firstorder}]Without loss of generality, we may assume $U$ is smooth. If $U$ is singular, we resolve the singularity, denoted by $\pi:\hat U\to U$. By pulling back \( X_U/U \) and \( \sigma_{\Dol}|_U \) to \( \hat{U} \), and then we just prove Theorem~\ref{thm_main_1} on $\hat U$ and descend to $U$. 

Thus $U$ is a smooth complex analytic subvariety. For any point in $U$ and any direction $v\in TU$, the first-order version of Theorem~\ref{thm_main_1} is true by Proposition~\ref{prop_firstorder}. This implies we have the claim in Theorem~\ref{thm_main_1} by the first-order deformation theory.
\end{proof}

\begin{proof}[Proof of Proposition~\ref{prop_firstorder}]
By \cite[Page 45]{Simp92}, the harmonic metric of the family $(\E,\bar\p_t,\lambda\cdot\theta_t)$ is still $h_t$, and thus we have \eqref{eq_flatconnection}.
 Since $(\E,\bar\p_t,\theta_t)$ is isomonodromic, we have
$$D=D_{h_t}+ \theta_t+\theta_t^{\star_{h_t}}=\Ch+\theta_0+\theta_0^\st.$$
Substituting the above equation into \eqref{eq_flatconnection}, we obtain 
\begin{align}\label{eq_D_t}D_t=D_{h_t}+\lambda \theta_t+\bar\lambda\theta_t^{\star_{h_t}}=\Ch+\lambda\theta_0+\bar\lambda\theta_0^\st+(\lambda-1)(\theta_t-\theta_0)+(\bar\lambda-1)(\theta_t^{\star_{h_t}}-\theta_0^\st).
\end{align}
We have the following first-order deformation equations by Proposition~\ref{prop_Taylor_Exp_DolHig}
\begin{equation}\label{eq_firstorder_D}\begin{aligned}\theta_t=&\theta_0+t(-\eta_1(\theta_0)+\cc[\theta_0,g_1]-\cc \Ch^{1,0} g_1)+\bar t(\bar \eta_1(\theta_0^{\star_{h_0}})+\cc[\theta_0,\bg_1]-\cc \Ch^{1,0} \bg_1),\\
\theta_t^{\star_{h_t}}=&\theta_0^{\star_{h_0}}+t(\eta_1(\theta_0)+\cc[\theta_0^{\star_{h_0}},g_1]-\cc \bar\p_0 g_1)+\bar t(-\bar \eta_1(\theta_0^{\star_{h_0}})+\cc[\theta_0^{\star_{h_0}},\bg_1]-\cc \bar\p_0 \bg_1).
\end{aligned}\end{equation}
Now we prove the equivalence between (i) and (ii).

\noindent \textbf{(i) implies (ii):} by the vanishing equation \eqref{eq_ob1exact} in Proposition~\ref{prop_firstob}, we have
\begin{align}\label{eq_van_ob1}\bar\eta_1(\theta_0^\st)&=\cc\Ch^{1,0}\bg_1,\quad 0=\cc[\theta_0^\st,\bg_1];\qquad 
\eta_1(\theta_0)=\cc\bar\p_0 g_1,\quad 0=\cc[\theta_0,g_1],
\end{align}
where the last two equations are obtained by applying $\st$ to the first two. Substituting those equations into \eqref{eq_firstorder_D}, we obtain
\begin{align*}\theta_t=\theta_0-t\cdot\cc\Ch g_1+\bar t\cdot\cc[\theta_0,\bg_1],\qquad
\theta_t^{\star_{h_t}}=\theta_0^{\star_{h_0}}+t\cdot\cc[\theta_0^{\star_{h_0}},g_1]-\bar t\cdot\cc \Ch \bg_1.
\end{align*}
Substituting those equations into \eqref{eq_D_t}, we obtain
\begin{align*}D_t=\Ch+\lambda\theta_0+\bar\lambda\theta_0^\st+&(\lambda-1)\Big(-t\cdot\cc\Ch g_1+\bar t\cdot\cc[\theta_0,\bg_1]\Big)\\
+&(\bar\lambda-1)\Big(t\cdot\cc[\theta_0^{\star_{h_0}},g_1]-\bar t\cdot\cc \Ch \bg_1\Big).
\end{align*}
Applying the gauge transformation $\mathscr G:=\id+t\cdot\frac{\lambda-1}{2}g_1+\bar t\cdot\frac{\bar\lambda-1}{2}\bg_1$ to the above $D_t$, we obtain
\begin{align*}\mathscr G^{-1}\circ D_t\circ \mathscr G=&\Ch+\lambda\theta_0+\bar\lambda\theta_0^\st+(\Ch+\lambda\theta_0+\bar\lambda\theta_0^\st)^{\End}\mathscr G+\\
&(\lambda-1)\Big(-t\cdot\cc\Ch g_1+\bar t\cdot\cc[\theta_0,\bg_1]\Big)
+(\bar\lambda-1)\Big(t\cdot\cc[\theta_0^{\star_{h_0}},g_1]-\bar t\cdot\cc \Ch \bg_1\Big)\\
=&\Ch+\lambda\theta_0+\bar\lambda\theta_0^\st+(\lambda\theta_0+\bar\lambda\theta_0^\st)^{\End}\mathscr G+(\lambda-1)\bar t\cdot\cc[\theta_0,\bg_1]
+(\bar\lambda-1)t\cdot\cc[\theta_0^{\star_{h_0}},g_1]\\
=&\Ch+\lambda\theta_0+\bar\lambda\theta_0^\st,
\end{align*}
where the last equality follows from \eqref{eq_van_ob1}.

\noindent \textbf{(ii) implies (i):} by assumption, there exists a gauge transformation $\mathscr G=\id+t\cdot f+\bar t\cdot w\in \A^0(\End\E)\otimes \C[t,\bar t]/(t,\bar t)^2$ such that 
$$\mathscr G^{-1}\circ D_t\circ \mathscr G=\Ch+\lambda\theta_0+\bar\lambda\theta_0^\st.$$
Combining this with \eqref{eq_D_t} yields 
\begin{align}\label{eq_flat1}(\Ch+\lambda\theta_0+\bar\lambda\theta_0^\st)^{\End}(t\cdot f+\bar t\cdot w)+(\lambda-1)(\theta_t-\theta_0)+(\bar\lambda-1)(\theta_t^{\star_{h_t}}-\theta_0^\st)=0.\end{align}
Comparing the coefficient of $t$ in \eqref{eq_flat1}, we obtain
\begin{align*}&\Ch f+\lambda[\theta_0,f]+\bar\lambda[\theta_0^\st,f]+(\lambda-1)(-\eta_1(\theta_0)+\cc[\theta_0,g_1]-\cc \Ch^{1,0} g_1)\\
&+(\bar\lambda-1)(\eta_1(\theta_0)+\cc[\theta_0^{\star_{h_0}},g_1]-\cc \bar\p_0 g_1)=0.
\end{align*}
By comparing $(1,0)$ and $(0,1)$ parts of the above equation, we obtain
\begin{equation}\label{eq_obflatvan}\begin{aligned}\Ch^{1,0} f+\lambda[\theta_0,f]+(\lambda-1)(\cc[\theta_0,g_1]-\cc \Ch^{1,0} g_1)=0\\
\bar\p_0 f+\bar\lambda[\theta_0^\st,f]+(-\lambda+\bar\lambda)\eta_1(\theta_0)+(\bar\lambda-1)(\cc[\theta_0^{\star_{h_0}},g_1]-\cc \bar\p_0 g_1)=0.
\end{aligned}\end{equation}
Applying $\bar\p_0(-)$ to the first equation above and $\lambda\cdot[\theta_0,-]$ to the second equation above and summing them, we obtain
\begin{align*}\bar\p_0\Ch^{1,0} (f-\frac{\lambda-1}{2}g_1)+[\theta_0,[\theta_0^\st,f-\frac{\lambda-1}{2}g_1]=0,
\end{align*}
which gives $f=\frac{\lambda-1}{2}g_1+c\cdot \id$ for some $c\in \C$ by the ``harmonicity'' in Lemma~\ref{vanlem}. Substituting this into \eqref{eq_obflatvan}, we obtain
\begin{align*}\frac{\lambda^2-1}{2}[\theta_0,g_1]=0,\quad
\frac{\lambda-\bar\lambda}{2}\bar\p_0 g_1=(\lambda-\bar\lambda)\eta_1(\theta).
\end{align*}
Since $\lambda\ne\pm1$, we obtain $0=\cc[\theta_0,g_1],\ \eta_1(\theta_0)=\cc\bar\p_0 g_1$, which implies $[(\bar\eta_1(\theta_0^\st),0)]$ vanishes and gives the first-order holomorphicity by Proposition~\ref{prop_firstob}.
\end{proof}

\subsection{On the commutativity between $\mathbb R^*$-action and the isomonodromic deformation}\label{sec_42}
In this section, we give some evidence for Question~\ref{Ker} for $\lambda\in\mathbb R^*$. 
\begin{proposition}\label{prop_pm1}When $\lambda=\pm 1$, $\lambda\cdot\sigma_{\Dol}$ is an isomonodromic deformation.
\end{proposition}
\begin{proof}
It suffices to prove the case $\lambda=-1$. Fix a point $0\in S$ and a vector $v\in T_0S$.  We may choose a first-order curve $\gamma:\operatorname{Spec}A_1\to S$ with $\gamma(0)=0$ and $\gamma_*(\frac{d}{d t})=v$. The first-order argument used in Proposition~\ref{prop_firstorder} then proves the claim. 

Combining \eqref{eq_D_t} and \eqref{eq_firstorder_D}, we obtain
\begin{align*}D_t=&\Ch-\theta_0-\theta_0^\st-2(\theta_t-\theta_0)-2(\theta_t^{\star_{h_t}}-\theta_0^\st)\\
=&\Ch-\theta_0-\theta_0^\st+t(\Ch g_1-[\theta_0+\theta_0^\st,g_1])+\bar t(\Ch\bg_1-[\theta_0+\theta_0^\st,\bg_1]).
\end{align*}
Thus $(\id-tg_1-\bar t \bg_1)^{-1}\circ D_t\circ(\id-tg_1-\bar t \bg_1)=\Ch-\theta_0-\theta_0^\st$.

There is an alternative proof. Note that by \cite[Lemma 2.11]{Simp92}, the $-1$-action has the following property: for any Higgs bundle $(E,\theta)$, let $\rho$ be the corresponding monodromy representation by $\NHC$.  Then the dual Higgs bundle of $(E,-\theta)$ corresponds to the complex conjugate representation $\overline\rho^T$. This property proves our claim.
\end{proof}

Now we prove Question~\ref{Ker} for rank-one Higgs bundles. 

\begin{proposition}\label{prop_rank1}Let $(E_0,\theta_0)$ be a rank-one stable Higgs bundle on $X_0$ and let $\sigma_{\Dol}$ be its isomonodromic deformation along $X/S$. For any $\lambda\in\mathbb R^*$, the deformation $\lambda\cdot\sigma_{\Dol}$ is also isomonodromic.
\end{proposition}
We remark that this proposition together with Theorem~\ref{thm_main_1} implies Question~\ref{Ker} for rank-one Higgs bundles.
\begin{proof}[Proof of Proposition~\ref{prop_rank1}]
In the rank-one case, the $\NHC$ between $M_{\mathrm{B}}(X_0)$ and $M_{\Dol}(X_0)$ can be described explicitly: assume the rank of the abelianization of $\pi_1(X_0)$ is $2p$. Then
\begin{align*}M_{\mathrm{B}}(X_0)=(\C^*)^{2p}=(S^1)^{2p}\times (\mathbb R_+)^{2p}\quad\text{and}\quad M_{\Dol}(X_0)=\mathrm{Pic}^0(X_0)\times H^0(\Omega^1_{X_0}).
\end{align*}
There is a homeomorphism from $\mathrm{Pic}^0(X_0)$ to $(S^1)^{2p}$, which coincides with $\NHC$ on $\mathrm{Pic}^0(X_0)\times 0$. Let $\{\gamma_i\}_{i=1}^{2p}$ be a set of generators of the abelianization of $\pi_1(X_0)$. For the Higgs bundle $[(E_0,\theta_0)]\in M_{\Dol}(X_0)$, the associated monodromy representation by $\NHC$ is \begin{align*}\rho:\pi_1(X_0)&\to \mathbb C^*\\
\gamma_i&\mapsto \rho(\gamma_i)=\rho_1(\gamma_i)\rho_2(\gamma_i),\end{align*}
where $\{\rho_1(\gamma_i)\}_{i=1}^{2p}\in (S^1)^{2p}$ and $\rho_2(\gamma_i)=\exp\Big(-\int_{\gamma_i}\theta_0+\bar\theta_0\Big)\in \mathbb R_+$. See \cite[Page 21]{Simp92} for this property for the curve case. 

For any $s\in S$, $\sigma_{\Dol}(s)=[(E_s,\theta_s)]\in M_{\Dol}(X_s)$ is the isomonodromic deformation. Thus we obtain the corresponding monodromy representation $\rho_s\in M_{\mathrm{B}}(X_s)\cong (\mathbb C^*)^{2p}$. By assumption, $\{\rho_s(\gamma_i)\}_{i=1}^{2p}$ is independent of $s$. For $\lambda\in\mathbb R^*$, consider $\lambda\cdot\sigma_{\Dol}(s)=[(E_s,\lambda\cdot\theta_s)]$ and the corresponding monodromy representation $\rho_s^{(\lambda)}$. It suffices to prove that, for each $\gamma_i$ with $i=1,2,\cdots,2p$, $\rho_s^{(\lambda)}(\gamma_i)$ is independent of $s$. By definition, $\{\rho_{1,s}^{(\lambda)}(\gamma_i)\}_{i=1}^{2p}=\{\rho_{1,s}(\gamma_i)\}_{i=1}^{2p}\in (S^1)^{2p}$, so $\rho_{1,s}^{(\lambda)}$ is independent of $s$. Moreover, for each $\gamma_i$ with $i=1,2,\cdots,2p$,
\begin{align*}\rho_{2,s}^{(\lambda)}(\gamma_i)=\exp\Big(-\int_{\gamma_i}(\lambda\theta_s+\overline{\lambda\theta_s})\Big)=(\rho_{2,s}(\gamma_i))^\lambda,
\end{align*}
so $\rho_{2,s}^{(\lambda)}$ is independent of $s$.
\end{proof}

\subsection{The Zariski tangent space of the non-abelian Noether--Lefschetz locus}\label{sec_43}

Let $(\E,\bar\p_0,\theta_0)$ be a graded stable Higgs bundle on $X_0$ with weight $w$. First, we have canonical decomposition
$\End \E=\bigoplus\limits_{l=-w}^w(\End\E)^{l,-l}$,
where $$(\End \E)^{l,-l}:=\{f\in \End \E\mid f(\E^{p,w-p})\subset \E^{p+l,w-p-l}\}.$$
We extend the grading by setting $(\End \E)^{l,-l}=0$ for all $|l|>w$, so that we have 
\begin{align*}
\End \E = \bigoplus_{l\in \mathbb{Z}} (\End\E)^{l,-l}.
\end{align*}
Recall the setup for the non-abelian Noether--Lefschetz locus given in equation \eqref{eq_NLlocus}. Let $$\theta_{0,*}:H^1(T_{X_0})\to \mathbb{H}^1\bigl(X_0,(\End E_0,\operatorname{ad}(\theta_0))\bigr)$$ be the non-abelian Higgs field defined in \eqref{eq_nabHiggs} and let $\tau_0:T_0^{1,0}S\to H^1(T_{X_0})$ be the Kodaira--Spencer map of the family $X/S$ at $0$. By Proposition~\ref{prop_firstob}, the condition $\theta_{0,*}\circ\tau_0(v)=0$ is equivalent to the first-order holomorphicity of the Dolbeault $\sigma_{\Dol}$ along any tangent vector $v\in T_0S$. By virtue of this result, the first-order truncated case of Theorem~\ref{thm_main_tru} reduces to the following theorem. For all higher-order truncations, we use the same strategy to establish the full statement of Theorem~\ref{thm_main_tru} in Section~\ref{sec_high_hol_proof}.

\begin{thm}\label{thm_Zar_NL}
The Zariski tangent space of $\mathcal{NL}$ at $0\in S$ is  
\begin{align*}T^{\mathrm{Zar}}_0\mathcal{NL}=\{v\in T^{1,0}_0S\mid  v\in \ker(\theta_{0,*}\circ\tau_0)\}.
\end{align*}
\end{thm}
\begin{proof}
By the same argument as in \cite[Theorem C]{HSZ}, for any $v\in T^{1,0}_0S$, if $v\notin \ker(\theta_{0,*}\circ\tau_0)$, then the isomonodromically deformed Higgs bundle $\sigma_{\Dol}$ is not graded (in fact not nilpotent) along $v$. Thus $v\notin T^{\mathrm{Zar}}_0\mathcal{NL}$.

If $v\in \ker(\theta_{0,*}\circ\tau_0)$, we prove the isomonodromically deformed Higgs bundle $\sigma_{\Dol}$ coincides with a holomorphic family of graded Higgs bundles up to first order along $v$. Let $[\eta_1]:=\tau_0(v)\in H^1(T_{X_0})$. The condition $\theta_{0,*}([\eta_1])=0$ is equivalent to the existence of $f_1\in \A^0(\End\E)$ such that \eqref{eq_ob1vanish} holds after applying $\st$. Thus, by step 2 in the proof of Proposition~\ref{prop_firstob}, we know that $\bg_1=2f_1+c\cdot \id$, for some $c\in\C$. Thus \eqref{eq_ob1vanish} gives
\begin{align}\label{eq_ob1_iden}\bar\eta_1(\theta_0^\st)=\cc\Ch^{1,0}\bg_1,\quad 0=\cc[\theta_0^\st,\bg_1].\end{align}
Write $\bg_1=\sum_a {(\bg_1)}^{a,-a}$.  Since
$\bar\eta_1(\theta_0^\st)\in \A^{1,0}\bigl((\End E)^{1,-1}\bigr)$, for every $a\neq1$ we have
\[
  \Ch^{1,0}{(\bg_1)}^{a,-a}=0,
  \qquad
  [\theta_0^\st,{(\bg_1)}^{a,-a}]=0.
\]
By stability, ${(\bg_1)}^{a,-a}=c\cdot\id$ for some $c\in \C$. Thus after normalization, we have \begin{align}\label{eq_bg1_type}\bg_1=(\bg_1)^{1,-1}.\end{align}

Substituting \eqref{eq_ob1_iden} into the deformation terms $\varphi_1,\psi_1,\alpha_1,\beta_1$ in Proposition~\ref{prop_Taylor_Exp_DolHig}, we have
\begin{align*}\varphi_1=\cc[\theta_0,\bg_1],&\quad \psi_1=\cc\bar\p_0\bg_1,\\
\alpha_1=-\eta_1(\theta_0)-\cc\Ch^{1,0}g_1,&\quad \beta_1=-\cc[\theta_0^\st,g_1]+\bar\p_0 g_1.
\end{align*}
By taking $\mathscr U:=\id-\frac{t}{2}g_1-\frac{\bar t}{2}\bg_1$, we have on $X_1$
\begin{align}\label{eq_firstdeforma}\mathscr U^{-1}\circ\bar\p_t\circ\mathscr U=\pi''_{\eta}\Ch+t\cdot\cc[\theta_0^\st,g_1],\qquad
\mathscr U^{-1}\circ\theta_t\circ\mathscr U=\theta_0-t\eta_1(\theta_0)-t\cdot\cc\Ch^{1,0}g_1.
\end{align}

By \eqref{eq_ob1_iden}, it follows immediately that
\begin{align*}
[\theta_0^\st,g_1]\in\mathcal{A}^{1}((\End\E)^{0,0}) \qquad \text{and} \qquad \eta_1(\theta_0)+\cc\Ch^{1,0}g_1\in \mathcal{A}^{1}((\End\E)^{-1,1}).
\end{align*}
Together with \eqref{eq_firstdeforma}, this implies that the isomonodromically deformed Higgs bundle preserves a graded structure up to first order.\end{proof}

We will extend this method to higher orders to answer Question~\ref{EsnKer}.

\newpage

\section{Higher-order holomorphicity}\label{sec_high_hol_proof}
In this section we assume the isomonodromic deformation is holomorphic up to order $n$. Using the obstruction theory developed in Section~\ref{sec_ob_hol}, we give a simple gauge normal form of the isomonodromic deformation in Theorem~\ref{thm_sec3_ordered_normal_forms}. This normal form together with the previous higher order theory proves Theorem~\ref{thm_main_tru} on the holomorphic isomonodromic deformation of a graded Higgs bundle and Theorem~\ref{thm_sec4_arbitrary_nilpotency} on the holomorphic isomonodromic deformation of a generically regular nilpotent Higgs bundle simultaneously. Their proofs are given in Section~\ref{subsec_sec3_graded_case} and Section~\ref{sec:regular-nilpotent-case}.

\subsection{The holomorphic Taylor coefficients}
\label{subsec_sec3_holomorphic_taylor}
Assume the isomonodromic deformation $(\E,\bar\p_t,\theta_t)$ of $(\E,\bar\p_0,\theta_0)$ is holomorphic on $X_n$. Recall $G:=\id+\sum\limits_{i=1}^n t^i g_i$, where $g_i$ is defined in \eqref{metric}. Let
\begin{align*}
\mathscr{U} = \operatorname{id} + \sum_{i=1}^n \bar{t}^i u_i + \sum_{i=0}^{n-1} \sum_{j=1}^{n-i} t^j \bar{t}^i u_{j\bar i}\in\A^0(\End\E)\otimes B_n,
\end{align*}
be a gauge transformation such that $(\E,\mathscr U^{-1}\circ \bar\p_t\circ\mathscr U,\mathscr U^{-1}\circ \theta_t\circ\mathscr U)$ satisfies \eqref{eq_hol}; the existence of such $\mathscr U$ is guaranteed by the holomorphicity assumption.
We may choose $u_{j\bar 0}$ freely without affecting the validity of \eqref{eq_hol}. We will take $V:=\id+\sum_{j=1}^n t^j u_{j\bar 0}$ to be the unique solution of the following equation:
\begin{align}\label{eq_V}\mathfrak EV=-\cc V\mathfrak EY, \quad V(0)=\id,
\end{align}
where we define $y_m:=x_m^\st$ (the notion $x_m$ is given in \eqref{eq_beta_change}), and 
\begin{equation}  \label{eq_sec3_Y_E_def}
  Y:=\sum_{m=1}^nt^my_m,
  \qquad
  \mathfrak E:=t\frac{\partial}{\partial t}.
\end{equation}
We remark that the same recursive argument as in Lemma \ref{lem_wonder_iden} shows \eqref{eq_V} admits a unique solution expressed in terms of $y_1,\cdots,y_n$. Similarly, we have $W\in\mathcal A^0(\End\E)\otimes A_n$ as a unique solution of 
\begin{align}\label{eq_W}\mathfrak EW=-\cc(\mathfrak EY)W,\qquad W(0)=\id.
\end{align}
Note that $y_m$ is a polynomial of $g_1,\cdots ,g_m$. Similarly, we can prove
\begin{equation}  \label{eq_sec3_G_factorization}
  G=W^{-1}V^{-1}.
\end{equation}

By the holomorphicity Proposition~\ref{prop_criterion_holo_deform} and Proposition~\ref{prop_Taylor_Exp_DolHig}, there exists $A'\in \mathcal A^{1,0}(\End\E)\otimes (t)$ and $B'\in \A^{0,1}(\End\E)\otimes (t)$ such that
\begin{align*}
\mathscr U^{-1}\circ\theta_t\circ \mathscr U=&  V^{-1}(\theta_0+A)V=P'_\eta(\theta_0+A'),\\
\mathscr U^{-1}\circ\bar\p_t\circ\mathscr U=&  V^{-1}\circ (\pi''_\eta\Ch+B)\circ V
  =\pi''_\eta\Ch+P''_\eta(B').
\end{align*}
Note that the explicit formulas of $A,B$ are given in Proposition~\ref{prop_Taylor_Exp_DolHig}, and hence we have the following formulas on $A'$ and $B'$:
\begin{align}
  AV+[\theta_0,V]&=VA',
  \label{eq_sec3_A_gauge_equation}\\
  \bar\p_0 V+\eta(\Ch^{1,0}V)+B^{0,1}V&=VB'.
  \label{eq_sec3_B_gauge_equation}
\end{align}
We prove the following normalization theorem for holomorphic isomonodromic deformation
\begin{thm}[Higher-order normalization]
\label{thm_sec3_ordered_normal_forms}Assume the isomonodromic deformation $(\E,\bar\p_t,\theta_t)$ of $(\E,\bar\p_0,\theta_0)$ is holomorphic on $X_n$. Let $A'$ and $B'$ be the coefficients defined by
\eqref{eq_sec3_A_gauge_equation} and \eqref{eq_sec3_B_gauge_equation}.
Then
\begin{equation}  \label{eq_sec3_Aprime_Bprime_ordered}
  A'=-\frac{1}{2}\Ch^{1,0}Y,
  \qquad
  B'=-\frac{1}{2}[\theta_0^\st,Y].
\end{equation}
Consequently,
\begin{align*}
  \mathscr U^{-1}\circ\theta_t\circ \mathscr U
  =P'_\eta\left(\theta_0-\frac{1}{2}\Ch^{1,0}Y\right),\quad 
\mathscr U^{-1}\circ\bar\p_t\circ \mathscr U
  =\pi''_\eta\Ch
    +P''_\eta\left(-\frac{1}{2}[\theta_0^\st,Y]\right).
\end{align*}
\end{thm}

To prove this theorem, it suffices to prove \eqref{eq_sec3_Aprime_Bprime_ordered}. For the calculation below, set
\begin{align*}
  F_+
  &:=( [\theta_0,W]-\Ch^{1,0}W )W^{-1}
     -V^{-1}(\Ch^{1,0}V+[\theta_0,V]),\\
  S_+
  &:=( [\theta_0,W]-\Ch^{1,0}W )W^{-1}
     +V^{-1}(\Ch^{1,0}V+[\theta_0,V]),\\
  F_-
  &:=( [\theta_0^\st,W]-\bar\p_0 W )W^{-1}
     -V^{-1}(\bar\p_0 V+[\theta_0^\st,V]),\\
  S_-
  &:=( [\theta_0^\st, W]-\bar\p_0 W)W^{-1}
     +V^{-1}(\bar\p_0 V+[\theta_0^\st,V]).
\end{align*}
We first give the following lemma.

\begin{lemma}
\label{lem_sec3_Euler_solutions}
One can verify that the Leibniz rule and \eqref{eq_V} and \eqref{eq_W} give the Euler systems
\begin{align}
  \mathfrak EF_+
  &=\Ch^{1,0}(\mathfrak EY)+\frac{1}{2}[S_+,\mathfrak EY],
  &
  \mathfrak ES_+
  &=-[\theta_0,\mathfrak EY]+\frac{1}{2}[F_+,\mathfrak EY],
  \label{eq_sec3_Euler_plus}\\
  \mathfrak EF_-
  &=\bar\p_0 (\mathfrak EY)+\frac{1}{2}[S_-,\mathfrak EY],
  &
  \mathfrak ES_-
  &=-[\theta_0^\st,\mathfrak EY]+\frac{1}{2}[F_-,\mathfrak EY].
  \label{eq_sec3_Euler_minus}
\end{align}
The unique solutions of the above Euler systems with zero constant term are
\begin{align}
  F_+&=\Ch^{1,0}Y,
  &S_+&=0,
  \label{eq_sec3_plus_solution}\\
  F_-&=\bar\p_0 Y-\frac{1}{2}\mathfrak E^{-1}
          [[\theta_0^\st,Y],\mathfrak EY]
       =2\eta(\theta_0)-\eta(\Ch^{1,0}Y),
  &S_-&=-[\theta_0^\st,Y].
  \label{eq_sec3_minus_solution}
\end{align}
\end{lemma}

\begin{proof}
By holomorphicity, we have \eqref{eq_compact_obvan}. Taking $\st$ in \eqref{eq_compact_obvan} gives
\begin{align}
  \eta(\theta_0)
  -\frac{1}{2}\eta(\Ch^{1,0}Y)
  +\frac{1}{4}\mathfrak E^{-1}
       [[\theta_0^\st,Y],\mathfrak EY]
  &=\frac{1}{2}\bar\p_0 Y,
  \label{eq_sec3_holo_radial_1}\\
  \frac{1}{4}\mathfrak E^{-1}[\Ch^{1,0}Y,\mathfrak EY]
  &=\frac{1}{2}[\theta_0,Y].
  \label{eq_sec3_holo_radial_2}
\end{align}
Applying $\mathfrak E$ to \eqref{eq_sec3_holo_radial_2} gives $
  \frac{1}{2}[\Ch^{1,0}Y,\mathfrak EY]=[\theta_0,\mathfrak EY].$ Hence $F_+=\Ch^{1,0}Y$ and $S_+=0$ solve
\eqref{eq_sec3_Euler_plus}.

For the minus system, substitute $S_-=-[\theta_0^\st,Y]$ into the first equation
of \eqref{eq_sec3_Euler_minus}.  We obtain
\[
  F_-=\bar\p_0 Y-\frac{1}{2}\mathfrak E^{-1}
       [[\theta_0^\st,Y],\mathfrak EY].
\]
Equation \eqref{eq_sec3_holo_radial_1} gives the second expression for $F_-$.
It remains to check
\begin{equation}
  [F_-,\mathfrak EY]=0.
  \label{eq_sec3_Fminus_commute}
\end{equation}
Recall $\eta=\sum_{p=1}^nt^p\eta_p$.  Applying $\mathfrak E$ to
\eqref{eq_sec3_holo_radial_2} and taking the coefficient of $t^q$ gives
\[
  2q[\theta_0,y_q]
  =\sum_{r+s=q}s[\Ch^{1,0}y_r,y_s]\quad\text{and hence}\quad 
  2q[\eta_p(\theta_0),y_q]
  =\sum_{r+s=q}s[\eta_p(\Ch^{1,0}y_r),y_s].
\]
After summing over $p+q=m$, this is exactly the vanishing of the coefficient
of $t^m$ in
\[
  [2\eta(\theta_0)-\eta(\Ch^{1,0}Y),\mathfrak EY].
\]
Thus \eqref{eq_sec3_Fminus_commute} holds.  The second equation in
\eqref{eq_sec3_Euler_minus} follows.  Uniqueness in both systems is obtained
recursively from the Euler equations.
\end{proof}

\begin{proof}[Proof of Theorem~\ref{thm_sec3_ordered_normal_forms}]
Using the factorization
\eqref{eq_sec3_G_factorization} and \eqref{eq_sec3_A_gauge_equation}, a direct Leibniz-rule calculation gives $ 2A'=-F_+.$
Lemma~\ref{lem_sec3_Euler_solutions} therefore yields $A'=-\frac12\Ch^{1,0}Y$.

Similarly, 
\eqref{eq_sec3_G_factorization} and \eqref{eq_sec3_B_gauge_equation} give $
  B'=\frac{1}{2}S_-+\frac{1}{2}\eta(S_+).$
Since $S_+=0$ and $S_-=-[\theta_0^\st,Y]$, we obtain
$B'=-\frac12[\theta_0^\st,Y]$.  This proves \eqref{eq_sec3_Aprime_Bprime_ordered} and the rest is clear by definition.
\end{proof}

\subsection{The graded case}
\label{subsec_sec3_graded_case}
Let $(\E,\bar\p_0,\theta_0)$ be a graded stable Higgs bundle on $X_0$ with weight $w$. In this section we give a short proof of Theorem~\ref{thm_main_tru} via the normalization Theorem~\ref{thm_sec3_ordered_normal_forms}.

\begin{proposition}[Type properties]
\label{prop_sec3_purity_x}
Recall the type decomposition $\End \E=\bigoplus\limits_{l=-w}^w(\End\E)^{l,-l}$ defined in Section~\ref{sec_43}. For every $1\leq m\leq n$, we have
\begin{equation}
  x_m=(x_m)^{1,-1}.
  \label{eq_sec3_x_pure}
\end{equation}
\end{proposition}

\begin{proof}
We argue by induction on $m$.  For $m=1$, we have $x_1=\bg_1$ by \eqref{eq_beta_change}. Thus $x_1=x_1^{1,-1}$ by \eqref{eq_bg1_type}. Assume now that $x_k=x_k^{1,-1}$ for all $k<m$.  By \eqref{eq_obvan} and induction, we have
\begin{align*}
  \Ch^{1,0}x_m^{a,-a}=0,\ [\theta_0^\st,x_m^{a,-a}]=0
  \qquad(a\neq1).
\end{align*}
By stability, we have $x_m=x_m^{1,-1}$ after omitting a scalar in $\C\cdot\id$.
\end{proof}

\begin{proof}[Proof of Theorem~\ref{thm_main_tru}]
By holomorphicity on $X_n$, we have Theorem~\ref{thm_sec3_ordered_normal_forms}. Thus $(\E,\bar\p_t,\theta_t)$ is gauge equivalent to the Higgs bundle
\begin{align*}
  (\E, \mathscr U^{-1}\circ\bar\p_t\circ \mathscr U, \mathscr U^{-1}\circ\theta_t\circ \mathscr U)=(\E, \pi''_\eta\Ch
    +P''_\eta\left(-\frac{1}{2}[\theta_0^\st,Y]\right),P'_\eta\left(\theta_0-\frac{1}{2}\Ch^{1,0}Y\right)).
\end{align*}
Taking $\st$ in
\eqref{eq_sec3_x_pure}, we obtain
$ y_m=y_m^{-1,1}$, and $Y\in\A^0\bigl((\End E)^{-1,1}\bigr)\otimes(t).$ Thus we have
\[
  \theta_0-\frac{1}{2}\Ch^{1,0}Y
  \in\A^{1,0}\bigl((\End E)^{-1,1}\bigr)\otimes A_n,\ \text{ and } [\theta_0^\st,Y]\in\A^{0,1}\bigl((\End E)^{0,0}\bigr)\otimes(t).
\]
Therefore $(\E, \mathscr U^{-1}\circ\bar\p_t\circ \mathscr U, \mathscr U^{-1}\circ\theta_t\circ \mathscr U)$ is a graded Higgs bundle on $X_n$ and so is $(\E,\bar\p_t,\theta_t)$.
\end{proof}

\subsection{The generically regular nilpotent case}
\label{sec:regular-nilpotent-case}
Let $(\E,\bar\p_0,\theta_0)$ be a stable generically regular nilpotent Higgs bundle of rank $r$ on $X_0$. In this section we give a proof of Theorem~\ref{thm_sec4_arbitrary_nilpotency} via the normalization Theorem~\ref{thm_sec3_ordered_normal_forms}. Under the assumption of Theorem~\ref{thm_sec4_arbitrary_nilpotency}, the isomonodromic deformation $(\E,\bar\p_t,\theta_t)$ is holomorphic along $X_n$. By Theorem~\ref{thm_sec3_ordered_normal_forms}, we have the following normalization by the holomorphicity on $X_n$
\begin{equation}  \label{eq_sec4_M_pure_normal_form}
  \mathscr U^{-1}\theta_t\mathscr U=P'_\eta(M),
  \qquad
  M:=\theta_0-\frac{1}{2}\Ch^{1,0}Y.
\end{equation}
It is enough to prove nilpotency for $M$. By definition and \eqref{eq_sec3_holo_radial_2}, one has
\begin{equation}
  [M,\mathfrak EY]=0,
  \qquad
  \mathfrak EM=-\frac{1}{2}\Ch^{1,0}\mathfrak EY.
  \label{eq_sec4_radial_identities}
\end{equation}

\begin{lemma}[Trace normalization]
\label{lem_sec4_trace_normalization}
For each $1\leq i\leq n$, $\tr(y_i)$ is a constant function on $X_0$, where $y_i$ is defined in \eqref{eq_sec3_Y_E_def}. 
\end{lemma}

\begin{proof}
Taking traces in \eqref{eq_sec3_holo_radial_1}, using
$\tr(\eta(\theta_0))=\eta(\tr\theta_0)=0$ and the vanishing of the trace of a
commutator, gives
\begin{align}\label{eq_tr_y}
  \bar\p \tr(Y)+\eta\bigl(\p\tr(Y)\bigr)=0.
\end{align}
Thus $\tr(y_1)$ is a holomorphic function on $X_0$ by comparing the $ t$-coefficient of \eqref{eq_tr_y}. Hence $\tr(y_1)$ is a constant. By comparing the $t^2$-coefficient of \eqref{eq_tr_y}, we have the constancy of $\tr(y_2)$ and so on.
\end{proof}

The generically regular nilpotent assumption gives the following property:
\begin{lemma}
\label{lem_sec4_generically_regular}
Suppose $(E_0,\theta_0)$ is generically regular nilpotent on $X_0$. Then there is a dense Zariski-open subset $X_0^{\mathrm{reg}}\subset X_0$ such that, for every
$x\in X_0^{\mathrm{reg}}$, the set
\begin{align*}
  \mathcal R_x
  :=\left\{\xi\in T_{X_0,x}\;\middle|\;
      \theta_0(\xi)(x)\text{ is regular nilpotent}\right\}
\end{align*}
is a nonempty Zariski-open subset of
$T_{X_0,x}$.
\end{lemma}
Fix $x\in X_0^{\mathrm{reg}}$ and $\xi\in\mathcal R_x$.  Choose a local
holomorphic vector field $v$ of $X_0$ with $v(x)=\xi$.  After shrinking the
neighborhood of $x$, we may assume that $ N_v:=\theta_0(v)$
is regular nilpotent throughout that neighborhood.  Put
\begin{align*} 
  M_v:=M(v)=N_v+O(t).
\end{align*}
By the regular nilpotency, we can choose a local section $e$ of $\E$ such that $e,N_ve,\ldots,N_v^{r-1}e$ is a frame.  Since $M_v=N_v+O(t)$, we know that $e,M_ve,\ldots,M_v^{r-1}e$ is also a frame.

\begin{lemma}
\label{lem_sec4_cyclic_centralizer}
There exist smooth functions $a_0,\cdots,a_{r-1}$ of $X_0$ near $x$ with coefficients in $A_n$ such that
\begin{equation}  \label{eq_sec4_Z_polynomial}
  \mathfrak EY=\sum_{j=0}^{r-1}a_jM_v^j.
\end{equation}
Moreover, $a_j=O(t)$ for any $j$.
\end{lemma}

\begin{proof}
There are functions $a_j$ for which
\[
  (\mathfrak EY)e=(a_0\id+a_1M_v+\cdots+a_{r-1}M_v^{r-1})e,
\]
because $e,M_ve,\ldots,M_v^{r-1}e$ is a frame. The first identity in \eqref{eq_sec4_radial_identities} gives
$[M_v,\mathfrak EY]=0$. Thus we have \eqref{eq_sec4_Z_polynomial} by considering $(\mathfrak EY)(M_v^ie)$. Since $\mathfrak EY=O(t)$ and $\id,N_v,\ldots,N_v^{r-1}$ are linearly independent, we have $a_j=O(t)$.
\end{proof}

For $m\geq2$, the second identity of \eqref{eq_sec4_radial_identities} gives
\begin{equation}  \label{eq_sec4_radial_hmv_start}
  \mathfrak E \tr(M_v^m)
  =-\frac{m}{2}\tr\bigl(M_v^{m-1}D^{1,0}_{h_0,v}\mathfrak EY\bigr).
\end{equation}
Differentiating \eqref{eq_sec4_Z_polynomial}, we obtain
\[
  D^{1,0}_{h_0,v}\mathfrak EY
  =\sum_{j=0}^{r-1}(\p_va_j)M_v^j
   +\sum_{j=1}^{r-1}a_jD^{1,0}_{h_0,v}(M_v^j)=\sum_{j=0}^{r-1}(\p_va_j)M_v^j+\sum_{j=1}^{r-1}a_j\sum_{\ell=0}^{j-1}
    M_v^\ell(D^{1,0}_{h_0,v}M_v)M_v^{j-1-\ell}.
\]
Substituting this into \eqref{eq_sec4_radial_hmv_start} and taking trace, we have
\begin{equation}  \label{eq_sec4_power_sum_system}
  \mathfrak E \tr(M_v^m)
  =-\frac{m}{2}\sum_{j=0}^{r-1}(\p_va_j)\tr(M_v^{m+j-1})
   -\frac{m}{2}\sum_{j=1}^{r-1}
       \frac{j}{m+j-1}a_j\p_v \tr(M_v^{m+j-1}).
\end{equation}

\begin{proof}[Proof of Theorem~\ref{thm_sec4_arbitrary_nilpotency}]
Fix $x\in X_0^{\mathrm{reg}}$ and a regular direction
$\xi\in\mathcal R_x$, and choose $v$ as above.  We first prove,
simultaneously for every $m\geq1$, that
\begin{equation}
  \tr(M_v^m)=O(t^{n+1}).
  \label{eq_sec4_hmv_final_order}
\end{equation}
Since $N_v$ is nilpotent and by \eqref{eq_sec4_M_pure_normal_form}, we have $ \tr(M_v^m)=O(t)$ for every $m\geq1$. When $m=1$, by Lemma~\ref{lem_sec4_trace_normalization}, we have
\[\tr(M_v)=\tr N_v-\frac{1}{2}\p_v\tr(Y)=0.\]

Assume inductively on $1\leq q\leq n$, that $\tr(M_v^m)=O(t^q)$, for every $m\geq2.$ Then $\p_v \tr(M_v^m)=O(t^q)$ for all $m$.  By Lemma~\ref{lem_sec4_cyclic_centralizer}, every term on the right-hand side of
\eqref{eq_sec4_power_sum_system} is $O(t^{q+1})$.  Therefore
\[
  \mathfrak E \tr(M_v^m)=O(t^{q+1})
  \qquad(m\geq2).
\]
 This proves
\eqref{eq_sec4_hmv_final_order}. By Cayley--Hamilton's theorem, we have $ M_v^r=0
  \pmod{t^{n+1}}$. By the definition of $M$ and the density in Lemma~\ref{lem_sec4_generically_regular}, we obtain that $\theta_t^r=O(t^{n+1})$.
\end{proof}

\appendix
\newpage

\section{Real-analytic manifolds and real-analytic deformations}
\label{sec_complexification}

In this appendix we spell out the convention used in the paper for
real-analytic deformations of holomorphic objects.  The guiding principle is
that a real-analytic family is obtained by restricting a holomorphic family on
the complexification of the real-analytic base to the diagonal.

\subsection{Real-analytic functions and complexification}

Let \(\mathbb D_{\epsilon}\subset \mathbb R^2\) be a sufficiently small disk
centered at the origin. Write $z=x+\sqrt{-1}y,\quad
  \bar z=x-\sqrt{-1}y.$ A function $f:\mathbb D_{\epsilon}\longrightarrow \C$ is said to be real-analytic near \(0\) if it admits a convergent expansion
\[
  f(x,y)=\sum_{i,j\geq 0}a_{ij}z^i\bar z^j
\]
near \(0\).  Two such functions define the same germ at \(0\) if they agree on
some neighborhood of \(0\).

Let $\mathcal R_0$ be the ring of germs of \(\C\)-valued real-analytic functions at
\(0\in\mathbb R^2\), and let \(\mathcal R\) be the corresponding sheaf on
\(\mathbb R^2\).

We regard $\C=\mathbb R^2$ as the complex line with its standard complex structure, and $\overline{\C}$ as the same real vector space with the opposite complex structure.  We write
\(z\) for the holomorphic coordinate on \(\C\), and write \(\zeta\) for the
holomorphic coordinate on \(\overline{\C}\).  Under the diagonal embedding
\[
  i:\mathbb R^2\hookrightarrow \C\times\overline{\C},
  \qquad
  (x,y)\longmapsto (z,\bar z),
\]
one has $ i^*z=z,\quad i^*\zeta=\bar z.$

\begin{lemma}
\label{lem:real-analytic-functions-complexification}
There is a natural isomorphism of sheaves of \(\C\)-algebras
\[
  \mathcal R\simeq i^*\mathcal O_{\C\times\overline{\C}}.
\]
Equivalently, every real-analytic germ in the variables \((z,\bar z)\) is the
restriction of a holomorphic germ in the independent variables \((z,\zeta)\).
\end{lemma}

\begin{proof}
The assertion is local.  A holomorphic function on a sufficiently small
polydisc in \(\C\times\overline{\C}\) has a convergent expansion $F(z,\zeta)=\sum\limits_{i,j\geq 0}a_{ij}z^i\zeta^j.$ Restricting to the diagonal \(\zeta=\bar z\) gives the real-analytic function $i^*F(z,\bar z)=\sum\limits_{i,j\geq 0}a_{ij}z^i\bar z^j.$

Conversely, if $f(z,\bar z)=\sum\limits_{i,j\geq 0}a_{ij}z^i\bar z^j$ is a convergent real-analytic germ, then $F(z,\zeta):=\sum\limits_{i,j\geq 0}a_{ij}z^i\zeta^j$ is a convergent holomorphic germ on a sufficiently small polydisc in
\(\C\times\overline{\C}\), and \(i^*F=f\).  These local constructions glue to the desired isomorphism. 
\end{proof}
By a similar argument, this complexification lemma also holds for real-analytic germs at $0\in\mathbb R^{2k}$.
Let \(M\) be a complex manifold and let \(\overline M\) be the conjugate
complex manifold.  We write \(M^\circ\) for the underlying real-analytic
manifold.  The diagonal embedding is
\[
  i_{M^\circ}:M^\circ\hookrightarrow M\times\overline M.
\]
We define the sheaf of \(\C\)-valued real-analytic functions on \(M^\circ\) by $ \mathcal R_{M^\circ}:=i_{M^\circ}^*\mathcal O_{M\times\overline M}.$ Thus the associated real-analytic space is the locally ringed space $(M^\circ,\mathcal R_{M^\circ}).$ The sheaf $\mathcal R_{M^{\circ}}$ carries a canonical complex conjugation. For example, the complex conjugation of $F(z,\bar z)=\sum a_{IJ}z^I\bar z^J$ is $\sum \overline{a_{IJ}}\bar z^I z^J$.

The complexified tangent bundle of \(M^\circ\) decomposes as
\[
  \C TM^\circ
  =
  TM^\circ\otimes_{\mathbb R}\C
  =
  T^{1,0}M^\circ\oplus T^{0,1}M^\circ.
\]
Under the above complexification, this is naturally identified with $
  i_{M^\circ}^*T_M\oplus i_{M^\circ}^*T_{\overline M}.$

\subsection{Infinitesimal real-analytic disks}

For \(n\geq 1\), define
\[
  A_n:=\C[t]/(t^{n+1}),
  \qquad
  \overline A_n:=\C[\bar t]/(\bar t^{n+1}),\qquad 
\text{and}\qquad
  B_n:=\C[t,\bar t]/(t,\bar t)^{n+1}.
\]
Here \(\bar t\) is a formal variable.  Geometrically, \(t\) is the holomorphic
coordinate on the first factor of the complexified disk, while \(\bar t\) is
the holomorphic coordinate on the conjugate factor.

For \(n=1\), we have $B_1=\C[t,\bar t]/(t,\bar t)^2
      =\C\oplus \C t\oplus \C \bar t,$ and its maximal ideal
\[
  I_{B_1}:=(t,\bar t)\cong \C t\oplus \C\bar t.
\]
satisfies \(I_{B_1}^2=0\). Equivalently, $B_1\simeq A_1\times_{\C}\overline A_1.$ We remark that $B_n$ also carries a canonical complex conjugation: the complex conjugation of $\sum a_{ij}t^i\bar t^j$ is $\sum \bar a_{ij}\bar t^i t^j$.

The definition of the complexified tangent
space gives the following property:
\begin{proposition}
\label{prop:real-analytic-first-jets}
Let \(x\in M^\circ\). We define
\[\mathrm{Hom}^{\dagger}_{\C\text{-alg}}
  \bigl(\mathcal R_{M^\circ,x},B_1\bigr)\]
to be $\C$-algebra morphism compatible with both complex conjugations. Precisely, for \[F\in R_{M^\circ,x},\quad\text{and}\quad\varphi\in \mathrm{Hom}^{\dagger}_{\C\text{-alg}}
  \bigl(\mathcal R_{M^\circ,x},B_1\bigr),\] we have $\varphi(\bar F)=\overline{\varphi(F)}$. We call it 
the set of first-order real-analytic arcs in
\(M^\circ\) through \(x\) and this is 
naturally identified with the complexified tangent space:
\[\mathrm{Hom}^{\dagger}_{\C\text{-alg}}
  \bigl(\mathcal R_{M^\circ,x},B_1\bigr)
  \cong\C T_xM^\circ\cong
  T_{M,x}\oplus T_{\overline M,x}.
\]

For general \(n\), we define the space of \(n\)-jets of real-analytic arcs
through \(x\) by
\[
  J_n^{\mathbb R\mathrm{an}}(M^\circ)_x
  :=
  \mathrm{Hom}^{\dagger}_{\C\text{-alg}}
  \bigl(\mathcal R_{M^\circ,x},B_n\bigr).
\]
\end{proposition}

\subsection{Real-analytic deformations of complex manifolds}

Let \(X\) be a complex manifold.  A real-analytic deformation of \(X\) over
\(M^\circ\) is a diagram
\[
\begin{tikzcd}
X \arrow[r, hook] \arrow[d]
&
\mathcal X^\circ \arrow[r, hook] \arrow[d]
&
\mathcal X \arrow[d]
\\
\operatorname{Spec}\C \arrow[r, hook]
&
M^\circ \arrow[r, hook, "i_{M^\circ}"]
&
M\times\overline M
\end{tikzcd}
\]
such that \(\mathcal X\to M\times\overline M\) is a holomorphic deformation,
and $\mathcal X^\circ
  =
  \mathcal X\times_{M\times\overline M}M^\circ.$

Equivalently, a real-analytic deformation is the restriction to the diagonal
of a holomorphic deformation over the complexification of the base.

For an infinitesimal real-analytic deformation over \(B_n\), we have a
cartesian diagram
\[
\begin{tikzcd}
X \arrow[r, hook] \arrow[d]
&
X_n \arrow[d, "\pi"]
\\
\operatorname{Spec}\C \arrow[r, hook]
&
\operatorname{Spec}B_n.
\end{tikzcd}
\]

Assume now \(n=1\).  The cotangent sequence for
\(\pi:X_1\to\operatorname{Spec}B_1\) gives
\[
  0
  \longrightarrow
  I_{B_1}^{\vee}\otimes_{\C}\mathcal O_X
  \longrightarrow
  \Omega^1_{X_1}\big|_X
  \longrightarrow
  \Omega_X^1
  \longrightarrow
  0.
\]
Its extension class is the real-analytic Kodaira--Spencer class
\[
  \mathrm{KS}^{\mathbb R\mathrm{an}}(X_1)
  \in
  \mathrm{Ext}^1_X
  \bigl(\Omega_X^1,I_{B_1}^{\vee}\otimes_{\C}\mathcal O_X\bigr).
\]
Since \(X\) is smooth, $\mathrm{Ext}^1_X(\Omega_X^1,\mathcal O_X)
  =
  H^1(X,T_X).$ Therefore $\mathrm{KS}^{\mathbb R\mathrm{an}}(X_1)
  \in
  H^1(X,T_X)\otimes_{\C}I_{B_1}^{\vee}.$

Using $I_{B_1}^{\vee}\cong \C\,dt\oplus\C\,d\bar t,$ we obtain a decomposition
\[
  \mathrm{KS}^{\mathbb R\mathrm{an}}(X_1)
  =
  \mathrm{KS}_t(X_1)\,dt+\mathrm{KS}_{\bar t}(X_1)\,d\bar t\quad\text{where}\quad 
\mathrm{KS}_t(X_1),\mathrm{KS}_{\bar t}(X_1)\in H^1(X,T_X).
\]

If one records the \(\bar t\)-direction after conjugation,
then the same decomposition is written as
\[
  \mathrm{KS}^{\mathbb R\mathrm{an}}(X_1)
  \in
  H^1(X,T_X)
  \oplus
  H^1(\overline X,T_{\overline X}).
\]
The first summand is the holomorphic Kodaira--Spencer direction, and the
second summand is the anti-holomorphic direction.

\begin{proposition}
\label{prop:KS-real-analytic-manifold}
First-order real-analytic deformations of \(X\) over
\(\operatorname{Spec}B_1\) are classified by
\[
  H^1(X,T_X)\otimes_{\C}I_{B_1}^{\vee}
  \cong
  H^1(X,T_X)\oplus H^1(X,T_X).
\]
Equivalently, after conjugating the \(\bar t\)-direction, this classification
may be written as
\[
  H^1(X,T_X)\oplus H^1(\overline X,T_{\overline X}).
\]
\end{proposition}

\begin{proof}
Take a sufficiently fine Stein open cover \(\{U_i\}\) of \(X\).  A first-order
deformation over \(B_1\) is obtained by gluing the trivial thickenings $U_i\times\operatorname{Spec}B_1$ by transition functions of the form
\[
  z_i
  =
  f_{ij}(z_j)
  +t\,v_{ij}(z_j)
  +\bar t\,w_{ij}(z_j),
\]
where \(f_{ij}\) are the original transition functions of \(X\), and
\(v_{ij},w_{ij}\) are holomorphic vector fields on \(U_{ij}\).  The cocycle
condition modulo \((t,\bar t)^2\) says precisely that
\[
  \{v_{ij}\}\in Z^1(\{U_i\},T_X),
  \qquad
  \{w_{ij}\}\in Z^1(\{U_i\},T_X).
\]
Changing the local trivializations modifies these cocycles by coboundaries.
Thus the isomorphism class of the deformation is determined by
\[
  \bigl([\{v_{ij}\}],[\{w_{ij}\}]\bigr)
  \in
  H^1(X,T_X)\oplus H^1(X,T_X).
\]
Conversely, any such pair of cocycles defines a first-order deformation by the
above gluing formula.  This proves the classification.  The final formulation
with $H^1(\overline X,T_{\overline X})$ is obtained by applying complex conjugation to the \(\bar t\)-part.
\end{proof}

\subsection{Real-analytic deformations of coherent sheaves on a fixed space}

Let \(X\) be a fixed complex manifold or, more generally, a fixed complex
scheme, and let \(\mathcal F\) be a coherent \(\mathcal O_X\)-module.

\begin{definition}
\label{def:ra-sheaf-fixed}
A real-analytic deformation of \(\mathcal F\) over \(M^\circ\), with the space
\(X\) fixed, is a diagram
\[
\begin{tikzcd}
\mathcal F \arrow[r] \arrow[d]
&
\mathcal F^\circ \arrow[r] \arrow[d]
&
\mathcal F^{\mathbb C} \arrow[d]
\\
X\times \operatorname{Spec}\C \arrow[r, hook]
&
X\times M^\circ \arrow[r, hook]
&
X\times M\times\overline M
\end{tikzcd}
\]
such that \(\mathcal F^{\mathbb C}\) is a coherent sheaf on
\(X\times M\times\overline M\), flat over \(M\times\overline M\), and $\mathcal F^\circ
  =
  \mathcal F^{\mathbb C}
  \big|_{X\times M^\circ}.$

\end{definition}

For the first-order base \(B_1\), this is the same as a coherent $\mathcal O_X\otimes_{\C}B_1$-module \(\mathcal F_1\), flat over \(B_1\), together with an isomorphism $\mathcal F_1\otimes_{B_1}\C\simeq \mathcal F.$ Because \(I_{B_1}^2=0\), every such deformation fits into an exact sequence of
\(\mathcal O_X\)-modules
\[
  0
  \longrightarrow
  \mathcal F\otimes_{\C}I_{B_1}
  \longrightarrow
  \mathcal F_1
  \longrightarrow
  \mathcal F
  \longrightarrow
  0.
\]
The extension class is
\[
  \mathrm{KS}^{\mathbb R\mathrm{an}}(\mathcal F_1)
  \in
  \mathrm{Ext}^1_X
  \bigl(
    \mathcal F,
    \mathcal F\otimes_{\C}I_{B_1}
  \bigr).
\]
Since $I_{B_1}\cong \C t\oplus \C\bar t,$ we have
\[
  \mathrm{Ext}^1_X
  \bigl(
    \mathcal F,
    \mathcal F\otimes_{\C}I_{B_1}
  \bigr)
  \cong
  \mathrm{Ext}^1_X(\mathcal F,\mathcal F)
  \oplus
  \mathrm{Ext}^1_X(\mathcal F,\mathcal F).
\]

\begin{proposition}
\label{prop:ra-fixed-sheaf-KS}
First-order real-analytic deformations of a coherent sheaf \(\mathcal F\) on a
fixed \(X\) over \(\operatorname{Spec}B_1\) are classified by
\[
  \mathrm{Ext}^1_X(\mathcal F,\mathcal F)\otimes_{\C}I_{B_1}
  \cong
  \mathrm{Ext}^1_X(\mathcal F,\mathcal F)
  \oplus
  \mathrm{Ext}^1_X(\mathcal F,\mathcal F).
\]
The two components are the \(t\)- and \(\bar t\)-Kodaira--Spencer classes of
the real-analytic family.
\end{proposition}

\begin{proof}
The standard deformation theory of coherent sheaves over a square-zero
extension gives the classification by
\[
  \mathrm{Ext}^1_X
  \bigl(
    \mathcal F,
    \mathcal F\otimes_{\C}I_{B_1}
  \bigr).
\]
For completeness, we recall the construction.  Choose an open cover
\(\{U_i\}\) on which the deformation is locally trivial.  Then the local
trivial deformations glue by automorphisms of the form 
$1+t\,a_{ij}+\bar t\,b_{ij},$ where \(a_{ij},b_{ij}\) are local endomorphisms of \(\mathcal F\).  The gluing
condition modulo \((t,\bar t)^2\) says that
\[
  \{a_{ij}\},\{b_{ij}\}
\]
are \(1\)-cocycles with values in \(\mathcal End(\mathcal F)\), or more
generally represent classes in 
$\mathrm{Ext}^1_X(\mathcal F,\mathcal F)$ when \(\mathcal F\) is not locally free.  Changing the local trivializations
changes these cocycles by coboundaries.  Therefore the isomorphism class of the
first-order deformation is determined by the pair
\[
  \bigl([\{a_{ij}\}],[\{b_{ij}\}]\bigr)
  \in
  \mathrm{Ext}^1_X(\mathcal F,\mathcal F)
  \oplus
  \mathrm{Ext}^1_X(\mathcal F,\mathcal F).
\]
Conversely, a pair of such extension classes defines the required gluing data,
hence a first-order real-analytic deformation.
\end{proof}

\begin{remark}
If one also conjugates the fixed holomorphic space \(X\) in the
\(\bar t\)-direction, then the second summand may equivalently be written as
\[
  \mathrm{Ext}^1_{\overline X}
  (\overline{\mathcal F},\overline{\mathcal F}).
\]
In the fixed-space convention of Definition~\ref{def:ra-sheaf-fixed},
however, the two summands are both naturally Ext-groups on \(X\); the second
one is anti-holomorphic only with respect to the parameter.
\end{remark}

\subsection{Real-analytic deformations of coherent sheaves on a moving space}\label{sec_joint_realanalytic}
Let $f:\mathcal X\longrightarrow S$ be a holomorphic family of complex manifolds or smooth complex schemes. Fix a point in $S$, say \(0\in S\) and let $X_0:=f^{-1}(0).$ For a coherent \(\mathcal O_{X_0}\)-module \(\mathcal G\), a real-analytic deformation of \(\mathcal G\) along the family
\(\mathcal X\to S\) is obtained by restricting a holomorphic deformation on the
complexification.  Thus it is represented by a diagram
\[
\begin{tikzcd}
\mathcal G \arrow[r] \arrow[d]
&
\mathcal G^\circ \arrow[r] \arrow[d]
&
\mathcal G^{\mathbb C} \arrow[d]
\\
X_0 \arrow[r, hook] \arrow[d]
&
\mathcal X^\circ \arrow[r, hook] \arrow[d]
&
\mathcal X\times\overline{\mathcal X}
  \arrow[d, "{(f,\bar f)}"]
\\
\operatorname{Spec}\C \arrow[r, hook]
&
S^\circ \arrow[r, hook]
&
S\times\overline S,
\end{tikzcd}
\]
where \(\mathcal G^{\mathbb C}\) is flat over \(S\times\overline S\), and $\mathcal G^\circ
  =
  \mathcal G^{\mathbb C}\big|_{\mathcal X^\circ}.$

For first-order deformation theory, let \(I\) be a finite-dimensional
\(\C\)-vector space with \(I^2=0\).  Let $X_I$ be a square-zero deformation of \(X_0\) with ideal \(I\otimes_{\C}\mathcal O_{X_0}\).
Its Kodaira--Spencer class is
\[
  \kappa(X_I)
  \in
  \Ext^1_{X_0}
  \bigl(
    \Omega^1_{X_0},
    I\otimes_{\C}\mathcal O_{X_0}
  \bigr)
  \cong
  H^1(X_0,T_{X_0})\otimes_{\C}I.
\]

Let
\[
  \operatorname{At}(\mathcal G)
  \in
  \mathrm{Ext}^1_{X_0}
  \bigl(
    \mathcal G,
    \mathcal G\otimes\Omega^1_{X_0}
  \bigr)
\]
be the Atiyah class of \(\mathcal G\).  Contracting the Atiyah class with the
Kodaira--Spencer class gives
\[
  \operatorname{At}(\mathcal G)\cup\kappa(X_I)
  \in
  \mathrm{Ext}^2_{X_0}
  \bigl(
    \mathcal G,
    \mathcal G\otimes_{\C} I
  \bigr).
\]

\begin{thm}
\label{thm:ra-moving-sheaf-KS}
Let \(X_I\) be a square-zero deformation of \(X_0\) with ideal
\(I\otimes_{\C}\mathcal O_{X_0}\).

\begin{enumerate}
\item The obstruction to lifting \(\mathcal G\) to a coherent sheaf
\(\mathcal G_I\) on \(X_I\), flat over \(\C\oplus I\), is
\[
  \operatorname{At}(\mathcal G)\cup\kappa(X_I)
  \in
  \mathrm{Ext}^2_{X_0}
  \bigl(
    \mathcal G,
    \mathcal G\otimes_{\C}I
  \bigr).
\]

\item If this obstruction vanishes, the set of isomorphism classes of such
lifts is a torsor under
\[
  \mathrm{Ext}^1_{X_0}
  \bigl(
    \mathcal G,
    \mathcal G\otimes_{\C}I
  \bigr)
  \cong
  \mathrm{Ext}^1_{X_0}(\mathcal G,\mathcal G)\otimes_{\C}I.
\]

\item The infinitesimal automorphisms of a fixed lift are given by
\[
  \mathrm{Ext}^0_{X_0}
  \bigl(
    \mathcal G,
    \mathcal G\otimes_{\C}I
  \bigr).
\]
\end{enumerate}
\end{thm}

\begin{proof}
The statement is the standard Atiyah--Kodaira--Spencer obstruction theory for
a sheaf on a varying space.  We recall the construction.

The square-zero deformation \(X_I\) is classified by the extension
\[
  0
  \longrightarrow
  I\otimes_{\C}\mathcal O_{X_0}
  \longrightarrow
  \Omega^1_{X_I}\big|_{X_0}
  \longrightarrow
  \Omega^1_{X_0}
  \longrightarrow
  0,
\]
whose class is \(\kappa(X_I)\).  The Atiyah class of \(\mathcal G\) is the
extension class of the first jet sequence
\[
  0
  \longrightarrow
  \mathcal G\otimes\Omega^1_{X_0}
  \longrightarrow
  P^1(\mathcal G)
  \longrightarrow
  \mathcal G
  \longrightarrow
  0.
\]
Splicing these two extensions gives the Yoneda product $\operatorname{At}(\mathcal G)\cup\kappa(X_I)
  \in
  \mathrm{Ext}^2_{X_0}
  \bigl(
    \mathcal G,
    \mathcal G\otimes_{\C}I
  \bigr).$ This product measures precisely the failure of the transition functions of
\(\mathcal G\) to be lifted compatibly to the deformed structure sheaf
\(\mathcal O_{X_I}\).  Hence it is the obstruction to the existence of a lift.

If the obstruction vanishes, choices of lifted transition data differ by
\(1\)-cocycles with values in
\[
  \mathcal Hom(\mathcal G,\mathcal G\otimes_{\C}I),
\]
or, for an arbitrary coherent sheaf, by classes in $\mathrm{Ext}^1_{X_0}
  \bigl(
    \mathcal G,
    \mathcal G\otimes_{\C}I
  \bigr).$ Thus the set of lifts is a torsor under this Ext-group.  Automorphisms of a
fixed lift are similarly given by \(0\)-cocycles, namely by $\mathrm{Ext}^0_{X_0}
  \bigl(
    \mathcal G,
    \mathcal G\otimes_{\C}I
  \bigr).$
\end{proof}

Applying Theorem~\ref{thm:ra-moving-sheaf-KS} to $I=I_{B_1}=\C t\oplus \C\bar t$ gives the real-analytic first-order deformation theory. Let $X_1^{\mathbb R\mathrm{an}}
  \rightarrow
  \operatorname{Spec}B_1$ be the first-order real-analytic deformation of \(X_0\).  Its
Kodaira--Spencer class decomposes as
\[
  \kappa(X_1^{\mathbb R\mathrm{an}})
  =
  \kappa_t\,t+\kappa_{\bar t}\,\bar t,\quad 
\text{where}\quad  \kappa_t,\kappa_{\bar t}\in H^1(X_0,T_{X_0}).
\]
Equivalently, after conjugating the \(\bar t\)-direction, one writes $\kappa_t\in H^1(X_0,T_{X_0})$ and 
$\kappa_{\bar t}\in H^1(\overline{X_0},T_{\overline{X_0}}).$

\begin{corollary}
\label{cor:ra-moving-sheaf-first-order}
The obstruction to a first-order real-analytic deformation of
\(\mathcal G\) over \(X_1^{\mathbb R\mathrm{an}}\) is the pair
\[
  \bigl(
    \operatorname{At}(\mathcal G)\cup\kappa_t,\,
    \operatorname{At}(\mathcal G)\cup\kappa_{\bar t}
  \bigr)
  \in
  \mathrm{Ext}^2_{X_0}(\mathcal G,\mathcal G)
  \oplus
  \mathrm{Ext}^2_{X_0}(\mathcal G,\mathcal G).
\]
If this pair vanishes, then the set of first-order real-analytic deformations
of \(\mathcal G\) over the fixed deformation
\(X_1^{\mathbb R\mathrm{an}}\) is a torsor under $\mathrm{Ext}^1_{X_0}(\mathcal G,\mathcal G)
  \oplus
  \mathrm{Ext}^1_{X_0}(\mathcal G,\mathcal G).$

\end{corollary}

\begin{proof}
This is Theorem~\ref{thm:ra-moving-sheaf-KS} applied to the square-zero ideal $I_{B_1}=\C t\oplus\C\bar t.$ The obstruction group splits according to this decomposition:
\[
  \mathrm{Ext}^2_{X_0}
  \bigl(
    \mathcal G,
    \mathcal G\otimes_{\C}I_{B_1}
  \bigr)
  \cong
  \mathrm{Ext}^2_{X_0}(\mathcal G,\mathcal G)
  \oplus
  \mathrm{Ext}^2_{X_0}(\mathcal G,\mathcal G).
\]
The same splitting holds for the torsor of lifts, giving the Ext\(^1\)
statement.
\end{proof}

It is often useful to package the deformation theory of the pair
\((X_0,\mathcal G)\) into a single complex.  The Atiyah class induces a
morphism in the derived category $T_{X_0}
  \longrightarrow
  R\mathcal Hom_{X_0}(\mathcal G,\mathcal G)[1].$ Define the Atiyah--Kodaira--Spencer complex of the pair by
\[
  \mathcal K_{\mathcal G}
  :=
  \operatorname{Cone}
  \left(
    T_{X_0}
    \longrightarrow
    R\mathcal Hom_{X_0}(\mathcal G,\mathcal G)[1]
  \right)[-1].
\]
Then first-order deformations of the pair \((X_0,\mathcal G)\) over a
square-zero ideal \(I\) are governed by
\[
  \mathbb H^1(X_0,\mathcal K_{\mathcal G})\otimes_{\C}I.
\]
The natural long exact sequence contains
\[
\begin{aligned}
  \mathrm{Ext}^1_{X_0}(\mathcal G,\mathcal G)\otimes I
  \longrightarrow
  \mathbb H^1(X_0,\mathcal K_{\mathcal G})\otimes I
  \longrightarrow
  H^1(X_0,T_{X_0})\otimes I
\xrightarrow{\operatorname{At}(\mathcal G)\cup -}
  \mathrm{Ext}^2_{X_0}(\mathcal G,\mathcal G)\otimes I.
\end{aligned}
\]
Thus the image of a deformation of the pair in \(H^1(X_0,T_{X_0})\otimes I\)
is the Kodaira--Spencer class of the moving space, and the connecting map is
the Atiyah obstruction to lifting the sheaf.

For \(I=I_{B_1}\), this becomes
\[
  \mathbb H^1(X_0,\mathcal K_{\mathcal G})
  \oplus
  \mathbb H^1(X_0,\mathcal K_{\mathcal G}).
\]
After conjugating the \(\bar t\)-direction, one may equivalently write the
real-analytic tangent space of the pair as $\mathbb H^1(X_0,\mathcal K_{\mathcal G})
  \oplus
  \mathbb H^1(\overline{X_0},\mathcal K_{\overline{\mathcal G}}).$

\begin{definition}
\label{def:ra-KS-map-pair}
Let $f:\mathcal X\to S$ be a holomorphic family and let \(\mathcal G^\circ\) be a real-analytic
deformation of \(\mathcal G\) along \(S^\circ\).  The real-analytic
Kodaira--Spencer map of the pair \((\mathcal X^\circ,\mathcal G^\circ)\) at
\(0\in S\) is the linear map
\[
  \KS^{\mathbb R\mathrm{an}}_{\mathcal G}
  :
  \C T_0S^\circ
  \longrightarrow
  \mathbb H^1(X_0,\mathcal K_{\mathcal G})
\]
obtained by pulling the family back to first order real-analytic arcs $\operatorname{Spec}B_1\longrightarrow S^\circ.$

Under the splitting $\C T_0S^\circ
  \cong
  T_{S,0}\oplus T_{\overline S,0},$ this map decomposes into its holomorphic and anti-holomorphic components.
\end{definition}

\begin{remark}
When the ambient space \(X\) is fixed, the Kodaira--Spencer class of the space
is zero.  Therefore the Atiyah obstruction vanishes automatically, and the
complex \(\mathcal K_{\mathcal G}\) reduces to
\(R\mathcal Hom_X(\mathcal G,\mathcal G)\).  Hence the moving-space theory
recovers Proposition~\ref{prop:ra-fixed-sheaf-KS}.
\end{remark}

\section{NHC of relative moduli spaces is real-analytic} \label{sec_joint_real_ana}
Let $\pi:X\longrightarrow S$ be a smooth projective morphism of complex manifolds. Fix a topological type of complex vector bundle such that the usual non-abelian Hodge correspondence applies.  In our applications, this means that the rational Chern classes vanish.  Let
\[
  M_{\Dol}(X/S)
  \qquad\text{and}\qquad
  M_{\Betti}(X/S)
\]
be respectively the relative Dolbeault moduli space and the relative Betti moduli space.

In this section, we prove the following result, which was proved in \cite[Theorem 4.23]{CTW} when the fibers of $\pi$ are compact Riemann surfaces.

\begin{thm}[Relative real-analyticity]
\label{R-analyticity}
The relative non-abelian Hodge correspondence
\[
  \NHC:M_{\Dol}(X/S)
  \longrightarrow
  M_{\Betti}(X/S)
\]
is a real-analytic isomorphism near their smooth points. Moreover, for any holomorphic family of flat bundles over $X$, the corresponding family of Higgs bundles over $X$ by taking $\NHC$ is a real-analytic family of Higgs bundles.
\end{thm}

The key analytic input used in the proof is the real-analytic implicit function theorem in Banach spaces; see, for example, \cite{Mujica}, which is reviewed here in the form we need:

\begin{proposition}[Analytic implicit-function principle]
\label{prop_analytic_impli}
Let $Q$ be a finite-dimensional real-analytic manifold, let $E,F$ be vector bundles on a compact manifold $M$, and let
\[
 \mathcal F:Q\times \mathcal U\longrightarrow C^{k,\alpha}(F),
 \qquad \mathcal U\subset C^{k+2,\alpha}(E),
\]
be real-analytic near $(q_0,u_0)$. Here $C^{k,\alpha}(E)$ denotes the Banach space of $C^k$-sections
of $E$ whose $k$-th derivatives are H\"older continuous with
exponent $\alpha\in(0,1)$. Assume $\mathcal F(q_0,u_0)=0$ and that
\[
 D_u\mathcal F(q_0,u_0):C^{k+2,\alpha}(E)\longrightarrow C^{k,\alpha}(F)
\]
is an isomorphism. Then, after shrinking $Q$, there is a unique real-analytic map $q\mapsto u(q)$ with $u(q_0)=u_0$ and $\mathcal F(q,u(q))=0$.
\end{proposition}

\begin{proof}[Proof of Theorem~\ref{R-analyticity}]
The assertion is local. Shrink $S$ to a polydisc $\Delta$ and choose a real-analytic differentiable trivialization $X_\Delta\simeq X_0\times\Delta$. Thus the fiberwise complex structures $J_s$ and K\"ahler forms $\omega_s$ of $X/S$ are real-analytic families on the fixed compact manifold $X_0$.

Fix a stable Higgs bundle $q_0=(s_0,\bar\partial_0,\theta_0)$, a harmonic metric $h_0$, and a finite-dimensional real-analytic slice $\Sigma$ through $q_0$. Write $q=(s,\bar\partial_q,\theta_q)\in\Sigma$. On the fixed smooth bundle let
\[
 \mathcal B^{m,\alpha}_0
 =C^{m,\alpha}\bigl(X_0,\operatorname{Herm}_0(E,h_0)\bigr),
\]
the trace-free $h_0$-self-adjoint endomorphisms, and write a nearby normalized metric as $h_a=h_0e^a$, $a\in\mathcal B^{k+2,\alpha}_0$. Define
\[
 \mathcal H(q,a)={\operatorname{pr}_{0}}\!\left[\sqrt{-1}\Lambda_{\omega_s}
 \left(F_{h_a,\bar\partial_q}+[\theta_q,\theta_q^{*h_a}]\right)\right]
 \in\mathcal B^{k,\alpha}_0,
\]
where
$
\operatorname{pr}_0(\cdot)$ denotes the trace-free part. The operations $e^a$, inversion, curvature, adjoint, contraction, and bracket are real-analytic on H\"older spaces; hence $\mathcal H$ is a real-analytic nonlinear elliptic map.

Put $\mathcal D''_0=\bar\partial_0+\operatorname{ad}(\theta_0)$. Up to a positive convention-dependent factor, the metric linearization at $(q_0,0)$ is
\[
 L_0=(\mathcal D''_0)^*\mathcal D''_0.
\]
For $b\in\mathcal B^{k+2,\alpha}_0$ one has
\[
 \langle L_0b,b\rangle_{L^2}
 =\|\bar\partial_0b\|_{L^2}^2+\|[\theta_0,b]\|_{L^2}^2.
\]
Thus $\ker L_0$ consists of trace-free Higgs endomorphisms. Stability implies that every Higgs endomorphism is scalar, so $\ker L_0=0$. Since $L_0$ is self-adjoint elliptic, we have an isomorphism
\[
 L_0:\mathcal B^{k+2,\alpha}_0\xrightarrow{\sim}\mathcal B^{k,\alpha}_0.
\]
Proposition~\ref{prop_analytic_impli} gives a unique real-analytic solution $a=a(q)$. Hence the normalized harmonic metric $h_q=h_0e^{a(q)}$ depends real-analytically on $q$. The associated flat connection
\[
 D_q=\bar\partial_q+\partial_{h_q,\bar\partial_q}
       +\theta_q+\theta_q^{*h_q}
\]
is therefore real-analytic. This proves the Dolbeault-to-Betti direction.

Conversely, let $D_q$ be a real-analytic family of irreducible flat connections. For a Hermitian metric $h$, write $D_q=d_{q,h}+\Psi_{q,h}$ with $d_{q,h}$ unitary and $\Psi_{q,h}$ self-adjoint. The normalized harmonic metric is characterized by Corlette's equation
\[
 d_{q,h}^{*}\Psi_{q,h}=0.
\]
Its metric linearization is the self-adjoint elliptic Jacobi operator. The Bochner identity identifies its kernel with $D_{q_0}$-parallel self-adjoint endomorphisms; irreducibility makes them scalar, and the trace-free normalization removes them. Proposition~\ref{prop_analytic_impli} again gives real-analytic dependence of $h_q$. Taking the type decomposition with respect to $(J_s,h_q)$ yields
\[
 D_q=(\bar\partial_{E,q}+\theta_q)
     +(\partial_{E,q}^{h_q}+\theta_q^{*h_q}),
\]
so $(\bar\partial_{E,q},\theta_q)$ depends real-analytically on $q$. This proves the inverse direction.  
\end{proof}

\end{document}